\pdfoutput=1
\documentclass[a4paper]{amsart}
\usepackage{amssymb, amsthm, enumerate, mathtools, mathrsfs, stmaryrd, diagrams}
\usepackage[headings]{fullpage}
\usepackage[all, cmtip]{xy}
\usepackage[pagebackref]{hyperref}
\usepackage[hyphenbreaks]{breakurl}

\newcommand{\fa}{\mathfrak{a}}
\newcommand{\fd}{\mathfrak{d}}
\newcommand{\fl}{\mathfrak{l}}
\newcommand{\fm}{\mathfrak{m}}
\newcommand{\fn}{\mathfrak{n}}
\newcommand{\fp}{\mathfrak{p}}
\newcommand{\fq}{\mathfrak{q}}
\newcommand{\fM}{\mathfrak{M}}
\newcommand{\fN}{\mathfrak{N}}

\renewcommand{\AA}{\mathbf{A}}
\newcommand{\CC}{\mathbf{C}}
\newcommand{\GG}{\mathbf{G}}
\newcommand{\QQ}{\mathbf{Q}}
\newcommand{\RR}{\mathbf{R}}
\newcommand{\ZZ}{\mathbf{Z}}

\newcommand{\Zp}{\ZZ_p}
\newcommand{\Qp}{\QQ_p}
\newcommand{\QQbar}{\overline{\QQ}}

\newcommand{\cA}{\mathcal{A}}
\newcommand{\cE}{\mathcal{E}}
\newcommand{\cF}{\mathcal{F}}
\newcommand{\cG}{\mathcal{G}}
\newcommand{\cH}{\mathcal{H}}
\newcommand{\cO}{\mathcal{O}}
\newcommand{\cS}{\mathcal{S}}
\newcommand{\cT}{\mathcal{T}}

\newcommand{\sH}{\mathscr{H}}

\newcommand{\into}{\hookrightarrow}
\renewcommand{\ge}{\geqslant}
\renewcommand{\le}{\leqslant}

\newcommand{\ah}{\mathrm{ah}}
\newcommand{\ab}{\mathrm{ab}}
\newcommand{\Asai}{\mathrm{Asai}}
\newcommand{\mot}{\mathrm{mot}}
\newcommand{\et}{\text{\rm \'et}}
\newcommand{\dR}{\mathrm{dR}}
\newcommand{\Betti}{\mathrm{Betti}}
\newcommand{\Dcris}{\mathbf{D}_\mathrm{cris}}

\newcommand{\AF}{\mathrm{AF}}
\newcommand{\cAI}{{}_c\mathcal{AI}}
\newcommand{\cAF}{{}_c\mathcal{AF}}
\newcommand{\cEI}{{}_c\mathcal{EI}}

\renewcommand{\d}{\,\mathrm{d}}
\DeclareMathOperator{\CHM}{CHM}
\DeclareMathOperator{\Cl}{Cl}
\DeclareMathOperator{\Eis}{Eis}
\DeclareMathOperator{\Fil}{Fil}
\DeclareMathOperator{\Frob}{Frob}
\DeclareMathOperator{\Gal}{Gal}
\DeclareMathOperator{\GL}{GL}
\DeclareMathOperator{\Gr}{Gr}
\DeclareMathOperator{\Hom}{Hom}
\DeclareMathOperator{\id}{id}
\DeclareMathOperator{\imp}{imp}
\DeclareMathOperator{\Ind}{Ind}
\DeclareMathOperator{\loc}{loc}
\DeclareMathOperator{\mom}{mom}
\DeclareMathOperator{\Nm}{Nm}
\DeclareMathOperator{\ord}{ord}
\DeclareMathOperator{\pr}{pr}
\DeclareMathOperator{\Res}{Res}
\DeclareMathOperator{\SL}{SL}
\DeclareMathOperator{\Spec}{Spec}
\DeclareMathOperator{\Sym}{Sym}
\DeclareMathOperator{\Tr}{Tr}
\DeclareMathOperator{\tInd}{\otimes\text{-}Ind}
\DeclareMathOperator{\TSym}{TSym}

\newcommand{\tbt}[4]{\begin{pmatrix}#1 & #2 \\ #3 & #4\end{pmatrix}}
\newcommand{\stbt}[4]{\left(\begin{smallmatrix}#1 & #2 \\ #3 & #4\end{smallmatrix}\right)}
\newcommand{\rstd}{\rho_{\cF, v}^{\mathrm{std}}}
\newcommand{\urt}{(\underline{r}, \underline{t})}
\newcommand{\dblquot}[3]
{
 {
  \left.
  \raisebox{-.1em}{$#1$}
  \,\middle\backslash\,
  \raisebox{.1em}{$#2$}
  \,\middle/\,
  \raisebox{-.1em}{$#3$}
  \right.
 }
}

\newtheorem{theorem}{Theorem}[subsection]
\newtheorem{lemma}[theorem]{Lemma}
\newtheorem{proposition}[theorem]{Proposition}
\newtheorem{corollary}[theorem]{Corollary}

\newtheorem{conjecture}[theorem]{Conjecture}
\newtheorem{definition}[theorem]{Definition}
\newtheorem{notation}[theorem]{Notation}
\newtheorem{ltheorem}{Theorem} 

\newtheorem{subtheorem}{Theorem}[theorem] 

\theoremstyle{remark}
\newtheorem{remark}[theorem]{Remark}
\newtheorem*{remark*}{Remark}

\begin{document}

\author[A.~Lei]{Antonio Lei}
\address[Lei]{D\'epartement de math\'ematiques et de statistique\\
Pavillon Alexandre-Vachon\\
Universit\'e Laval\\
Qu\'ebec, QC, Canada G1V 0A6. ORC ID: \href{https://orcid.org/0000-0001-9453-3112}{0000-0001-9453-3112}.}
\email{antonio.lei@mat.ulaval.ca}

\author[D.~Loeffler]{David Loeffler}
\address[Loeffler]{Mathematics Institute\\
Zeeman Building, University of Warwick\\
Coventry CV4 7AL, UK. ORC ID: \href{https://orcid.org/0000-0001-9069-1877}{0000-0001-9069-1877}.}
\email{d.a.loeffler@warwick.ac.uk}

\author[S.L.~Zerbes]{Sarah Livia Zerbes}
\address[Zerbes]{Department of Mathematics \\
University College London\\
Gower Street, London WC1E 6BT, UK. ORC ID: \href{https://orcid.org/0000-0001-8650-9622}{0000-0001-8650-9622}.}
\email{s.zerbes@ucl.ac.uk}

\date{\today}

\thanks{The authors are grateful to acknowledge support from the following grants: NSERC Discovery Grants Program 05710 (Lei); Royal Society University Research Fellowship (Loeffler); ERC Consolidator Grant ``Euler Systems and the Birch--Swinnerton-Dyer conjecture'' (Zerbes).}

\begin{abstract}
 We construct an Euler system -- a compatible family of global cohomology classes -- for the Galois representations appearing in the geometry of Hilbert modular surfaces. If a conjecture of Bloch and Kato on injectivity of regulator maps holds, this Euler system is non-trivial, and we deduce bounds towards the Iwasawa main conjecture for these Galois representations.
\end{abstract}


\title{Euler systems for Hilbert modular surfaces}

\setcounter{tocdepth}{1}
\setcounter{secnumdepth}{4}
\hypersetup{bookmarksdepth=4}

 \maketitle
 \tableofcontents

 \section{Introduction}

  One of the central problems of number theory is the study of cohomology groups of global Galois representations, and the relation between these cohomology groups and the values of $L$-functions. A crucial tool in this study is the theory of \emph{Euler systems}: collections of Galois cohomology classes for a given Galois representation over abelian extensions of the base field, satisfying compatibility conditions as the field changes. These have powerful applications to studying Selmer groups, and thus they are inevitably difficult to construct.

  In the present paper, we construct Euler systems for a new class of Galois representations: the \emph{Asai}, or \emph{twisted tensor product}, Galois representations attached to Hilbert modular eigenforms over real quadratic fields. These are the Galois representations which appear in the middle-degree cohomology of Hilbert modular surfaces. More precisely, we shall prove the following:

  \begin{ltheorem}
   \label{lthm:A}
   Let $\cF$ be a Hilbert modular eigenform over the real quadratic field $F$, of level $\fN$ and weights $(k + 2, k' + 2)$, with $k, k' \ge 0$ and $k = k' \bmod 2$; and let $L$ be a finite extension of $\QQ$ containing the Hecke eigenvalues of $\cF$. Let $v$ be a place of $L$ above the rational prime $p \ne 2$, and let $M_{L_v}(\cF)$ be the Asai Galois representation of $\cF$ at $v$ (see Definition~\ref{defn:asairep} below). Let $a$ be a generator of $\cO_F / \ZZ$, and $j$ an integer with $0 \le j \le \min(k, k')$.

   Then there exists a collection of cohomology classes, the \emph{Asai--Flach classes},
   \[ \AF_{\et, m, a}^{[\cF, j]} \in H^1\left(\ZZ\left[\mu_m, \tfrac{1}{m p \Delta \Nm_{F / \QQ}(\fN)}\right], M_{L_v}(\cF)^*(-j)\right), \]
   for integers $m \ge 1$, which satisfy Euler-system-type norm relations as $m$ varies.
  \end{ltheorem}

  Note that we do not need to impose any assumptions on the character of $\cF$, because our constructions do not require any self-duality properties of the Galois representations involved. See Definition \ref{def:AFf} below for the definition of these classes, and Corollary \ref{cor:AFf} for the norm relation. This construction can be regarded as an analogue of previous work of the present authors and Guido Kings \cite{leiloefflerzerbes14a, kingsloefflerzerbes15a, kingsloefflerzerbes15b} in the setting of Rankin--Selberg convolutions of two elliptic modular forms.

  \begin{remark*}
   In \cite{liu16}, Liu uses Hirzebruch--Zagier cycles to construct a collection of global cohomology classes for the self-dual twist of $M_{L_v}(\cF)^*\otimes M_{L_v}(g)^*$, where $\cF$ is a Hilbert modular form of parallel weight $2$ and $g$ is an elliptic modular form of weight $2$. These cohomology classes stand in the same relation to the Euler system constructed in this paper as the cohomology classes arising from `diagonal cycles'  (constructed in \cite{darmonrotger12}) do to the Euler system of Beilinson--Flach elements \cite{leiloefflerzerbes14a, kingsloefflerzerbes15a, kingsloefflerzerbes15b}.
  \end{remark*}

  Using Kings' theory of $\Lambda$-adic sheaves, we can construct a ``$p$-adic interpolation'' of the above classes for varying $j$ and $m$, assuming that $\cF$ is \emph{ordinary at $p$} in the sense of Definition \ref{def:ordinary}:

  \begin{ltheorem}
   \label{lthm:iwasawa}
   Suppose $\cF$ is ordinary at $p$. Let $\Gamma = \Gal(\QQ(\mu_{p^\infty}) / \QQ)$, $\Lambda_\Gamma$ its Iwasawa algebra, and $\mathbf{j}: \Gamma \to \Lambda_\Gamma^{\times}$ the canonical character. Let $m \ge 1$ be coprime to $p$, and $c > 1$ be coprime to $6pm\Nm_{F/\QQ}(\fN)$.

   Then there exists an Iwasawa cohomology class
   \[ \cAF^{\cF}_{m, a} \in H^1\left(\ZZ\left[\mu_m, \tfrac{1}{m p \Delta \Nm_{F / \QQ}(\fN)}\right], M_{L_v}(\cF)^* \otimes \Lambda_{\Gamma}(-\mathbf{j})\right) \]
   which interpolates the \'etale Asai--Flach classes $\AF_{\et, mp^r, a}^{[\cF, j]}$, for all $0 \le j \le \min(k, k')$ and $r \ge 0$.

   Moreover, the restriction of $M_{L_v}(\cF)^*$ to $\Gal(\QQbar_p / \Qp)$ has a canonical 1-dimen\-sional unramified quotient, and the projection of the localisation $\loc_p\left(\cAF^{\cF}_{m, a}\right)$ to the cohomology of this quotient is zero.
  \end{ltheorem}

  Our main application of this Euler system is (in some sense) a version of the Iwasawa main conjecture for the motive $M_{L_v}(\cF)$ over the cyclotomic tower. For this theorem, we assume that $p$ is split in $F$, and we fix a prime $\fp \mid p$. Using Perrin-Riou's big logarithm map (see \S\ref{sect:motivicL}), we construct a ``motivic $p$-adic $L$-function'' ${}_c L^{\imp}_{\fp, \Asai}(\cF) \in L_v \otimes_{\Zp} \Lambda_\Gamma$ interpolating the Bloch--Kato logarithms of the Asai--Flach classes; and we define a dual Selmer group $X(\QQ(\mu_{p^\infty}), \cF)$, which is a $\Lambda_{\Gamma}$-module of finite type. Then we prove the following theorem:

  \begin{ltheorem}
   \label{lthm:boundSel}
   Assume that $\cF$ and $v$ satisfy the list of hypotheses given in \S\ref{sect:boundSel} below. Then the characteristic ideal\footnote{We adopt the convention that the ``characteristic ideal'' of a non-torsion $\Lambda_\Gamma$-module is the zero ideal.} of the dual Selmer group $X(\QQ(\mu_{p^\infty}), \cF)$ divides the $p$-adic $L$-function ${}_c L^{\imp}_{\fp, \Asai}(\cF)$ in $L_v \otimes_{\Qp} \Lambda_{\Gamma}$.
  \end{ltheorem}

  Sadly, this theorem is rather less powerful than it seems, since we have at present no analogue of the explicit reciprocity laws available in the Rankin--Selberg setting (cf.~\cite{BDR15b}, \cite[Theorem B]{kingsloefflerzerbes15b}); thus we cannot rule out the possibility that ${}_c L^{\imp}_{\fp, \Asai}(\cF)$ is identically zero, in which case the above theorem is vacuous.

  We can show that our Euler system is non-zero in many cases if one assumes a standard conjecture in arithmetic geometry:

  \begin{ltheorem}
   \label{lthm:nonvanish}
   Suppose that $\cF$ is new of level $\fN$, $k, k' \ge 1$ and $|k - k'| \ge 3$, and $F$ has trivial narrow class group. Let $r = k + k'$ and assume that Conjecture 5.3(i) of \cite{blochkato90} holds for some smooth compactification $\mathcal{A}^{r} \into \widetilde{\mathcal{A}^{r}}$ of the $r$-fold fibre product of the universal abelian variety over the Hilbert modular surface of level $\fN$. Then the class $\AF_{\et, 1, a}^{[\cF, j]}$ is non-zero, for any $0 \le j \le \min(k, k')$.
  \end{ltheorem}

  See Proposition \ref{prop:nonvanish} below for details. A second piece of evidence for the non-triviality of our construction is a formula expressing the localisations of our \'etale classes at $p$ in terms of overconvergent $p$-adic modular forms. This is joint work of the second and third authors with Chris Skinner, and is explained in a separate paper \cite{loefflerskinnerzerbes16}.

  \subsection*{Acknowledgements}

   The authors would like to thank Giuseppe Ancona, Henri Darmon, Mladen Dimitrov, Eyal Goren, Payman Kassaei,  Guido Kings, Francesco Lemma and Chris Skinner for valuable conversations relating to the preparation of this paper. We are also grateful to the two anonymous referees for their numerous valuable comments on the original draft of this paper.

 \section{Setup and notations}
  \label{S:defn}

  \subsection{Fields and groups}
   \label{sect:fldsgps}

   Let $F$ be a totally real field, with ring of integers $\cO_F$, different $\fd$ and discriminant $\Delta = N_{F/\QQ}(\fd)$. Later in the paper we shall specialise to the case where $F$ is the real quadratic field $\QQ(\sqrt{d})$, for $d > 1$ a square-free integer, but our initial discussion (up to the end of \S\ref{sect:heckecorr}) is valid for general $F$.

   We write $\Cl^+(F)$ for the narrow class group of $F$ (the group of non-zero fractional ideals of $\cO_F$, modulo principal fractional ideals with a totally positive generator). We refer to ideals whose class in $\Cl^+(F)$ is trivial as \emph{narrowly principal}.

   \begin{definition}[cf.~{\cite[\S 1]{dimitrov05}}]
    We define the algebraic groups
    \begin{align*}
     D &= \Res^F_\QQ\GG_m, &
     G &=\Res^F_\QQ\GL_2, &
     G^* &= G \times_{D} \GG_m.
    \end{align*}
   \end{definition}

   There is a natural embedding $\iota: \GL_2 \into G^*$, which will be of great importance in the present paper. The embedding $\jmath: G^* \into G$ will also be needed, particularly in \S\ref{sect:eigenforms}.

   If $H$ is any of the three groups $\{ \GL_2, G^*, G\}$, then we let $H(\RR)^+$ be the identity component of $H(\RR)$, which is the subgroup of elements whose determinant is (totally) positive. We write $H(\QQ)^+ = H(\QQ) \cap H(\RR)^+$.

   We define $\cH_F$ to be the elements of $F \otimes \CC$ of totally positive imaginary part, with its natural action of $G(\RR)^+$.

  \subsection{Arithmetic quotients}

   Let $\AA$ be the ad\`ele ring of $\QQ$, and $\AA_f$ be the subring of finite adeles.

   \begin{definition}
    \label{def:suffsmall}
    Let $H$ be one of the three groups $G$, $G^*$, or $\GL_2$. We say an open compact subgroup $U \subset H(\AA_f)$ is \emph{sufficiently small} if, for any $h \in H(\AA_f)$, the quotient group
    \[ \frac{H(\QQ)^+ \cap h U h^{-1}}{U \cap \{ \stbt u 0 0 u: u \in \cO_F^{\times +}\}} \]
    acts without fixed points on $\cH_F$ (or $\cH_\QQ$ if $H = \GL_2$).
   \end{definition}

   (The denominator is, of course, trivial if $H = \GL_2$ or $H = G^*$.)

   \begin{definition}
    For $U \subseteq G(\AA_f)$ an open compact subgroup (respectively $U^* \subset G^*(\AA_f)$, $U_\QQ \subset \GL_2(\AA_f)$) we write
    \begin{align*}
     Y(U) &\coloneqq \dblquot{G(\QQ)^+}{\left[G(\AA_f) \times \cH_F\right]}{U},\\
     Y^*(U^*) &\coloneqq \dblquot{G^*(\QQ)^+}{\left[G^*(\AA_f) \times \cH_F\right]}{U^*},\\
     Y_{\QQ}(U_{\QQ}) &\coloneqq \dblquot{\GL_2(\QQ)^+}{\left[\GL_2(\AA_f) \times \cH\right]}{U_{\QQ}},
    \end{align*}
   \end{definition}

   If $U$ is sufficiently small (in the sense of Definition \ref{def:suffsmall}), then the quotient $Y(U)$ is naturally the set of complex points of a smooth variety defined over $\QQ$. The same holds for the varieties $Y^*(U^*)$, $Y_\QQ(U_\QQ)$.

   For $g \in G(\AA_f)$ we have a map
   \[ g: Y(U) \to Y(g U g^{-1}),\quad g \cdot [(h, \tau)] = [(h g^{-1}, \tau)], \]
   which gives a left action of $G(\AA_f)$ on the inverse system of varieties $Y \coloneqq \{ Y(U) : U \subset G(\AA_f) \text{ open compact}\}$ for varying $U$, compatible with the usual left action of $G(\QQ)^+ \subset G(\AA_f)$ on $\cH_F$. The same applies verbatim for $\GL_2 / \QQ$, and for $G^*$.

   \begin{remark}
    We shall mostly work with $G^*$ rather than $G$, because the Shimura varieties $Y^*(U^*)$ for $G^*$ are of PEL type: they are moduli spaces for abelian varieties with $\cO_F$-action, as we shall recall in \S \ref{sect:abvar} below.

    The chief disadvantage of $G^*$ is that automorphic representations of $G^*$ do not satisfy the multiplicity one property, whereas those for $G$ do. In order to work around this, we shall enlarge the group of transformations acting on the varieties $Y^*(U^*)$ using a construction due to Shimura, which we shall now recall.
   \end{remark}

   \begin{definition}[{cf.~\cite[p.~643]{shimura78}}]
    \label{def:calG}
    We let $\cG$ denote the subgroup $G(\QQ)^+ G^*(\AA_f) \subseteq G(\AA_f)$.
   \end{definition}

   Then there are bijections $Y^*(U^*) = \dblquot{G(\QQ)^+ }{[\cG \times \cH_F]}{U^*}$ for each $U^*$, and we therefore obtain maps of $\QQ$-varieties $Y^*(U^*) \to Y^*(g U^* g^{-1})$ for any $g \in \cG$, which assemble into a left action of $\cG$ on the pro-variety $Y^*$.

   \begin{proposition}[{\cite{tianxiao16}, Proposition 2.4}]
    \label{prop:Gquotient}
    If $U^* = U \cap G^*$, then there is a natural map $\jmath: Y^*(U^*) \to Y(U)$, and its fibres are the orbits for an action of the finite group
    \[
     \frac{\cG \cap U}{U^* \cdot (Z(\cG) \cap U)} \cong \frac{ \cO_F^{\times +} \cap \left(\hat\ZZ^\times \cdot \det(U) \right)}{ \left\{\epsilon^2 : \epsilon \in \cO_F^\times, \stbt \epsilon 0 0 \epsilon \in U\right\}}
    \]
    on $Y^*(U^*)$. The subgroup stabilising each component of $Y^*(U^*)$ is $\cO_F^{\times+}\cap \det(U)$.\qed
   \end{proposition}

  \subsection{Congruence subgroups}

   Let us now define the specific level groups $U$ that we shall use.

   \begin{definition}
    Let $\fM, \fN, \fa$ be non-zero ideals of $\cO_F$. We define open compact subgroups of $G(\AA_f)$ by
    \begin{align*}
     U(\fM, \fN) &\coloneqq \left\{\gamma \in \GL_2(\widehat{\cO}_F): \gamma = 1 \bmod \tbt{\fM}{\fM}{\fN}{\fN}\right\},\\
     U(\fM(\fa), \fN) &\coloneqq
     \left\{ \gamma: \gamma = 1 \bmod \tbt{\fM}{\fa \fM}{\fN}{\fN}\right\},  \\
     U(\fM, \fN(\fa)) &\coloneqq
     \left\{ \gamma: \gamma = 1 \bmod \tbt{\fM}{\fM}{\fa \fN}{\fN}\right\}
    \end{align*}
    We write $U^*(\fM, \fN)$ for $U(\fM, \fN) \cap G^*$ (and similarly $U^*(\fM, \fN(\fa))$ etc.) We shall often abbreviate $U(1, \fN)$ as $U_1(\fN)$.

    Similarly, for $M, N, A$ non-zero integers, we write $U_{\QQ}(M, N)$, $U_{\QQ}(M, N(A))$, $U_\QQ(M(A), N)$ for the analogous subgroups of $\GL_2(\AA_f)$.
   \end{definition}

   \begin{notation}
    We adopt the general notation scheme that if $U(-)$ is some subgroup of $G$, then $Y(-)$ denotes $Y(U(-))$, and similarly if $U^*(-)$ is a subgroup of $G^*$. Thus $Y_1(\fN)$, $Y_1^*(\fN)$, $Y_{1, \QQ}(N)$ are shorthand notations for $Y(U_1(\fN))$, $Y^*(U_1^*(\fN))$, $Y_{\QQ}(U_{1, \QQ}(N))$ respectively.
   \end{notation}

   \begin{remark}
    \label{rmk:suffsmall}
    Note that if $\fN$ does not divide $2$, $3$, or $\Delta$, then $U_1(\fN)$ and $U_1^*(\fN)$ are sufficiently small \cite[Lemma 2.1]{dimitrov09}.
   \end{remark}

  \subsection{Hecke algebras}
   \label{sect:heckealg}

   Let $\fM$ and $\fN$ be non-zero ideals of $\cO_F$, with $\fM \mid \fN$. We shall now define various elements of the abstract Hecke algebra $\ZZ\left[U^* \backslash \cG / U^*\right]$, where $U^* = U^*(\fM, \fN)$ and $\cG$ is as in \ref{def:calG}.

   \begin{remark}
    The reason for working with $\cG$, rather than the smaller group $G^*(\AA_f)$, is that $G^*(\AA_f)$ only gives rise to a Hecke operator $T(n)$ when $n$ is in $\ZZ$. Working with $\cG$ allows us to consider Hecke operators $T(x)$ for general $x \in \cO_{F, +}$ (while still working with a Shimura variety of PEL-type).
   \end{remark}

   \subsubsection{Diamond operators}

   For $x \in (\cO_F / \fN)^\times$, we define $\langle x \rangle$ to be the double coset of $\tbt {x^{-1}} 0 0 x \in \SL_2(\widehat\cO_F)$, for any lift of $x$ to $\widehat\cO_F^\times$.

   \subsubsection{Frobenius maps}
    \label{sect:frob}

   For $x \in \left(\ZZ / \ZZ \cap \fM\right)^\times$, we define $\sigma_x$ as the double coset of $\stbt x{}{} 1 \in G^*(\hat\ZZ)$.

   \subsubsection{Scalar multiplications \texorpdfstring{$R(x)$}{R(x)}}

   For $x \in F^\times$, we write $R(x)$ for the double coset of the scalar matrix $\tbt {x^{-1}} 0 0 {x^{-1}}$.

   \subsubsection{The operator \texorpdfstring{$S(x)$}{S(x)}} For any $x \in F^\times$ which is a unit at the primes above $\fN$, we write $S(x)$ for $\langle x \rangle R(x)$.

   \subsubsection{The operator \texorpdfstring{$T(x)$}{T(x)}} For $x \in \cO_F$ which is totally positive and square-free\footnote{That is, $x$ is not divisible by the square of any nontrivial ideal (principal or otherwise)}, we define $T(x)$ as the double coset of $\tbt {x^{-1}} 0 0 1$.

   More generally, we may define $T(x)$ for any totally-positive $x \in \cO_F$, not necessarily square-free, using the formal sum of all double cosets contained in the set
   \[ \left\{ \begin{array}{c}
   \tbt a b c d \in M_{2 \times 2}(\widehat\cO_F) : ad-bc \in  x^{-1} \cdot (1 + M \hat\ZZ),
   \\
   \tbt a b c d = \tbt {x^{-1}} 0 0 1 \bmod \tbt \fM \fM \fN \fN \end{array}\right\}, \]
   where $M$ is the positive integer generating the ideal $\fM \cap \ZZ$. This set is clearly contained in $\cG$, and it is left and right invariant under $U^*(\fM, \fN)$.

   The operators defined above, for all valid choices of $x$, define a commutative subalgebra of $\ZZ[U^* \backslash \cG / U^*]$. In this algebra we have the following familiar identities:
   \begin{subequations}
    \begin{align}
     T(xy) = T(x) T(y) & \text{ if $x$ and $y$ are coprime}, \\
     T(x)^2 =  S(x) & \text{ if $x \in \cO_F^{\times}$,}\\
     T(x)^2 = T(x^2) + \Nm_{F/\QQ}(x)\cdot S(x) & \text{ if $x\cO_F$ is prime}.
    \end{align}
   \end{subequations}

   If $x$ divides $\fN$, we denote the double coset $T(x)$ defined above by the more familiar alternative notation $U(x)$.

   \subsubsection{Hecke operators for \texorpdfstring{$G$}{G}}
    \label{sect:heckeG}

   We shall also need to work with some Hecke operators for the group $G$; these will not make an appearance until \S \ref{sect:eigenforms}. We denote these by calligraphic letters to reduce the risk of confusion with their analogues for $G^*$. It will suffice to consider levels of the form $U_1(\fN)$.
   \begin{itemize}
    \item For $\fm \triangleleft \cO_F$, we denote by $\cT(\fm)$ the double coset of $\stbt {x^{-1}} 0 0 1$ where $x$ is any element of $\widehat\cO_F$ generating the ideal $\fm \widehat \cO_F$. As before, when $\fm \mid \fN$ we use the alternative notation $\mathcal{U}(\fm)$ for this element.

    \item For $\fm$ an ideal coprime to $\fN$, we let $\cS(\fm)$ be the double coset of $\stbt {x^{-1}} 0 0 {x^{-1}}$, where $x$ is any generator of $\fm \widehat \cO_F$ congruent to 1 modulo $\fN$.
   \end{itemize}

   Note that if $\fm = (\lambda)$ is a narrowly-principal ideal, and we write the element $T(\lambda) \in \ZZ[U^*_1(\fN) \backslash \cG / U_1^*(\fN)]$ as a sum of single cosets $\sum_i U_1^*(\fN) g_i$, then we also have $\cT(\fm) = \sum_i U_1(\fN) g_i$. Similarly, if $\lambda$ is coprime to $\fN$ then we can find a (single) element of $\cG$ representing both of the double cosets $\cS(\fm)$ and $S(\lambda)$.

  \subsection{Abelian varieties}
   \label{sect:abvar}

   We now introduce certain abelian varieties over the Shi\-mura varieties $Y^*(U^*)$ defined above. (These are the universal objects for appropriate PEL-type moduli problems, but we will not use this directly in the present paper.)

   \begin{definition} \
    \begin{enumerate}[(i)]

     \item Let $P$ be the subgroup of $\Res_{F / \QQ} \GL_3$ consisting of matrices of the form
     \[ \begin{pmatrix} 1 & r &s \\ 0 & a & b \\ 0 & c & d \end{pmatrix}, \]
     and let $P^*$ be the subgroup with $\stbt a b c d \in G^*$. We let $N \cong \Res_{F /\QQ} \AA^2$ be the unipotent radical of $P$ and $P^*$.

     \item Let $\CC_F = F \otimes_{\QQ} \CC$. Then we define a left action of $P(\RR)^+$ on the space $\mathcal{J}_F = \cH_F \times \CC_F$ via
     \[ \begin{pmatrix} 1 & r &s \\ 0 & a & b \\ 0 & c & d \end{pmatrix} \cdot (\tau, z) = \left(\frac{a\tau + b}{c\tau + d}, \frac{z + r \tau + s}{c\tau + d}\right).\]
     (This corresponds to identifying $(\tau, z) \in \mathcal{J}_F$ with $[z : \tau : 1] \in \mathbf{P}^2(\CC_F)$ and acting on this by left-multiplication.)

     \item For $V^* \subseteq P^*(\AA_f)$ open compact, we write
     \[ A(V^*) = \dblquot{P^*(\QQ)^+}{(P^*(\AA_f) \times \mathcal{J}_F)}{V^*},\]
     which is an example of a \emph{mixed Shimura variety}.
    \end{enumerate}
   \end{definition}

   \begin{proposition}[{cf.~\cite[Example VI.1.10]{milne-mixed}}]
    If the image of $V^*$ in $G^*(\AA_f)$ is sufficiently small, the quotient $A(V^*)$ is the complex points of a quasiprojective algebraic variety over $\QQ$; and the natural map $\mathcal{J} \to \cH$ makes $A(V^*)$ into an abelian variety of dimension $[F : \QQ]$ over $Y^*(U^*)$, where $U^*$ is the image of $V^*$ in $G^*$, with endomorphisms by an order in $\cO_F$.
   \end{proposition}

   \begin{remark}
    If we define $\mathcal{P}$ to be the subgroup of $P(\AA_f)$ with $\stbt a b c d \in \cG$, where $\cG$ is as in Definition \ref{def:calG}, then $P^*(\AA_f) \subset \mathcal{P}$ and we have
    \[ A(V^*) = P(\QQ)^+ \backslash (\mathcal{P} \times \mathcal{J}_F) / V^*, \]
    so that the left action of $P^*(\AA_f)$ on the system of varieties $A(V^*)$ for varying $V^*$ naturally extends to an action of $\mathcal{P}$.
   \end{remark}

   We are particularly interested in subgroups of the form
   \[ V^* = \widehat{\cO}_F^2 \rtimes U^*,\]
   where $U^* \subset G^*(\hat\ZZ)$. The corresponding abelian varieties have endomorphisms by $\cO_F$. We shall abuse notation slightly by writing $A(U^*)$ for $A(\widehat{\cO}_F^2 \rtimes U^*)$; this notation will only be used when $U^* \subset G^*(\hat\ZZ)$, so that this object is well-defined.

   \begin{definition}
    Let $g \in \cG$ be such that $g^{-1}$ has entries in $\widehat\cO_F$, and let $U^* \subset G^*(\AA_f)$ be a sufficiently small open compact subgroup such that $U^*$ and $g U^* g^{-1}$ are both contained in $G(\hat\ZZ)$.

    We define an $\cO_F$-linear isogeny
    \[ \Phi_g: A(U^*) \to g^* A(g U^* g^{-1}) \]
    of abelian varieties over $Y^*(U^*)$ as the composite map
    \[ A(\widehat{\cO}_F^2 \rtimes U^*) \rTo^{\tilde g} A\left( \left(\widehat{\cO}_F^2 \cdot g^{-1}\right) \rtimes g U^* g^{-1}\right) \rTo A( \widehat{\cO}_F^2 \rtimes g U^* g^{-1}), \]
    where the first map is the left action of the element $\tilde g = \stbt 1 0 0 g \in \mathcal{P}$, and the second map is the natural quotient map given by the inclusion $\widehat{\cO}_F^2 \cdot g^{-1} \subset \widehat{\cO}_F^2$.
   \end{definition}

   If $g = \stbt {x}00{x}^{-1}$, for $x \in \cO_F$, then $\Phi_g$ is simply the endomorphism action of $x$ on $A(U^*)$. We can use this to extend the definition of $\Phi_g$ to all $g \in \cG$ as an element of $\Hom\left(A(U^*), g^*A(gU^*g^{-1})\right) \otimes \QQ$.

   \begin{remark}
    One easily verifies that the isogenies $\Phi_{g}$ satisfy the obvious cocycle condition $\Phi_{g_1 g_2} = g_2^*(\Phi_{g_1}) \circ \Phi_{g_2}$ wherever both sides are defined. Thus the collection of abelian varieties $A(U^*)$, for varying $U^* \subseteq G^*(\hat\ZZ)$, defines an abelian variety $A$ over the pro-variety $Y^*$ which is ``$\cG$-equivariant up to isogeny'' (i.e.~it is a $\cG$-equivariant object in the isogeny category of abelian varieties over $Y^*$).
   \end{remark}

  \subsection{Hecke correspondences}
   \label{sect:heckecorr}

   The Hecke operators defined in \S \ref{sect:heckealg} can naturally be regarded as algebraic correspondences on $Y^*(\fM, \fN)$, and hence as endomorphisms of the cohomology of this variety (for any reasonable cohomology theory). Using the isogenies $\Phi_g$ of the previous section, we can extend this to define actions (both contravariant and covariant) of Hecke operators on the cohomology of the abelian varieties $A(\fM, \fN)$ over $Y^*(\fM, \fN)$.

   We have, in fact, \emph{two} possible actions of Hecke operators on cohomology, via contravariant (pullback) and covariant (pushforward) functoriality. We shall distinguish between the two by using a prime symbol when the covariant action is intended, so that $T(x)$ and $T'(x)$ denote the contravariant and covariant actions of the same abstract double coset. Since pushforward by an automorphism coincides with pullback by its inverse, we have $\langle x \rangle' = \langle x^{-1} \rangle$ and $\sigma_x' = \sigma_x^{-1}$.

   (In the norm relations for our Euler system, the covariant Hecke operators $T'(x)$ and $U'(x)$ will play the main role; philosophically, this reflects the fact that Euler systems are in some sense homological rather than cohomological objects.\footnote{This is analogous to the distinction between Picard and Albanese functoriality in the construction of the Euler system of Heegner points. (We are grateful to one of the anonymous referees for this observation.)})

  \begin{remark}
   If $x \in \cO_F$ is totally positive and square-free, the action of $T(x)$ on the cohomology of $A(\fM, \fN)$ is given by the composition of the following four maps:
   \begin{itemize}
    \item pullback along the natural degeneracy map $A(\fM, \fN(x)) \to A(\fM, \fN)$;
    \item pullback via the isomorphism $Y(\fM(x), \fN) \to Y(\fM, \fN(x))$ given by the matrix $g = \stbt {x^{-1}} 0 0 1$, which corresponds to $\tau \mapsto \tau/x$ on $\cH_F$;
    \item pullback along the isogeny
    \[ \Phi_g: A(\fM(x), \fN) \to g^* A(\fM, \fN(x)) \]
    of abelian varieties over $Y(\fM(x), \fN)$, which corresponds to the map
    \[ (\tau, z \bmod \cO_F\cdot \tau + \cO_F) \to (\tau, z \bmod \cO_F \tfrac{\tau}{x} + \cO_F) \]
    on $\mathcal{J}$;
    \item pushforward via the natural degeneracy map $A(\fM(x), \fN) \to A(\fM, \fN)$.
   \end{itemize}
   The action of $T'(x)$ is exactly the dual of this (i.e.~interchanging pullbacks and pushforwards). This is the natural analogue for $G^*$ of Kato's description of the Hecke operators $T(\ell)$ for modular curves in \cite[\S 2.8 \& \S 4.8]{kato04}. As remarked in \S 4.9.4 of \emph{op.cit.}, the correspondences $T(x)$ and $T'(x)$ thus defined preserve the geometric connected components of $Y^*(\fM, \fN)$. They do \emph{not} commute with the action of $\SL_2(\cO_F / \fN)$ in the case $\fM = \fN$.
  \end{remark}

  \subsection{The Asai Euler factor}
   \label{sect:asaiEF}

   We now impose the assumption that $[F : \QQ] = 2$. For a (rational) prime $\ell \nmid \Delta \cdot \Nm_{F/\QQ}(\fN)$, we define the following polynomial with coefficients in the Hecke algebra of level $U^*(\fM, \fN)$ (where we continue to assume, as before, that $\fM \mid \fN$):

   \begin{definition}
    \label{def:eulerfactor}
    The \emph{Asai Euler factor} is the polynomial $P_\ell(X)$ defined as follows:
    \begin{itemize}
     \item If $\ell$ is inert, we set
     \[ P_\ell(X) = \left(1 - T(\ell) X + \ell^2  S(\ell) X^2\right)\left(1 - \ell^2 S(\ell)  X^2\right).\]
     \item If $\ell$ is split, we set
     \begin{multline*}
      P_\ell(X) = 1 -  T(\ell) X
      + \Big(T(\ell)^2 - T(\ell^2) -  \ell^2 S(\ell)\Big) X^2 \\
      - \ell^2  S(\ell) T(\ell) X^3 + \ell^4S(\ell)^2  X^4.
     \end{multline*}
    \end{itemize}
   \end{definition}

   We shall see in Proposition \ref{prop:proj-Pl} that the action of $P_\ell(X)$ on a Hilbert modular eigenform will give the local factor at $\ell$ of the Asai $L$-function (justifying the term ``Asai Euler factor'').

   \begin{remark}
    If $\ell$ is split and the primes above $\ell$ are narrowly principal, so we can write $\ell = \lambda \bar\lambda$ where $\lambda \in \cO_F^+$, then the $X^2$ coefficient in $P_\ell(X)$ can also be written
    \[
     \ell \langle \lambda \rangle R(\lambda) T(\bar\lambda)^2 + \ell \langle \bar\lambda \rangle R(\bar\lambda) T(\lambda)^2 - 2 \ell^2 \langle\ell \rangle R(\ell).
    \]
    This latter formula will be used in the proofs of the norm relations, where we will always be assuming that the primes above $\ell$ are narrowly principal. However, the version using $T(\ell^2)$ is more general, and in particular it shows that the Hecke operators appearing in $P_\ell$ always lie in $G^*(\AA_f)$, rather than the slightly larger group $\cG$.
   \end{remark}

 \section{Asai--Flach classes}

  We now define a collection of motivic cohomology classes for Hilbert modular surfaces. We make no claim to originality here: this construction is fundamentally the same as that of \cite{kings98}, although we express it in a slightly different language and setup in order to reinforce the similarities to the construction of \cite{kingsloefflerzerbes15a}.

  \subsection{Formalism of relative motives}

   We begin by recalling the formalism of ``relative motives'' attached to families of varieties over a base; for more detail see \cite{deningermurre91}.

   \subsubsection{Relative Chow motives}

    Let $k$ be a field of characteristic 0, and $S$ a smooth, connected, quasi-projective $k$-variety. Then there exists a $\QQ$-linear, pseudo-abelian tensor category $\CHM(S)_{\QQ}$ of \emph{relative Chow motives} over $S$, equipped with a contravariant functor
    \[ M: \operatorname{SmPr}(S) \to \CHM(S)_{\QQ}.\]
    Here $\operatorname{SmPr}(S)$ denotes the category of smooth projective $S$-schemes. Similarly, for any coefficient field $L$ of characteristic 0 there is a category $\CHM(S)_L$, which coincides with the pseudo-abelian envelope of $L \otimes_{\QQ} \CHM(S)_\QQ$.

    \begin{remark}
     Concretely, an object of $\CHM(S)_L$ is given by a triple $(X, \alpha, n)$, where $X$ is a smooth projective $S$-variety, $\alpha \in \operatorname{CH}^{\dim(X / S)}(X \times_S X)$ is an idempotent, and $n \in \ZZ$. The Tate object is $(S, \id, 1)$.
    \end{remark}

   \subsubsection{Realisations}

    It is well known that Weil cohomology theories, such as de Rham, Betti, or \'etale cohomology, give functors on the category $\CHM(\Spec k)_L$ (where $L$ is the appropriate coefficient field for the Weil cohomology). These naturally take values in \emph{graded} $L$-vector spaces, equipped with various extra structures depending on the choice of cohomology theory.

    These have analogues in the relative setting, taking values in categories of sheaves on $S$ with extra structure:
    \begin{itemize}
     \item if $L$ is a $p$-adic field, the \emph{$p$-adic \'etale realisation} from $\CHM(S)_{L}$ to lisse \'etale $L$-sheaves on $S$;
     \item if $k = \RR$ or $\CC$ and $L$ is a subfield of $\CC$, then the \emph{Hodge realisation} from $\CHM(S)_L$ to the category of variations of pure $L$-Hodge structures on $S$.
    \end{itemize}

    (There are also realisations in de Rham cohomology, Betti cohomology, etc, but we shall not use these here.) If $\cF \in \operatorname{Obj}(\CHM(S)_L)$, we write $\cF_{\et}$, $\cF_{\cH}$, etc for its realisations in the appropriate cohomology theories. These are naturally graded objects: we have $\cF_\et = \bigoplus_j \Gr^j \cF_\et$, where $\Gr^j M(X)_{\et} = \cH^{j}_\et(X / S)$ is the $j$-th relative \'etale cohomology sheaf of $X / S$, and similarly for the other realisation functors.

    \begin{remark}
     Note that the grading need not be concentrated in degrees $\ge 0$: indeed, the realisations of the Tate motive over $S$ are concentrated in degree $-2$.
    \end{remark}

    \begin{theorem}[Deninger--Murre, \cite{deningermurre91}]\label{thm:DM}
     If $A / S$ is an abelian variety, there is a canonical decomposition in the category $\CHM(S)_{\QQ}$,
     \[ M(A) = \bigoplus_{i = 0}^{2 \operatorname{dim} A} M^i(A), \]
     such that, if $\spadesuit$ denotes any of the above realisations, $\Gr^j M^i(A)_{\spadesuit}$ is zero if $i \ne j$.
    \end{theorem}

   \subsubsection{Motivic cohomology and regulators}

    Since a smooth projective $S$-variety is in particular a smooth quasiprojective $k$-variety, one can define \emph{motivic cohomology groups} with $L$-coefficients, $H^i_{\mot}(X, L(j))$, for $X \in \operatorname{SmPr}(S)$ as in \cite{beilinson84}.

    We adopt the following convention: if $\cF \in \operatorname{Obj} \CHM(S)_L$ is given by a triple $(X, \alpha, n)$ as above, and the realisations of $\cF$ are non-zero in only one degree $r$, then we write
    \[
     H^i_{\mot}(S, \cF(j)) \coloneqq
     \alpha^* H^{(i + r + 2n)}_{\mot}(X, L(j + n)).
    \]
    With this convention, we obtain regulator maps
    \[ r_{\cT}: H^i_{\mot}(S, \cF(j)) \to H^i_{\cT}(S, \cF_\cT(j)) \]
    for each of the above cohomology theories.

    \begin{remark}
     The shift in indexing occurs because the $\cT$-realisation of $\cF$ should be considered as a complex of sheaves concentrated in degree $r$, but we are forgetting the grading, i.e.~treating it as if it were concentrated in degree 0.
    \end{remark}

    \subsubsection{Functoriality in S}

    Let $S$, $T$ be two smooth connected quasi-projective $k$-varieties, so that the categories $\CHM(S)_L$ and $\CHM(T)_L$ are defined. If $\iota: S \to T$ is a morphism, then there is a pullback functor
    \[ \iota^*: \CHM(T)_L \to \CHM(S)_L \]
    which maps the motive of a $T$-variety $X$ to the motive of $\iota^* (X) = S \times_{\iota, T} X$. This is clearly compatible with the pullback functors for the various realisations.

    If we assume $\iota$ to be a closed immersion of codimension $d$, there is a Gysin map
    \[ \iota_*: H^i_{\mot}(S, \iota^* \cF(n)) \to H^{i + 2d}_{\mot}(T, \cF(n + d))\]
    for any $\cF \in \operatorname{Obj}\left(\CHM(S)_L\right)$, compatible with the pushforward maps for the realisations described above. If $\cF = M(X)$ for a variety $X / T$, this is just the pushforward map
    \[
    H^i_{\mot}(\iota^*(X), L(n)) \to H^{i + 2d}_{\mot}(X, L(n + d))
    \]
    corresponding to the inclusion $\iota^*(X) \into X$ (cf.~\cite[Theorem 15.15]{mazzavoevodskyweibel06}).

  \subsection{Relative motives over Shimura varieties}

   We will be interested in the cohomology of certain relative motives (and their realisations) arising from the universal abelian varieties over modular curves and Hilbert modular surfaces. (This is an instance of a much more general construction, applying to arbitrary PEL Shimura varieties, due to Ancona \cite[\S 8]{ancona15}.)

   \subsubsection{Modular curves}

    Let $U$ be an open compact subgroup of $\GL_2(\AA_f)$, and suppose that $U$ is sufficiently small. Then we have a modular curve $Y_{\QQ}(U)$, which is a smooth affine $\QQ$-variety; and if $V = \hat\ZZ^2 \rtimes U \subset P_{\QQ}(\AA_f)$, then we obtain an elliptic curve $\cE = A_{\QQ}(V)$ over $Y_{\QQ}(U)$.

    \begin{definition}
     We write $\sH_L(\cE)$ for the motive $M^1(\cE)(1)$, where $M^1(\cE)$ is as given by Theorem~\ref{thm:DM}; and $\TSym^k \sH_L(\cE)$ for its $k$-th symmetric tensor power.
    \end{definition}

    Here the symmetric tensor power is defined as the invariants for the action of the symmetric group on $\sH_L(\cE)^{\otimes k}$, whereas the more familiar symmetric power $\Sym^k \sH_L(\cE)$ is the coinvariants. These two are in fact isomorphic, since $L$ is a field of characteristic 0 and thus $k!$ is invertible in $L$. However, the definition of the Clebsch--Gordan map in \S \ref{sect:CGmap} below is simpler to describe using the $\TSym$ modules; and we shall also later need to consider analogous coefficient sheaves in \'etale cohomology over $\Zp$, where the distinction between $\Sym^k$ and $\TSym^k$ is significant if $k \ge p$. See \cite[\S 2.2]{kingsloefflerzerbes15b} for further discussion.

    \begin{remark}
     This construction is consistent with the ad-hoc definition of the groups $H^i_{\mot}(Y_\QQ(U), \TSym^k \sH_\QQ(\cE)(j))$ given in \cite[Definition 3.2.2]{kingsloefflerzerbes15a}. (The case $L \ne \QQ$ was not considered \emph{loc.cit.}.) For general $L$ we have
     \[
      H^i_{\mot}(Y_\QQ(U), \TSym^k \sH_L(\cE)(j)) = L \otimes_{\QQ} H^i_{\mot}(Y_\QQ(U), \TSym^k \sH_\QQ(\cE)(j)).
     \]
    \end{remark}
   \subsubsection{Shimura varieties for \texorpdfstring{$G^*$}{G*}}

    We now suppose $U^*$ is an open compact subgroup in $G^*(\AA_f)$, where $G^*$ is the group defined in \S \ref{sect:fldsgps} above, and we set $V^* = \widehat{\cO}_F^2 \rtimes U^*$. Again, we suppose $U^*$ to be sufficiently small, so that the mixed Shimura variety $\cA = A(V^*)$ is an abelian surface over $Y^*(U^*)$.

    We define an object of $\CHM(Y^*(U^*))_L$ by
    \[ \sH_L(\cA) = M^3(\cA)(2) = M^1(\cA)^\vee. \]
    Since $\cO_F$ acts on $\cA$ by endomorphisms, for each $x \in \cO_F$ we have an endomorphism $[x]_*$ of $\sH_L(\cA)$, and if $x \in \ZZ$ then $[x]_*$ is simply multiplication by $x$.

    By enlarging $L$ if necessary, we now suppose that there exist non-zero embeddings $F \into L$, and we let $\theta_1, \theta_2$ be the two such embeddings. Then the relative motive $\sH_L(\cA)$ decomposes as $\sH_L(\cA)^{(1)} \oplus \sH_L(\cA)^{(2)}$ where $\sH_L(\cA)^{(i)}$ denotes the direct summand where we have
    \[ [x]_* = \theta_i(x) \]
    for $x \in \cO_F$.

    \begin{remark}
     We may take $\sH_L(\cA)^{(1)}$ and $\sH_L(\cA)^{(2)}$ to be the images of the orthogonal idempotents
     \[ \frac{ \theta_1(\sqrt{D}) + [\sqrt{D}]_*}{2\theta_1(\sqrt{D})} \quad\text{and}\quad
      \frac{\theta_1(\sqrt{D}) - [\sqrt{D}]_*}{2\theta_1(\sqrt{D})}.\]
    \end{remark}

    \begin{definition}
     Let $k, k' \ge 0$. The relative Chow motive $\TSym^{[k, k']} \sH_L(\cA)$ over $Y^*(U^*)$ is defined by
     \[ \TSym^k\left( \sH_L(\cA)^{(1)}\right) \otimes \TSym^{k'} \left(\sH_L(\cA)^{(2)}\right).\]
    \end{definition}

    \begin{remark}
     \label{rmk:powerofA}
     Thus $\TSym^{[k, k']} \sH_L(\cA)$ can be realised as a direct summand of the motive
     \[
      M^{3r}(\cA^r)(2r) = \left(M^{r}(\cA^{r})\right)^\vee,
     \]
     where $r = k + k'$. Hence, for any $i, n \in \ZZ$, the motivic cohomology group $H^i_{\mot}(Y^*(U^*), \TSym^{[k, k']} \sH_L(\cA)(n))$ is a direct summand of $H^{i + 3r}_{\mot}\left(\cA^r, L(n + 2r)\right)$. We can also realise it as a direct summand of $H^{i + r}_{\mot}\left(\cA^r, L(n + r)\right)$, since the canonical polarisation gives an isomorphism $M^3(\cA) \cong M^1(\cA)(-1)$.
     One can check that if $k = k'$, then $\TSym^{[k, k']}\sH_L(\cA)$ can be defined without assuming that $F \subseteq L$.
    \end{remark}

    Since we have defined the Hecke correspondences $T(\ell)$, $T'(\ell)$, $R(\ell)$, etc as correspondences on $\cA$, the cohomology groups $H^i_{\mot}\left(Y^*(U^*), \TSym^{[k, k']} \cH_L(\cA)(n)\right)$ for $i, n \in \ZZ$ acquire actions of these operators. Note that $R'(x)$ acts, by construction, as multiplication by $\theta_1(x)^k \theta_2(x)^{k'} \in L^\times$.

   \subsubsection{Shimura varieties for \texorpdfstring{$G$}{G}}
    \label{sect:shvarG}

    We now consider the case of a sufficiently small open compact subgroup $U \subset G(\AA_f)$. This case is not covered by \cite{ancona15}, since the Shimura datum for $G$ is not of PEL type; we are grateful to Giuseppe Ancona for explaining to us how to extend his construction to this case. We take $k, k' \ge 0$ such that $k = k' \bmod 2$, and we choose integers $t, t'$ such that $k + 2t = k' + 2t'$. We write $\mu$ for the quadruple $(k, k', t, t')$.

    \begin{definition}
     Let $U^* = U \cap G^*$. We let $\widetilde \sH_L^{[\mu]}$ be the relative Chow motive over $Y^*(U^*)$ defined by
     \begin{multline*}
      \Big[\TSym^k \left(\sH_L(\cA)^{(1)}\right) \otimes \det\left(\sH_L(\cA)^{(1)}\right)^t\Big] \\ \otimes \Big[\TSym^{k'} \left(\sH_L(\cA)^{(2)}\right) \otimes \det\left(\sH_L(\cA)^{(2)}\right)^{t'}\Big].
     \end{multline*}
    \end{definition}

    \begin{remark}
     In fact both $\det(\sH_L(\cA)^{(1)})$ and $\det(\sH_L(\cA)^{(2)})$ are isomorphic in $\CHM(Y^*(U^*))_L$ to the Tate motive $L(1)$, since there is a canonical $\cO_F$-linear isogeny from $\cA$ to its dual (the dual can be identified with $\cA \otimes_{\cO_F} \fd^{-1}$). However, this isomorphism is not $\cG$-equivariant, so it does not respect the Hecke action on the cohomology of $\sH^{[\mu]}$.
    \end{remark}

    Let $Y(U)_1$ denote the image of $Y^*(U^*)$ in $Y(U)$ (which is the union of a subset of the components of $Y(U)$), and let $H$ be the finite abelian group $(U \cap \cG) / U^*\cdot(U \cap Z(\cG))$ of Proposition \ref{prop:Gquotient}, so that
    \[ Y(U)_1 = Y^*(U^*) / H.\]
    Let $q: Y^*(U^*) \to Y(U)_1$ be the quotient map.

    \begin{proposition}
     There is a natural action of $H$ on $q_* \widetilde \sH_L^{[\mu]}$, acting via automorphisms in the category $\CHM(Y(U)_1)_L$.
    \end{proposition}

    \begin{proof}
     We have defined an action of $\cG \cap M_{2 \times 2}(\widehat\cO_F)$ on $\cA$ by isogenies, compatible with its action on $Y^*(U^*)$, and the elements of $\cG \cap U$ act as automorphisms; so it suffices to check that the subgroup $U \cap Z(\cG)$ acts trivially on the direct factor $\widetilde \sH_L^{[\mu]}$ of $\cA \times \dots \times \cA$.

     By construction, pushforward by a scalar matrix $x \in F^\times$ acts on $\widetilde \sH_L^{[\mu]}$ as multiplication by $\theta_1(x)^{k + 2t} \theta_2(x)^{k' + 2t'} = \Nm_{F/\QQ}(x)^w$, where $w$ is the common value $k + 2t = k' + 2t'$. Since $U \cap Z(\cG) = \hat\ZZ^\times \cdot \cO_F^{\times +}$, with the $\hat\ZZ^\times$ factor acting trivially, and units in $\cO_F^{\times +}$ all have norm 1, the action of $U \cap Z(\cG)$ is trivial as required.
    \end{proof}

    \begin{definition}
     We let $\sH_L^{[\mu]}$ be the relative Chow motive over $Y(U)_1$ defined as the direct summand of $q_* \widetilde \sH_L^{[\mu]}$ cut out by the projector $\tfrac{1}{|H|} \sum_{h \in H} h$.
    \end{definition}

    (This makes sense, since for any base $S$, the category of relative Chow motives over $S$ is by definition a \emph{Karoubian category} -- a semiabelian category in which any idempotent endomorphism has a kernel and image.)

    We extend this to $Y(U)$ as follows: if $g_1, \dots, g_n$ are a finite set of elements of $G(\AA_f)$ whose determinants are coset representatives for $\AA_{F, f}^\times / (F^{\times +} \AA_{\QQ, f}^\times \det(U))$, then $Y(U)$ is isomorphic to the disjoint union of the varieties $Y(g_i U g_i^{-1})_1$, and we may apply the above construction to each of these varieties individually. The resulting relative Chow motive over $Y(U)$ is independent (up to a canonical isomorphism) of the choice of the $g_i$, and its motivic cohomology has natural covariant and contravariant actions of the Hecke algebra $\ZZ[U \backslash G(\AA_f) / U]$.

  \subsection{The Clebsch--Gordan map}
   \label{sect:CGmap}

   Now let $U^*$ be an open compact in $G^*(\AA_f)$, and let $U_\QQ$ be its intersection with $\GL_2(\AA_f)$, so there is a closed embedding
   \[ \iota: Y_\QQ(U_\QQ) \into Y^*(U^*), \]
   and the abelian variety $\iota^*(\cA)$ is canonically isomorphic to $\cO_F \otimes_{\ZZ} \cE$ (compatibly with the $\cO_F$-action). Hence both $\iota^* \sH_L(\cA)^{(1)}$ and $\iota^* \sH_L(\cA)^{(2)}$ can be identified with $\sH_L(\cE)$.

   As explained in \cite[\S5.1]{kingsloefflerzerbes15b}, we  have the following maps:
   \[ \TSym^{k + k'} \sH_L(\cE) \into \TSym^{k} \sH_L(\cE) \otimes \TSym^{k'} \sH_L(\cE) = \iota^*\left(\TSym^{[k, k']}  \sH_L(\cA)\right);\]
   and
   \[ L(1) = \sideset{}{^2_L}\bigwedge \sH_L(\cE) \into \sH_L(\cE) \otimes \sH_L(\cE) = \iota^*\left( \TSym^{[1, 1]}  \sH_L(\cA)\right).
   \]
   Combining these two cases using multiplication in the symmetric tensor algebra $\TSym^\bullet$, we obtain the following:

   \begin{proposition}
    For any integers $k, k', j$ satisfying the inequality
    \[ 0 \le j \le \min(k, k'), \]
    we have a canonical morphism of relative Chow motives over $Y_\QQ(U_\QQ)$,
    \[ CG^{[k, k', j]}_{\mot} : \TSym^{k + k' - 2j} \sH_L(\cE) \to \iota^*\left( \TSym^{[k, k']} \sH_L(\cA) \right)(-j).\]
   \end{proposition}

   \begin{remark}
    Note that $CG^{[k, k', j]}_{\mot}$ does \emph{not} commute with maps induced by isogenies of the universal abelian varieties $\cE$ and $\cA$, since the identification $\bigwedge^2_L \sH_L(\cE) \cong L(1)$ is not preserved by isogenies. In particular, for $n \in \ZZ$, the Hecke operator $R'(n)$ (acting as pushforward via the $n$-multiplication map on $\cA$ and $\cE$) acts as $n^{k + k' - 2j}$ on the source of the map $ CG^{[k, k', j]}_{\mot}$ and as $n^{k + k'}$ on the target.
   \end{remark}

  \subsection{Construction of Asai--Flach classes over \texorpdfstring{$\QQ$}{Q}}

   \begin{definition}
    For $k \ge 0$ and $N \ge 5$, let
    \[ \Eis^k_{\mot, N} \in H^1_{\mot}\left(Y_{1, \QQ}(N), \TSym^k \sH_{\QQ}(\cE)(1)\right) \]
    be the class defined in \cite[Theorem 4.1.1]{kingsloefflerzerbes15a}. Via base extension, we regard this as an element of $H^1_{\mot}(Y_{1, \QQ}(N), \TSym^k \sH_{L}(\cE)(1))$ for any coefficient field $L$.
   \end{definition}

   (In \emph{op.cit.} this class is denoted by $\Eis^k_{\mot, b, N}$, as it depends on a choice of $b \in \ZZ / N \ZZ - \{0\}$; but we shall take $b = 1$ and drop it from the notation.)

   Now let $\fN \triangleleft \cO_F$ be a non-zero ideal such that $U_1^*(\fN)$ is sufficiently small (cf.~Remark \ref{rmk:suffsmall} above). We can now define the key objects of this paper:

   \begin{definition}
    For $k, k', j$ integers satisfying the inequality $0 \le j \le \min(k, k')$, we define the \emph{motivic Asai--Flach class}
    \[
     \AF^{[k, k', j]}_{\mot, \fN} \in
     H^3_{\mot}\left(Y_1^*(\fN), \TSym^{[k, k']} \sH_L(\cA)(2-j)\right)
    \]
    as the image of $\Eis^{k + k' - 2j}_{\mot, N}$ under $(\iota_* \circ CG^{[k, k', j]}_{\mot})$, where $N = \fN \cap \ZZ$.
   \end{definition}

   Similarly, we write $\AF^{[k, k', j]}_{\et, \fN}$ for the image of $\AF^{[k, k', j]}_{\mot, \fN}$ in \'etale cohomology.

  \subsection{Classes over cyclotomic fields}
   \label{sect:cycloclasses}

   Note that the Asai--Flach classes of the previous section are defined on the Hilbert modular varieties $Y_1^*(\fN)$, which are geometrically connected varieties over $\QQ$. We now define more general Asai--Flach classes, which are cohomology classes on the base-extensions of these varieties to cyclotomic fields.

   Let $M \in \ZZ_{\ge 1}$, and $\fN$ an ideal of $\cO_F$ as above. Via pullback along the natural map $Y^*(M, M\fN) \to Y_1^*(M\fN)$, we regard $\AF^{[k, k', j]}_{\mot, M\fN}$ as a class in the cohomology of $Y^*(M, M \fN)$.

   The variety $Y^*(M, M \fN)$ has an action of $\cO_F / M \cO_F$, since the corresponding subgroup of $G^*(\AA_f)$ is normalised by matrices of the form $u_a = \tbt 1 a 0 1$ with $a \in \cO_F$. Moreover, there is a map
   \[ s_M: Y^*(M, M\fN) \to Y^*_1(\fN) \times \mu_M^\circ,\]
   given by the right action of $\stbt {M^{-1}} 0 0 1$; here we identify $Y^*_1(\fN) \times \mu_M^\circ$ with the Shimura variety of level
   \[ \{ u \in U_1^*(\fN): \det(u) = 1 \bmod M \}, \]
   as in \cite[\S 6.1]{kingsloefflerzerbes15b}.

    Both the automorphisms $u_a$, and the map $s_M$, extend naturally to maps on the universal abelian variety $\cA$, and thus on our motivic coefficient sheaves.

   \begin{definition}
    We define
    \begin{multline*}
     \AF_{\mot, M, \fN, a}^{[k, k', j]} = (s_M \circ u_a)_*\left( \AF^{[k, k', j]}_{\mot, M\fN}\right) \\ \in H^3_{\mot}\left(Y_1^*(\fN) \times \mu_M^\circ, \TSym^{[k, k']} \sH_L(\cA)(2-j)\right),
    \end{multline*}
    and $\AF_{\et, M, \fN, a}^{[k, k', j]}$ its \'etale analogue.
   \end{definition}

   \begin{remark}
    In fact we can replace $Y^*(M, M\fN)$ with $Y^*(M, \fN')$ for any ideal $\fN'$ divisible by $M$ and $\fN$ and having the same prime factors as $M\fN$; this will follow from the norm-compatibility relations below.
   \end{remark}

   \begin{theorem}
    \label{thm:AFeltproperties}
    The above elements enjoy the following properties.
    \begin{enumerate}
     \item The class $\AF_{\mot, M, \fN, a}^{[k, k', j]}$ depends only on the image of $a$ in $\cO_F / (M\cO_F + \ZZ)$.
     \item For $b \in (\ZZ / M \ZZ)^\times$, we have
     \[ \sigma_b \cdot \AF_{\mot, M, \fN, a}^{[k, k', j]} = \AF_{\mot, M, \fN, b^{-1}a}^{[k, k', j]},\]
     where $\sigma_b$ denotes the image of $b$ in $\Gal(\QQ(\mu_M) / \QQ)$.
     \item (Level compatibility) If $\fl$ is a prime and $\pr_{1, \fl}$ denotes the natural projection $Y^*_1(\fl \fN) \to Y^*_1(\fN)$, then we have
     \[
      (\pr_{1, \fl})_*\left( \AF_{\mot, M, \fl \fN, a}^{[k, k', j]} \right) =
      \begin{cases}
       \AF_{\mot, M, \fN, a}^{[k, k', j]} \text{ if $\fl \mid M \fN$ or $\overline\fl \mid M\fN$,}\\
       \left( 1 - \ell^{k + k' - 2j} \langle \ell^{-1} \rangle \sigma_\ell^{-2} \right)
       \AF_{\mot, M, \fN, a}^{[k, k', j]}
        \text{ otherwise,}
      \end{cases}
     \]
     where $\ell$ is the rational prime below $\fl$.
     \item (Euler system norm relation) If $\ell$ is prime satisfying one of the conditions below, and $a$ is a generator of $\cO_F / (\ell M \cO_F + \ZZ)$ , we have
     \[
      \operatorname{norm}_{M}^{\ell M} \left(\AF_{\mot, \ell M, \fN, a}^{[k, k', j]}\right) = A \cdot \AF_{\mot, M, \fN, a}^{[k, k', j]}
     \]
     where $A$ is the Hecke operator given as follows:
     \begin{itemize}
      \item if $\ell \mid \fN$ and $\ell \mid M$, then $A = U'(\ell)$.
      \item if $\ell \mid \fN$ and $\ell \nmid M$, then $A = (U'(\ell) - \ell^j \sigma_\ell)$.
      \item if $(\ell, M\fN) = 1$ and either $\ell$ is inert, or $\ell$ is split and the primes above $\ell$ are narrowly principal, then
      \[ A = \ell^j \sigma_\ell\left( (\ell - 1)\left( 1 - \ell^{k + k' - 2j} \langle \ell^{-1} \rangle \sigma_\ell^{-2} \right) - \ell P_\ell'(\ell^{-1-j} \sigma_\ell^{-1}) \right), \]
      where $P_\ell'(X)$ is the operator-valued Asai Euler factor of Definition \ref{def:eulerfactor}, acting via the covariant Hecke action on cohomology.
     \end{itemize}
    \end{enumerate}
   \end{theorem}

   The Hecke operators appearing in the theorem are those defined in \S \ref{sect:heckealg} above.

   \begin{remark}
    \label{rmk:denoms} \
    \begin{enumerate}[(i)]
     \item Despite its conceptual importance -- asserting the existence of an ``Euler system'' in motivic cohomology -- we shall not actually use this theorem directly. The reason for this is that the definition given above of the motivic Asai--Flach classes, and even of the groups that they live, only makes sense with coefficients in $\QQ$; while the applications of Euler systems to bounding Selmer groups require \emph{uniformly bounded} denominators with respect to some appropriate lattice in the $p$-adic Asai Galois representation, and it is manifestly unclear from the above construction how this condition can be checked.

     However, in the next section we shall obtain (as a by-product of our $p$-adic interpolation calculations) a second, independent construction of the \'etale versions of these classes, from which the integrality will be clear.
     \item A conceptual interpretation of the ``wrong'' Euler factor appearing above can be given along the same lines as in the Rankin--Selberg case, see \cite[\S 8]{leiloefflerzerbes14a}.
    \end{enumerate}
   \end{remark}

   \begin{proof}[Sketch of proof]
    The proof of this theorem is virtually identical to the proof of the $\Lambda$-adic version, which we shall prove in \S \ref{sect:normrels} below, so we leave it to the reader to make the necessary modifications.
   \end{proof}

 \section{Eigenforms and Galois representations}
  \label{sect:eigenforms}
  Having constructed our Asai--Flach classes in the cohomology of the varieties $Y^*_1(\fN)$, we are now interested in projecting to quotients of these cohomology spaces corresponding to eigenforms. Since the multiplicity one property does not hold for automorphic representations of $G^*$, but it does hold for those of $G$ (cf.~\cite[\S 3.2]{brylinskilabesse84}), it is more convenient to work with the varieties $Y_1(\fN)$.

  \subsection{Hilbert modular forms for \texorpdfstring{$G$}{G}} \

   \begin{notation} \
    \begin{enumerate}[(i)]
     \item Let $\theta_1, \theta_2$ be the embeddings $F \into \RR$. For a pair $\underline{r} = (r_1, r_2) \in \ZZ^2$, and $z \in F$, we write $z^{\underline{r}}$ for $\theta_1(z)^{r_1} \theta_2(z)^{r_2}$; and we extend this to $z \in F \otimes \CC$ in the obvious fashion.

     \item For $\underline{r} = (r_1, r_2) \in \ZZ^2$, $f$ a function $\cH_F \to \CC$, and $\gamma = \stbt a b c d \in G(\QQ)^+ = \GL_2(F)^+$, we define $f \mid_{\underline{r}} \gamma$ by
     \[
      (f \mid_{\underline{r}} \gamma)(\tau) \coloneqq
      \Nm_{F/\QQ}(\det \gamma) (c\tau + d)^{-\underline{r}} f\left( \frac{a\tau + b}{c\tau + d}\right).
     \]

     \item If $\underline{r} = (r_1, r_2)$ and $\underline{t} = (t_1, t_2)$ are pairs of integers, $f$ is a function $\cH_F \to \CC$, and $\gamma = \stbt a b c d \in G(\QQ)^+$, then we write
     \[
      f \mid_{\urt} \gamma
      \coloneqq \det(\gamma)^{-\underline{t}} f \mid_{\underline{r}} \gamma.
     \]
    \end{enumerate}
   \end{notation}

   We shall only consider the $\urt$ action when the integers $\urt$ satisfy $r_1 + 2t_1 = r_2 + 2t_2$; in particular, this implies that $r_1 = r_2 \bmod 2$. We let $w$ be the common value $r_1 + 2t_1 = r_2 + 2t_2$. Then scalar matrices $\stbt u 0 0 u$ with $u \in F^\times$ act via $\Nm_{F/\QQ}(u)^{1-w}$, and in particular $\cO_F^{\times +}$ acts trivially.

   \begin{definition}\label{def:HMF} \
    Let $U \subset G(\AA_f)$ be open compact. A \emph{Hilbert modular form} of weight $\urt$ and level $U$ is a function $\cF: G(\AA_f) \times \cH_F \to \CC$ such that
    \begin{enumerate}[(i)]
     \item For every $g \in G(\AA_f)$, the function $\cF(g, -)$ is holomorphic on $\cH_F$.
     \item We have $\cF(gu, \tau) = \cF(g, \tau)$ for all $u \in U$, $g \in G(\AA_f)$ and $\tau \in \cH_F$.
     \item We have $\cF(\gamma g, -) = \cF(g, -) \mid_{\urt} \gamma^{-1}$ for all $\gamma \in G(\QQ)^+$ and $g \in G(\AA_f)$.
    \end{enumerate}
    We let $M_{\urt}(U, \CC)$ denote the space of such functions, and $S_{\urt}(U, \CC)$ the subspace of cusp forms.
   \end{definition}

   The space $S_{\urt}(U, \CC)$ is the subspace of $U$-invariants in a smooth right representation of $G(\AA_f)$, so it is a right module over the Hecke algebra $\CC[ U \backslash G(\AA_f) / U]$. In particular, if $U = U_1(\fN)$ for some $\fN$, the Hecke operators $\cT(\fm)$ and $\cS(\fm)$ defined in \S\ref{sect:heckeG} above act on $S_{\urt}(U_1(\fN), \CC)$.

   \begin{remark}
    To fix normalisations, we point out that if $\cF$ is a Hilbert cusp form, then $\cF$ has a \emph{Fourier--Whittaker expansion}
     \[ \cF\left(\stbt x 0 0 1, \tau\right) = \|x\|_{\AA_F} \sum_{\alpha \in F^{\times +}} \alpha^{-\underline{t}}\, c(\alpha x, \cF) e^{2 \pi i \Tr_{F/\QQ}(\alpha\tau)} \]
     where $c(-, \cF)$ is a locally-constant $\CC$-valued function on $\AA_{F, f}^\times$. If the level is of the form $U_1(\fN)$, then $c(x, \cF)$ only depends on the fractional $\widehat\cO_F$-ideal $\fn$ generated by $x$, and is zero unless $\fn \subseteq \fd^{-1}$; and the Hecke operators satisfy $c(\fn, \cT(\fm) \cF) = c(\fm\fn, \cF)$ whenever $(\fm, \fn\fd) = 1$.
   \end{remark}

   \begin{definition}
    We say that $\cF \in S_{\urt}(U_1(\fN), \CC)$ is an \emph{eigenform} if it is an eigenvector for the Hecke operators $\cT(\fm)$ for all ideals $\fm$. We say $\cF$ is \emph{normalised} if $c(\fd^{-1}, \cF) = D^{-(t_1 + t_2)/2}$.
   \end{definition}

   Note that if $\cF(g, \tau)$ is a Hilbert  modular form of weight $\urt$, then $\cF^{(s)}(g, \tau) \coloneqq \|\det g\|_{\AA_F}^{-s} \cF(g, \tau)$ is a Hilbert modular form of weight $(\underline{r}, \underline{t} + (s, s))$ for any $s \in \ZZ$, and its Fourier--Whittaker coefficients satisfy $c(\fn, \cF^{(s)}) = \Nm_{F/\QQ}(\fm)^s c(\fn, \cF)$. Our choice of normalisations is such that $\cF^{(s)}$ is a normalised eigenform if $\cF$ is. (Note that the restrictions of $\cF$ and $\cF^{(s)}$ to $\cH_F$ are identical.)

   If $\cF$ is an eigenform, then its $\cT(\fm)$-eigenvalues $\lambda(\fm)$ all lie in a number field $L$, and there is a finite-order Hecke character $\varepsilon: \Cl^+(F, \fN) \to L^\times$ such that $\cS(\fm)$ acts as $\Nm_{F/\QQ}(\fm)^{w-2} \varepsilon(\fm)$. Exactly as in the familiar case $F = \QQ$, a normalised eigenform is uniquely determined by its Hecke eigenvalues.

  \subsection{Pullback to \texorpdfstring{$G^*$}{G*}}

   Let $U^* \subset G^*(\AA_f)$ be open compact, and let $(r_1, r_2)$ be nonnegative integers (not necessarily of the same parity). We define the space $S_{\underline{r}}(U^*, \CC)$ of Hilbert cusp forms for $G^*$ of weight $\underline{r}$ and level $U^*$ as functions $G^*(\AA_f) \times \cH_F \to \CC$ satisfying the analogues of the conditions (i)--(iii) of Definition \ref{def:HMF}, using the weight $\underline{r}$ action of $G^*(\QQ)^+$ in place of the weight $\urt$ action of $G(\QQ)^+$ in (iii).

   This space has a right action of the Hecke algebra $\CC[U^* \backslash \cG / U^*]$, where $\cG$ is as in Definition \ref{def:calG}. In particular, when $U^* = U_1^*(\fN)$ for some $\fN \triangleleft \cO_F$, we have an action of the operators $R, S, T$ defined in \S\ref{sect:heckealg}, and $R(x)$ for $x \in F^{\times}$ acts as multiplication by $x^{\underline{r} - \underline{2}}$.

   If we now impose the assumption that $r_1 = r_2 \bmod 2$, and choose $\underline{t} = (t_1, t_2) \in \ZZ^2$ such that $r_1 + 2t_1 = r_2 + 2t_2$, we can compare the theories for $G^*$ and $G$.

   \begin{proposition}
    If $\cF \in S_{\urt}(U, \CC)$, then the function $\jmath^*(\cF)$ on $G^*(\AA_f) \times \cH_F$ defined by
    \[ \jmath^*(\cF)(g, \tau) = \|\det g\|_{\AA_\QQ}^{(t_1 + t_2)} \cF(g, \tau) \]
    is an element of $S_{\underline{r}}(U^*, \CC)$, where $U^* = U \cap G^*$.\qed
   \end{proposition}

   Note that this construction is twist-invariant, i.e.~$\jmath^*(\cF^{(s)}) = \jmath^*(\cF)$ for any $s$.

   \begin{proposition}
    \label{prop:jpullback}
    Let $\fN \triangleleft \cO_F$. Then the action of $\cO_F^{\times +}$ on $S_{\underline{r}}(U_1^*(\fN), \CC)$ via $x \mapsto x^{\underline{t}} T(x)$ factors through the finite quotient
    \[
     \mathfrak{T} = {\cO_F^{\times +}} / {\{ x^2: x \in \cO_F^{\times+},\, x = 1 \bmod \fN \}}.
    \]
    The image of the pullback map
    \[
     \jmath^*: S_{\urt}(U_1(\fN), \CC) \to  S_{\underline{r}}(U_1^*(\fN), \CC)
    \]
    is precisely the $\mathfrak{T}$-invariants of $S_{\underline{r}}(U_1^*(\fN), \CC)$.

    Moreover, if $\fm$ is a narrowly principal ideal, generated by some $x \in \cO_{F}^+$, then the map $\jmath^*$ intertwines the action of the operator $\cT(\fm)$ on the source with the action of $x^{\underline{t}} \cdot T(x)$ on the target, and similarly for $\cS(\fm)$ and $x^{\underline{2t}} S(x)$ if $\fm$ is coprime to $\fN$.
   \end{proposition}

   \begin{proof}
    Since $R(x)$ acts as $\Nm_{F/\QQ}(x)^{\underline{r}-\underline{2}}$, and for $x$ a unit we have $T(x^2) = S(x) = \langle x \rangle R(x)$, the action of $\cO_F^{\times +}$ factors through $\mathfrak{T}$.

    It is easy to see that $\jmath^*$ intertwines $\cT(\fn)$ with $x^{\underline{t}} \cdot T(x)$, since the double cosets $\cT(\fn)$ and $T(x)$ share a common set of single-coset representatives, and similarly for $\cS(\fn)$ and $S(x)$. Since $\cT(x)$ is the identity map for $x$ a unit, this shows in particular that the image of $\jmath^*$ is contained in the $\mathfrak{T}$-invariants.

    It remains to prove that any element of $S_{\underline{r}}(U_1^*(\fN), \CC)$ invariant under $\mathfrak{T}$ lies in the image of $\jmath$. This follows from the fact that $S_{\urt}(U_1(\fN), \CC)$ contains, as a direct summand, the space of holomorphic functions on $\cH_F$ which are invariant under the weight $\urt$ action of $\Gamma_1(\fN) = U_1(\fN) \cap \GL_2^+(F)$ and vanish at the cusps. An element of $S_{\underline{r}}(U_1^*(\fN), \CC)$ gives a function invariant under the subgroup $\Gamma_1^*(\fN) = \Gamma_1(\fN) \cap \SL_2(\cO_F)$; and the Hecke operators $T(x)$, for $x \in \cO_F^{\times +}$, give representatives for the quotient $\Gamma_1(\fN) / \Gamma_1^*(\fN)$. (This can be seen as an instance of the general result of Proposition \ref{prop:Gquotient}.)
   \end{proof}

   \begin{remark}
    If the narrow class group of $F$ is trivial, any $\cF \in S_{\urt}(U_1(\fN), \CC)$ is uniquely determined by its restriction to $\{1\} \times \cH_F$, so $\jmath^*$ is injective. Conversely, if the narrow class group is nontrivial, the map $\jmath^*$ is very rarely injective, because of the following construction. Let $\kappa$ be a non-trivial character of the narrow class group of $F$. Then for any $\cF \in S_{\urt}(U_1(\fN), \CC)$ there is a twisted form $\cF \otimes \kappa$ satisfying $c(\fn, \cF \otimes \kappa) = \kappa(\fn) c(\fn, \cF)$. We have $\jmath^*(\cF \otimes \kappa) = \jmath^*(\cF)$, but $\cF$ and $\cF \otimes \kappa$ are very unlikely to be equal.
   \end{remark}

   \begin{definition}
    We say an eigenform $\cF \in S_{\urt}(U_1(\fN), \CC)$ is an \emph{exceptional CM form} if it has complex multiplication by a quadratic extension $E/F$ contained in the Hilbert class field of $F$.
   \end{definition}

   Note that this notion depends only on the newform associated to $\cF$; and the field $E$ is necessarily totally imaginary (i.e.~it is a CM field) and biquadratic over $\QQ$.

   \begin{lemma}
    \label{lemma:CMvanish}
    Let $\cF \in S_{\urt}(U_1(\fN), \CC)$ be a normalised eigenform, with $r_i \ge 2$. If $\jmath^*(\cF)$ is zero, then $\cF$ is an exceptional CM form. Conversely, if $\cF$ is a newform and an exceptional CM form, then $\jmath^*(\cF) = 0$.
   \end{lemma}

   \begin{proof}
    Since $\cF$ is a normalised eigenform, we have $c(\fd^{-1}, \cF) \ne 0$. On the other hand, $\jmath^*(\cF) = 0$ if and only if $c(x, \cF) = 0$ for all $x \in F^{\times +}$. Using the relation between Hecke eigenvalues and Fourier--Whittaker coefficients, it follows that $\jmath^*(\cF)$ is zero if and only if the Hecke eigenvalue $\lambda(\fn)$ of $\cF$ is zero for every $\fn$ in the narrow ideal class of $\fd$.

    It is a standard fact that if an eigenform $\cF$ satisfies $\lambda(\fp) = 0$ for a positive-density set of prime ideals $\fp$, then $\cF$ must have CM by some totally imaginary quadratic extension $E / F$ (e.g.~this follows from the fact that there exist many primes $v$ for which the standard Galois representation $\rstd$ has large image, cf.~\cite[Theorem B.5.2]{nekovar12}). Moreover, in the CM case, we have $\lambda(\fp) \ne 0$ for every prime $\fp \nmid \fN$ split in $E$, since CM forms are automatically ordinary at split primes. Thus we deduce that every prime in the narrow ideal class of $\fd$ must be inert in $E / F$, from which it follows that $E$ is contained in the narrow class field.

    Conversely, suppose $\cF$ is a newform that is an exceptional CM form, and let $\kappa$ be the quadratic Hecke character corresponding to $E / F$. We have $\kappa(\fd) = -1$, so $\cF$ and $-\cF \otimes \kappa$ are normalised newforms. As they have the same Hecke eigenvalues at almost all primes, they are in fact equal, by the multiplicity one theorem. Hence we have $c(\fn, \cF) = -\kappa(\fn)c(\fn, \cF)$ for all $\fn$, and in particular $c(\fn, \cF)$ is zero for all ideals in the kernel of $\kappa$. Thus $\jmath^*(\cF) = 0$.
   \end{proof}

   \begin{proposition}
    \label{prop:jmultone}
    Suppose $\cF$ is a normalised eigenform of level $\fN$, with $r_i \ge 2$, which is not an exceptional CM form. Then any form $\cG \in S_{\underline{r}}(U_1^*(\fN))$ satisfying $T(x) \cG = x^{-\underline{t}} \lambda(x) \cG$, for all $x \in \cO_F^+$, is a scalar multiple of $\jmath^*\cF$.
   \end{proposition}

   \begin{proof}
    Let us write $\cG(1, \tau) = \sum_{\alpha} \alpha^{-\underline{t}}c(\alpha, \cG) e^{2\pi i \Tr(\alpha\tau)}$, where the sum runs over totally-positive $\alpha \in \fd^{-1}$. By assumption, $c(\alpha, \cG)$ depends only on the ideal generated by $\alpha$.

    Let $I$ denote the set of integral ideals in the narrow ideal class of $\fd$, and for each $\fn \in I$, let $c(\fn)$ be the common value of $c(\alpha, \cG)$ over all totally-positive generators $\alpha$ of the ideal $\fn \fd^{-1}$. Then we have the relation $c(\beta \fn) = \lambda(\beta) c(\fn)$ for $\alpha \in \cO_F^+$ coprime to $\fn$; and we want to show that this determines all the $c(\fn)$ up to a scalar.

    If $\fd$ is trivial in the narrow class group of $F$ (that is, the fundamental unit has norm $-1$) then $I$ is simply the set of principal ideals, and we see immediately that $c(\fn) = \lambda(\fn) c(1)$ for all $\fn \in I$, so we are done.

    So let us assume that the fundamental unit has norm $+1$, so that $I$ is the set of ideals that are principal but not narrowly-principal. Since $\cF$ is not an exceptional CM form, there exist infinitely many primes $\fp \in I$ such that $\lambda(\fp) \ne 0$. Let $\fp_0$ be one of these, and define $C = \lambda(\fp_0)^{-1} c(\fp_0)$.

    Let $\fn \in I$ be arbitrary. We want to show that $c(\fn) = C \lambda(\fn)$. Firstly, suppose $\fp_0 \nmid \fn$. Let $\fp_1$ be another prime in $I$ not dividing $\fn \fp_0$ with $\lambda(\fp_1) \ne 0$; then we have
    \[ \lambda(\fp_0 \fp_1) c(\fn) = c(\fn\fp_0 \fp_1) = \lambda(\fn\fp_1) c(\fp_0), \]
    and hence $c(\fn) = \frac{\lambda(\fn\fp_1)}{\lambda(\fp_0 \fp_1)} c(\fp_0) = C\lambda(\fn)$ for all $\fn \in I$ not divisible by $\fp_0$.

    On the other hand, if $\fn \in I$ is divisible by $\fp_0$, we can pick some $\fp_1 \in I$ with $\lambda(\fp_1) \ne 0$ and $\fp_1 \nmid \fn$. By the previously-handled case we have $C = c(\fp_1) / \lambda(\fp_1)$, and applying the previous argument with $\fp_0$ replaced by $\fp_1$ we are done.
   \end{proof}

  \subsection{Galois representations}
   \label{sect:galrep}

   \begin{theorem}[Blasius--Rogawski--Taylor]
    Let $\cF$ be a Hilbert modular eigenform of weight $\urt$, with $r_1, r_2 \ge 2$. Let $L$ be the number field generated by the Hecke eigenvalues $\lambda(\fn)$. Then for every finite place $v$ of $L$, there is an irreducible 2-dimensional ``standard'' Galois representation
    \[ \rstd:  \Gal(\overline{F} / F) \to \GL_2(L_v), \]
    such that for all primes $\fl \nmid \fN \Nm_{L/\QQ}(v)$, the representation $\rstd$ is unramified at $\fl$ and we have
    \[ \det\left(1 - X \rstd(\Frob_\fl^{-1})\right) = 1 - \lambda(\fl) X + \Nm_{F/\QQ}(\fl)^{w-1} \varepsilon(\fl) X^2.\]
    Moreover, the Hodge numbers\footnote{That is, the negatives of the Hodge--Tate weights. For the complete avoidance of doubt, we state that in this paper the cyclotomic character has Hodge--Tate weight $+1$ and Hodge number $-1$.} of $\rstd$ at the primes above $\Nm_{L/\QQ}(v)$ are $\{ t_1, t_1 + r_1 - 1\}$ at one embedding and $\{ t_2, t_2 + r_2 - 1\}$ at the other.
   \end{theorem}

   We are not actually interested in the standard representation per se, but in its tensor induction to $\Gal(\QQbar / \QQ)$. Recall that if $H \subset G$ are groups with $[G : H] = 2$, $\sigma \in G - H$, and $\rho$ is a representation of $H$ on some vector space $V$, we define $\tInd_H^G(\rho)$ to be the isomorphism class of the representation of $G$ whose underlying space is $V \otimes V$, with $G$ acting via
   \[
    h \cdot (v \otimes w) = (h\cdot v) \otimes (\sigma^{-1} h \sigma \cdot w),\qquad \sigma \cdot (v \otimes w) = (\sigma^2 \cdot w) \otimes v.
   \]
   (The isomorphism class of this representation is independent of $\sigma$.)

   \begin{definition}
    For an eigenform $\cF$ and place $v$ as above, we define the 4-dimensional ``Asai'' Galois representation
    \[ \rho_{\cF, v}^{\Asai}: \Gal(\QQbar / \QQ) \to \GL_4(L_v) \]
    by
    \[ \rho_{\cF, v}^{\Asai} =  \tInd_F^\QQ(\rstd) \otimes L_v(t_1 + t_2). \]
    The Hodge numbers of this Galois representation are $\{0, r_1-1, r_2 - 1, r_1 + r_2 - 2\}$.
   \end{definition}

   \begin{remark}
    The Asai representation is unramified at all primes not dividing $p\Delta \Nm_{F/\QQ}(\fN)$. The twist by $t_1 + t_2$ implies that the Asai representation of $\cF$ is the same as that of the twist $\cF^{(s)}$, for any integer $s$.

    Note also that $\rho_{\cF, v}^{\Asai}$ actually preserves an orthogonal form up to scaling, i.e. its image lands in the general orthogonal group $\mathrm{GO}_4 \subseteq \GL_4$. The subgroup landing in the connected component $\mathrm{GSO}_4$ is exactly $\Gal(\overline{F} / F)$. This can be interpreted in terms of an isomorphism between $\mathrm{GO}_4$ and the Langlands $L$-group of $G^*$; this is investigated in more detail in \cite[\S 5]{DLP}.
   \end{remark}

   \begin{proposition}
    \label{prop:proj-Pl}
    Let $\cF$ be an eigenform of level $\fN$ with coefficients in $L$, and let $\ell \nmid \Delta \Nm_{F/\QQ}(\fN)$ be prime.
    \begin{enumerate}[(i)]
     \item There is a polynomial $P_\ell(\cF, X) \in L[X]$ such that for all prime $v \nmid \ell$ of $L$ we have
     \[
      \det\left(1 - X\, \rho_{\cF, v}^{\Asai}\left(\Frob_\ell^{-1}\right)\right) = P_\ell(\cF, X).
     \]
     \item The operator-valued Euler factor $P_\ell(X)$ defined in \S\ref{sect:asaiEF} acts on $\jmath^*(\cF)$ as $P_\ell(\cF, X)$.
    \end{enumerate}
   \end{proposition}

   \begin{proof}
    Since we have defined $\rho_{\cF, v}^{\Asai}$ as a twist of the tensor induction of $\rstd$, one can read off the coefficients of the characteristic polynomial of $\rho_{\cF, v}^{\Asai}\left(\Frob_\ell^{-1}\right)$ from those of $\rstd(\Frob_\fl^{-1})$, for the primes $\fl \mid \ell$ of $F$. These are, in turn, given by the Hecke eigenvalues of $\cF$. This gives a polynomial $P_\ell(\cF, X) \in L[X]$ satisfying (i). (The computation for $\ell$ split in $F$ is identical to \cite[Proposition 4.1.2]{leiloefflerzerbes14a}; the inert case is analogous.) Using the fact that $\jmath^*$ intertwines $\cT(\ell)$ with $\ell^{t + t'} T(\ell)$ and $\cS(\ell)$ with $\ell^{2(t + t')}S(\ell)$, one sees that this polynomial coincides with the action of $P_\ell(X)$ on $\jmath^*(\cF)$, which is (ii).
   \end{proof}

  \subsection{\'Etale cohomology of Hilbert modular varieties}

   We fix a weight $\urt$ as above, with $r_1 \ge 2, r_2 \ge 2$. We also fix a prime $p$, a finite extension $L / \QQ$ containing $F$, and a prime $v \mid p$ of $L$. Let $\mu = (r_1 - 2, r_2 - 2, t_1, t_2)$.

   \begin{definition}
    We let $\sH^{[\mu]}_{L_v}$ be the \'etale sheaf of $L_v$-vector spaces on $Y(U)$, for each sufficiently small $U$, that is the \'etale realisation of the relative Chow motive $\sH^{[\mu]}_L$ of \S \ref{sect:shvarG} above. We write $\sH^{(\mu)}_{L_v}$ for the dual of $\sH^{[\mu]}_{L_v}$.
   \end{definition}

   \begin{definition}
    \label{defn:asairep}
    Let $\cF$ be a cuspidal Hilbert eigenform of weight $\urt$ and level $\fN$, with Hecke eigenvalues in $L$. Then we define
    \[
     M_{L_v}(\cF) \coloneqq H^2_{\et}\left(Y_1(\fN)_{\QQbar}, \sH^{(\mu)}_{L_v}(t_1 + t_2)\right)[\cT(\fn) = \lambda(\fn)\ \forall \fn],
    \]
    where $\lambda(\fn)$ is the $\cT(\fn)$-eigenvalue of $\cF$.
   \end{definition}

   We are assuming here that $U_1(\fN)$ is sufficiently small (the case of eigenforms $\cF$ of level dividing $2$, $3$ or $\Delta$ can be dealt with by replacing $\cF$ with its $\fl$-stabilisation, for some auxiliary prime $\fl$). We define similarly the $\cF$-eigenspaces $M_{\Betti}(\cF)$, $M_{\dR}(\cF)$ associated to $\cF$ in Betti and de Rham cohomology, using the Betti and de Rham realisations of $\sH^{(\mu)}_L$. Note that each of these spaces lifts isomorphically to compactly-supported cohomology, since $\cF$ is cuspidal, and they are related by comparison isomorphisms as given in \cite{grothendieck66}.

   \begin{theorem}[Brylinski--Labesse, Nekov\'{a}\v{r}]
    The space $M_{L_v}(\cF)$ is 4-dimen\-sional, and isomorphic as a representation of $\Gal(\QQbar / \QQ)$ to $\rho_{\cF, v}^{\Asai}$.
   \end{theorem}

   \begin{proof}
    See \cite[\S 3.4]{brylinskilabesse84}, where the result is shown up to semisimplification by a comparison of traces, and \cite[Theorem 5.20]{nekovar-semisimplicity}, which establishes that $M_{L_v}(\cF)$ is always semisimple.
   \end{proof}

   Via Poincar\'e duality, we can identify the dual $M_{L_v}(\cF)^*$ with the maximal quotient of $H^2_{\et}\left(Y_1(\fN)_{\QQbar}, \sH^{[\mu]}_{L_v}(2-t_1-t_2)\right)$ on which the covariant Hecke operators $\cT'(\fn)$ act as $\lambda(\fn)$ for all $\fn$.

   \begin{corollary}
    Let $\cF$ be an eigenform of level $\fN$ and weight $\urt$, and suppose $\underline{r} = (k+2, k' + 2)$ for $k, k' \ge 0$.

    Then there is a canonical $\Gal(\QQbar / \QQ)$-equivariant map
    \[ \pr_{\cF}: H^2_{\et}\left(Y_1^*(\fN)_{\QQbar}, \TSym^{[k, k']} \sH_{L_v}(\cA)(2)\right) \to M_{L_v}(\cF)^*; \]
    and for each prime $\ell \nmid p\Delta \Nm_{F/\QQ}(\fN)$, this intertwines the dual operator-valued Asai Euler factor $P_\ell'(X)$ on the left-hand side with the polynomial $P_\ell(\cF, X)$ of Proposition \ref{prop:proj-Pl}.
   \end{corollary}

   \begin{proof}
    We know that the pullback of the \'etale sheaf $\sH^{[\mu]}_{L_v}$ to $Y_1^*(\fN)$ is isomorphic to $\TSym^{[k,k']} \sH_{L_v}(\cA)(t_1 + t_2)$, so we have a pushforward map
    \[
     \jmath_*: H^2_{\et}\left(Y_1^*(\fN)_{\QQbar}, \TSym^{[k, k']} \sH_{L_v}(\cA)(2)\right) \to H^2_{\et}\left(Y_1(\fN)_{\QQbar}, \sH_{L_v}^{[\mu]}(\cA)(2-t_1-t_2)\right).
    \]
    We define $\pr_{\cF}$ to be the composite of this map $\jmath_*$ with the projection to the $\cF$-isotypical component for the covariant Hecke operators $\cT'(\fm)$, which we have seen is canonically isomorphic to $M_{L_v}(\cF)^*$.

    The map $\jmath_*$ intertwines the covariant Hecke operator $T'(x)$ with $x^{-\underline{t}} \cT'(x)$, for all totally-positive $x \in \cO_F$, and similarly for $S'(x)$ and $\cS'(x)$; the same computation as in Proposition \ref{prop:proj-Pl}(ii) thus shows that $\pr_{\cF}$ intertwines the action of $P_\ell'(X)$ with $P_\ell(\cF, X)$.
   \end{proof}

   \begin{remark}
    It is interesting to consider whether one can give a construction of $M_{L_v}(\cF)^*$ using the cohomology of $Y_1^*(\fN)$ alone, without using the map $\jmath$. Let us write $N_{L_v}(\cF)^*$ for the maximal quotient of $H^2_{\et}\left(Y_1^*(\fN)_{\QQbar}, \TSym^{[k, k']} \sH_{L_v}(\cA)(2)\right)$ on which $T'(x)$ acts as $x^{-\underline{t}} \lambda(x)$ for every $x \in \cO_{F}^+$. Then we can factor $\pr_{\cF}$ as the composite of projection to the quotient $N_{L_v}(\cF)^*$, and a map $N_{L_v}(\cF)^* \to M_{L_v}(\cF)^*$ induced by $\jmath_*$.

    This map $N_{L_v}(\cF)^* \to M_{L_v}(\cF)^*$ is, perhaps surprisingly, not always an isomorphism. Using duality and the comparison between \'etale and de Rham cohomology, one sees that this map is surjective if and only if $\jmath^*(\cF) \ne 0$, and it is injective if and only if $\jmath^*(\cF)$ spans the $\cF$-eigenspace for the Hecke operators $\{ T(x): x \in \cO_F^+\}$ acting on $S_{\underline{r}}(U_1^*(\fN), \CC)$. By Lemma \ref{lemma:CMvanish} and Proposition \ref{prop:jmultone}, if $\cF$ is not an exceptional CM form, both of these conditions are satisfied and thus $N_{L_v}(\cF)^* \to M_{L_v}(\cF)^*$ is an isomorphism.
   \end{remark}

   Since $\cF$ is cuspidal, the generalised eigenspace associated to $\cF$ in the cohomology $H^\bullet_{\et}\left(Y_1(\fN)_{\QQbar}, \sH^{(\mu)}_{L_v}\right)$ is concentrated in degree 2. Hence, if $K$ is a finite extension of $\QQ$, and $\cO_{K, S}$ is a localisation of $\cO_K$ at some finite set of rational primes $S$ containing all primes dividing $p \Delta \Nm_{F/\QQ}(\fN)$, then the Hochschild--Serre spectral sequence allows us to regard $\pr_{\cF}$ as a map
   \[ H^3_{\et}\left(Y_1^*(\fN)_{\cO_{K, S}}, \TSym^{[k, k']} \sH_{L_v}(\cA)(2-j)\right) \to H^1\left(\cO_{K, S}, M_{L_v}(\cF)^*(-j)\right). \]

   \begin{definition}
    \label{def:AFf}
    For $\cF$ an eigenform as in the previous corollary, and $0 \le j \le \min(k, k')$, we define
    \[
     \AF^{[\cF, j]}_{\et} = \pr_{\cF}\left( \AF^{[k, k', j]}_{\et, \fN}\right) \in H^1\left(\ZZ\left[\tfrac{1}{p \Delta \Nm_{F/\QQ}(\fN)}\right], M_{L_v}(\cF)^*(-j)\right),
    \]
    and for $M \ge 1$ and $a \in \cO_F / (M\cO_F + \ZZ)$,
    \[
     \AF^{[\cF, j]}_{\et, M, a} = \pr_{\cF}\left( \AF^{[k, k', j]}_{\et, M, \fN, a}\right) \in H^1\left(\ZZ\left[\mu_m, \tfrac{1}{p M \Delta \Nm_{F/\QQ}(\fN)}\right], M_{L_v}(\cF)^*(-j)\right).
    \]
   \end{definition}

   \begin{corollary}
    \label{cor:AFf}
    Let $\ell \nmid M \Nm_{F/\QQ}(\fN)$ be a prime, and suppose that either $\ell$ is inert in $F$, or $\ell$ is split and the primes above $\ell$ are narrowly principal. Then we have the relation
    \begin{multline*}
     \operatorname{norm}_{M}^{\ell M}\left( \AF^{[\cF, j]}_{\et, \ell M, a}\right) \\= \ell^j \sigma_\ell\left( (\ell - 1)(1 - \ell^{k + k' - 2j} \sigma_\ell^{-2} \varepsilon_\cF(\ell)) - P_\ell(\cF, \ell^{-1-j} \sigma_\ell^{-1})\right) \cdot \AF^{[\cF, j]}_{\et, M, a},
    \end{multline*}
    where $\varepsilon_{\cF}$ is the character of $\cF$.
   \end{corollary}

   \begin{proof}
    This is an immediate consequence of Theorem \ref{thm:AFeltproperties} (or of Theorem \ref{thm:TlrelationB}, which we shall prove below).
   \end{proof}

   This definition and corollary complete the proof of Theorem \ref{lthm:A} of the introduction.

 \section{The complex regulator}
  We now evaluate the image of our motivic cohomology classes $\AF^{[k, k', j]}_{\mot, \fN}$ under Beilinson's complex-analytic regulator map. We shall follow \cite{kings98} closely.

  Throughout this section, we fix a Hilbert modular eigenform $\cF$ of level $\fN$ and some weight $\urt$, where $\underline{r} = (k + 2, k' + 2)$, and a number field $L$ containing the Hecke eigenvalues of $\cF$. We choose a square root of $D$ in $F$, and normalise our embeddings such that $\theta_1(\sqrt{D})$ is the positive square root.

  \subsection{The Asai \texorpdfstring{$L$}{L}-function}
   \label{AsaiLfunction}

   For any prime $\ell$, we define a local Euler factor by
   \[ P_\ell(\cF, X) = \det\left(1 - X \Frob_\ell^{-1}: M_{L_v}(\cF)^{I_\ell} \right) \]
   for any $v \nmid \ell$. This lies in $L[X]$ and is independent of the choice of $v$. It agrees with the definition given (implicitly) in Proposition \ref{prop:proj-Pl} above when $\ell \nmid \Delta \Nm_{F/\QQ}(\fN)$ (in which case $I_\ell$ acts trivially).

   \begin{definition}
    Define the primitive Asai $L$-function of $\cF$ as the product
    \[ L_{\Asai}(\cF,s)=\prod_\ell P_\ell(\cF,\ell^{-s})^{-1}. \]
   \end{definition}

   This Euler product converges for $\Re(s) > \frac{k + k'}{2} + 2$; it has analytic continuation to all $s \in \CC$ (except for a possible pole at $s = k + 2$ if $k = k'$ and $\cF$ is twist-equivalent to its internal conjugate $\cF^\sigma$), and satisfies a functional equation relating $s$ with $k + k' + 3 - s$. The form of the functional equation forces $L_{\Asai}(\cF, s)$ to vanish to order exactly 1 at integers $s \in \{1, \dots, \min(k, k')\}$, and also at $s = 1 + \min(k, k')$ if there is no pole at $s = k + 2$.

   It will be convenient to work instead with an ``approximation'' to $L_{\Asai}$ which is more straightforwardly linked to period integrals. We let $\chi$ be the Dirichlet character given by the restriction of the nebentype character of $\cF$ to $(\ZZ / N \ZZ)^\times$, where $N \ZZ = \fN \cap \ZZ$.

   \begin{definition}
    \label{def:Limp}
    For $n \in \mathbf{N}$, let $\lambda(n)$ denote the $\cT(n)$-eigenvalue of $\cF$, and $\alpha(n) = n^{-(t + t')} \lambda(n)$. Then we define the imprimitive Asai $L$-function by
    \[ L^{\imp}_{\Asai}(\cF,s)=L_{(N)}(\chi, 2s + 2 - k - k')\cdot \sum_{n\geq 1} \alpha(n) n^{-s}.\]
   \end{definition}

   One checks that
   \[ L^{\imp}_{\Asai}(\cF, s) =  L_{\Asai}(\cF, s) \cdot \prod_{\ell \mid N} C_\ell(\ell^{-s}),\]
   where the ``error terms'' $C_\ell(X)$ are polynomials; cf.~\cite{asai77}. Moreover, $C_\ell(X)$ always divides $P_\ell(\cF, X)$.

   \begin{proposition}
    \label{prop:Lnonvanish}
    If $|k - k'| \ge 3$, then $\ord_{s = 1 + j} L^{\imp}_{\Asai}(\cF,s) = 1$ for all integers $j$ with $0 \le j \le \min(k, k')$.
   \end{proposition}

   \begin{proof}
    We have seen that when $k \ne k'$ the function $L_{\Asai}(\cF,s)$ vanishes to order exactly 1 at all such $s$, so it suffices to check that the error term $\prod_{\ell \mid N} C_\ell(\ell^{-s})$ is non-zero at these values. However, a case-by-case check similar to Proposition 4.1.3 of \cite{leiloefflerzerbes14a} shows that all zeroes of this error term have real parts in the interval $\left[ \tfrac{k+k'}{2}, \tfrac{k + k' + 2}{2}\right]$, and the assumption that $|k - k'| \ge 3$ implies that this range has no overlap with the range we consider.
   \end{proof}

  \subsection{Non-holomorphic eigenforms}
   \label{sect:nonhol}

   We define \emph{Hilbert modular forms anti-holo\-morphic at $\theta_1$} exactly as in Definition \ref{def:HMF}, but requiring that for each $g \in G(\AA_f)$, the function $\cF(g, -)$ on $\cH_F$ should be anti-holomorphic in $\tau_1$ and holomorphic in $\tau_2$, and using $(c_1 \bar\tau_1 + d_1)^{-r_1}$ in place of $(c_1 \tau_1 + d_1)^{-r_1}$ in the definition of the weight $\urt$ action.

   Then the Fourier--Whittaker expansion of such a form $\cF$ can be written as
   \[
    \cF\left(\stbt x 0 0 1, \tau\right) = \|x\|
    \sum_{\mathclap{\substack{\alpha \in F^\times \\
    \theta_1(\alpha) < 0, \theta_2(\alpha) > 0}}}
    |\alpha^{-\underline{t}}| \, c(\alpha x, \cF)
    e^{2\pi i (\theta_1(\alpha) \bar\tau_1 + \theta_2(\alpha) \tau_2)}.
   \]

   \begin{lemma}
    If $\cF$ is a holomorphic normalised eigenform as before, then there are unique anti-holomorphic forms $\cF^{\ah, 1}$ (anti-holomorphic at $\theta_1$ and holomorphic at $\theta_2$) and $\cF^{\ah, 2}$ (holomorphic at $\theta_1$ and anti-holomorphic at $\theta_2$), of the same level and weight as $\cF$, for which the Fourier--Whittaker coefficients $c(x, \cF^{\ah, i})$ coincide with those of $\cF$ for $i = 1, 2$.
    \end{lemma}

    \begin{proof}
     We sketch the construction of $\cF^{\ah,1}$; the construction of  $\cF^{\ah,2}$ is very similar. Let $\eta\in F^\times$ be any element such that $\theta_1(\eta)<0$ and $\theta_2(\eta)>0$. Define
     \[ \cF^{\ah,1}(g,\tau) = -\eta^{\underline{1}-\underline{t}}\cF\left( \begin{pmatrix} \eta & 0 \\ 0 & 1\end{pmatrix}g,(\theta_1(\eta)\bar{\tau}_1,\theta_2(\eta)\tau_2)\right).\]
     This is modular of the same level as $\cF$, and is independent of the choice of $\eta$. A straightforward computation shows that its Fourier--Whittaker coefficients are the same as those of $\cF$.
    \end{proof}

   \begin{remark}
    If the fundamental unit of $F$ has norm $-1$, we may choose $\eta$ to be a unit. In this case, the restriction of $\cF^{\ah, 1}$ to the upper half-plane $\cH_F$ can be described in terms of that of $\cF$: up to scalars, it is the pullback of $\cF$ via the map given by $(\tau_1, \tau_2) \mapsto (\eta_1 \bar\tau_1, \eta_2 \tau_2)$. However, if the fundamental unit of $F$ has norm $+1$, there is no direct relation between $\cF$ and $\cF^{\ah, 1}$ at the level of functions on $\cH_F$.
   \end{remark}

   \begin{notation}
    We write $\iota^*(\cF^{\ah, 1})$ for the $C^\infty$ function on $\cH$ defined by $\tau \mapsto \cF^{\ah, 1}(1, (\tau, \tau))$.
   \end{notation}

   (This is naturally the restriction to $\cH$ of a function on $\GL_2(\AA_f) \times \cH$ satisfying an appropriate automorphy property; but since $Y_{1, \QQ}(N)$ is geometrically connected, no information is lost in treating it as a function on the upper half-plane.)


  \subsection{The period integral}

   Let $k \in \ZZ$ and $\alpha\in\QQ/\ZZ$, with $\alpha \ne 0$. For $\tau\in \cH$ and $s\in\CC$ with $k+2\operatorname{Re}(s)>2$, we define the Eisenstein series
   \[ E_\alpha^{(k)}(\tau,s)=(-2\pi i)^{-k}\pi^{-s}\Gamma(s+k)\sum_{(m,n)\in\ZZ^2} \frac{\Im(\tau)^s}{(m\tau+n+\alpha)^k|m\tau+n+\alpha|^{2s}}.\]

   \begin{proposition}[See \cite{leiloefflerzerbes14a}, Proposition 4.2.2] \
    \begin{itemize}
     \item For fixed $k,\tau,\alpha$, $E_\alpha^{(k)}(\tau,s)$ has meromorphic continuation to the whole $s$-plane, which is holomorphic everywhere if $k\neq 0$.

     \item If $N\alpha=0$, then for a fixed $s$ the series $E_\alpha^{(k)}(\tau, s)$ is a $C^\infty$ function of $\tau$ which is preserved by the weight $k$ action of $\Gamma_{1, \QQ}(N)$. It is holomorphic in $\tau$ if $k \ge 1$ and $s = 0$ or $s = 1-k$.

     \item We have\footnote{Note that there is a sign error in \cite{leiloefflerzerbes14a}; the minus sign is correct.}
     \[ E^{(0)}_\alpha(\tau,0)=-2\log|g_{0,\alpha}(\tau)|,\]
     where $g_{0,\alpha}$ is the Siegel unit given in \S2.2 of \emph{op.~cit.}\qed
    \end{itemize}
   \end{proposition}

   Applying the usual Rankin--Selberg ``unfolding'' technique, one obtains the following formula, which is the analogue in our present setting of \cite[Equation (3.5.3)]{kingsloefflerzerbes15a}:

   \begin{theorem}[Asai]
    \label{thm:asai}
    We have
    \begin{multline*}
     \int_{\Gamma_1(N) \backslash \cH}
     \iota^*(\cF^{\ah, 1})(x + iy)\, E_{1/N}^{(k - k')}(x + iy, s - k - 1) y^k \d x \d y
     =\\
      \frac{D^{\frac{s}{2}}\Gamma(s) \Gamma(s - k' - 1)}{N^{k +k'-2s + 2}\, 2^{k-k' + 2s}\, (-i)^{k-k'} \pi^{2s - k' - 1}}\  L^{\imp}_{\Asai}(\cF, s).\qedhere\qed
    \end{multline*}
   \end{theorem}

   (This integral was first studied in \cite{asai77} in the case when $k = k'$, $\fN = 1$, and $F$ has narrow class number 1. For a more general treatment see \cite{im94}.)

   \begin{remark}\label{rmk:impderiv}
    If $0\leq j\leq k'$, then as remarked above, we have $L_{\Asai}^{\imp}(\cF,1+j)=0$, and $\Gamma(s-k'-1)$ has a simple pole at $s=j+1$ with residue $\frac{(-1)^{j-k'}}{(k'-j)!}$. Hence
    \begin{multline*}
     \int_{\Gamma_1(N) \backslash \cH}
     \iota^*(\cF^{\ah, 1})(x + iy)\, E_{1/N}^{(k - k')}(x + iy, j - k) y^k \d x \d y
     =  \\
      \frac{(-1)^{k'-j}D^{\frac{j+1}{2}}\Gamma(j+1) }{N^{k +k'-2j}\, 2^{k-k' + 2j+2}\, (-i)^{k-k'} \pi^{2j+1 - k' }(k'-j)!}\  \frac{d}{ds} L^{\imp}_{\Asai}(\cF, s)|_{s=1+j}.
    \end{multline*}
   \end{remark}


  \subsection{The regulator formula}

   Let $X_{\RR}\rightarrow \Spec \RR$ be a separated scheme, and denote by $\operatorname{MHM}_{\RR}(X_{\RR})$ the category of algebraic $\RR$-mixed Hodge modules on $X_{\RR}$ (see \cite{saito86, saito89, HW98}). For $M_\RR\in \operatorname{MHM}_{\RR}(X_{\RR})$, define the absolute Hodge cohomology groups
   \[ H^i_{\cH}(X_{\RR},M_{\RR})=R^i\Hom_{\operatorname{MHM}_{\RR}(X_{\RR})}(\RR(0),M_{\RR}).\]
   For the properties of this cohomology theory, see \cite[\S 2.3]{kingsloefflerzerbes15a}.

   \subsubsection{The Hodge realisation of \texorpdfstring{$M(\cF)$}{M(F)}}

    Recall that we have defined de Rham and Betti cohomology spaces $M_{\dR}(\cF)$ and $M_B(\cF)$ attached to $\cF$, which are 4-di\-men\-sion\-al $L$-vector spaces. Via the comparison isomorphism $M_{\dR}(\cF) \otimes \CC = M_B(\cF) \otimes \CC$, we can regard the pair $(M_{\dR}(\cF), M_B(\cF))$ as defining a pure $\RR$-Hodge structure $M_{\cH}(\cF)$ (whose weight is $k + k' + 2$).

    As in \cite[\S 5.4]{kingsloefflerzerbes15a}, cup-product gives a perfect duality of $(L \otimes_{\QQ} \RR)$-modules
    \begin{equation}
     \label{eq:perfdual}
     H^0(\RR, M_\cH(\cF)(1 + j)) \times H^1(\RR, M_{\cH}(\cF)^*(-j)) \to L\otimes_{\QQ} \RR
    \end{equation}
    for any $j \in \ZZ$, and we have
    \[
     H^0(\RR, M_{\cH}(\cF)(1 + j)) = \Fil^{1 + j} M_{\dR}(\cF)_{\RR} \cap i^{1 + j} M_B(\cF)_{\RR}^{F_\infty = (-1)^{1 + j}} \subset M_\dR(\cF)_\CC.
    \]
    Here $F_\infty$ is the ``infinite Frobenius'' (the endomorphism of Betti cohomology induced by complex conjugation on the $\CC$-points of $Y_1(\fN)$).

    \begin{proposition}
     For $0 \le j \le \min(k, k')$ the space $H^0(\RR, M_{\cH}(\cF)(1 + j))$ is free of rank 1 over $L \otimes_{\QQ} \RR$, and it is spanned by the class in $M_\dR(\cF)_\CC$ of the $C^\infty$ differential form $\varpi_{\cF, 1} + (-1)^{1 + j} \varpi_{\cF, 2}$, where
     \begin{align*}
       \varpi_{\cF, 1} &= (-1)^{k + 1} (2\pi i)^{k + k' + 2} \cF^{\ah, 1}(\tau) \d\bar{z}_1^k \d z_2^{k'} \d\bar{\tau}_1\d\tau_2,\\
       \varpi_{\cF, 2} &= (-1)^{k' + 1} (2\pi i)^{k + k' + 2} \cF^{\ah, 2}(\tau) \d z_1^k\d\bar{z}_2^{k'}\d\tau_1\d\bar\tau_2.
     \end{align*}
     Here, $\cF^{\ah, 1}$ and  $\cF^{\ah, 2}$  are as defined in Section \ref{sect:nonhol}.
    \end{proposition}
    \begin{proof}
     The same argument as in \cite[\S 5.4]{kingsloefflerzerbes15a} shows that $H^1(\RR,M_{\cH}(\cF)^*(-j))$ is free of rank 1 over  $L \otimes_{\QQ} \RR$, so the statement for $H^0(\RR, M_\cH(\cF)(1 + j))$ follows from \eqref{eq:perfdual}. The argument in \S 6.2 in \emph{op.cit} generalises immediately to show that $\varpi_{\cF, 1} + (-1)^{1 + j} \varpi_{\cF, 2}$ is a basis of this space.
    \end{proof}

   \subsubsection{The Hodge Eisenstein class for \texorpdfstring{$\GL_2$}{GL2}}

    We take $N \ge 1$ an integer, and we write $Y = Y_{1, \QQ}(N)_{\RR}$. On $Y$ we have a natural Hodge module $\sH_\RR$, defined as the Hodge realisation of $\sH(\cE)$, where $\cE$ is the universal elliptic curve over $Y$. We write $\pi_1:\TSym^k\sH_\CC\cong \TSym^k\sH_\RR\otimes\CC\to \TSym^k\sH_\RR(1)$ for the map induced by the projection $\CC\to \RR(1)$, $z\mapsto (z-\overline{z})/2$.

    \begin{proposition}
     \label{prop:explicitdelignecoho}
     The group $H^1_\cH\left(Y, \TSym^k\sH_\RR(1)\right)$ is the group of equivalence
     classes of pairs $(\alpha_\infty,\alpha_\dR)$, where
     \[
     \alpha_\infty\in \Gamma\left(Y(\CC), \TSym^k\sH_\RR\otimes \mathscr{C}^\infty\right)
     \]
     is a $\mathscr{C}^\infty$-section of $\TSym^k(\sH_\RR)(1)$, and
     \[ \alpha_\dR \in \Gamma\left(Y, \TSym^k(\Fil^0 \sH_{\dR}) \otimes \Omega^1_{X_1(N)}(C)\right)\]
     is an algebraic section with simple poles along $C \coloneqq X_1(N)\setminus Y_1(N)$,
     such that
     \[
      \nabla(\alpha_\infty)=\pi_1(\alpha_\dR).
     \]
     A pair $(\alpha_{\infty}, \alpha_{\dR})$ is equivalent to $0$ if we have
     \[ (\alpha_{\infty}, \alpha_{\dR}) = (\pi_1(\beta), \nabla(\beta)) \text{ for some }\beta \in \Gamma(X_1(N)_\RR, \TSym^k(\Fil^0 \sH_{\dR})(C)).\]
    \end{proposition}

    After pulling back to the upper half-plane $\cH$, these can be described by non-holomorphic modular forms. More precisely, the pullback of $\sH^\vee_{\CC}$ is the sheaf of relative differentials on $\mathcal J = \CC \times \cH$ over $\cH$, and is thus spanned by $\d z$ and $\d\bar z$, where $z$ is the coordinate on $\CC$.

    \begin{definition}
     For $r + s = k$, let $w^{(r,s)}$ be the $\mathscr{C}^\infty$ section $\d z^r \d \overline{z}^s$ of $\Sym^k\sH^\vee_\CC$, and $w^{[r,s]}$ the dual basis of $\TSym^k\sH_\CC$.
    \end{definition}

    \begin{proposition}[{\cite{kingsloefflerzerbes15a} Proposition~4.4.5}]
     \label{prop:deligne-eisenstein}
     The class
     \[ \Eis^k_{\cH,N} \in H^1_\cH\left(Y_1(N)_\RR, \TSym^k\sH_\RR(1)\right)\]
     is given by $(\alpha_{\infty},\alpha_{\dR})$ where
     \[
     \alpha_{\infty}\coloneqq
     \frac{-N^k}{2}
     \sum_{j=0}^k(-1)^{j}(k-j)!(2\pi i)^{j-k}(\tau-\overline{\tau})^{j}E^{(2j-k)}_{1/N}(\tau, -j) w^{[k-j,j]}
     \]
     and
     \[
     \alpha_{\dR}\coloneqq
     N^k E^{(k + 2)}_{1/N}(\tau, -1-k)(-2\pi i)(\tau-\overline{\tau})^kw^{[0,k]}d\tau.
     \]
    \end{proposition}

   \subsubsection{The regulator}

    Let $\cF$ be a cuspidal Hilbert eigenform, as before, of weight $(k+2, k'+2)$ and level $\fN$. Since $\cF$ is cuspidal, the maximal quotient of the Betti cohomology space $H^i_B(Y_1(\fN)(\CC), \TSym^{[k, k']} \sH_L(\cA))$ on which the dual Hecke operators act via the Hecke eigenvalues of $\cF$ is zero for $i \ne 2$, while for $i = 2$ it is $M_B(\cF)^*$, and similarly for the de Rham cohomology. Hence the Leray spectral sequence of absolute Hodge cohomology gives a projection map
    \[
     \operatorname{AJ}_{\cH, \cF}: H^3_{\cH}\left(Y_1^*(\fN)_{\RR}, \TSym^{[k, k']}(\sH(\cA)_\RR)(2-j)\right) \to
     H^1_{\cH}\left(\Spec \RR, M_{\cH}(\cF)^*(-j)\right),
    \]
    the \emph{Abel--Jacobi map} for $M_\cH(\cF)(1 + j)$.

    \begin{notation}
     For $0 \le j \le \min(k, k')$, we write
     \[
      {}^\vee CG^{[k, k', j]}:
       \iota^* (\Sym^{(k,k')}\sH_\CC^\vee) \to \Sym^{k+k'-2j}\sH_\CC^\vee
     \]
     for the dual of the Clebsch--Gordan map.
    \end{notation}

    \begin{proposition}\label{prop:integral-formula}
     If $[\omega] \in H^0_{\cH}(\Spec \RR, M_{\cH}(\cF)(1 + j))$ is the class of the $C^\infty$ differential form $\omega$, we have the formula
     \[
      \left\langle \mathrm{AJ}_{\cH, \cF}\left(\AF^{[k, k', j]}_{\cH, \fN}\right),
      [\omega] \right\rangle_{Y_1(\fN)} =
      \frac{1}{2\pi i}\int_{Y_{1, \QQ}(N)} \left( {}^\vee CG^{[k, k', j]}_B \circ\iota^*\right)(\omega)\wedge\alpha_\infty,
     \]
     where $\alpha_\infty$ is the differential form in Proposition \ref{prop:deligne-eisenstein}.
    \end{proposition}

    \begin{proof}
     See \cite[Proposition 6.2.2]{kingsloefflerzerbes15a}.
    \end{proof}

    \begin{proposition}
     \label{prop:deligne-regulator-formula}
     Let $0\le j\le\min\{k,k'\}$, and let $(\alpha_\infty, \alpha_\dR)$ be the representative for the Hodge Eisenstein class described in Proposition \ref{prop:deligne-eisenstein}.

     Then we have
     \begin{multline*}
      \left( {}^\vee CG^{[k, k', j]}_B \circ\iota^*\right)(\varpi_{\cF, 1})\wedge\alpha_\infty
      = \\
      \frac{(-1)^{j} N^{k+k'-2j}}{2} \binom{k}{j} k'!(2\pi i)^{k+2}(\tau - \bar\tau)^k
      \iota^*(\cF^{\ah, 1})(\tau) E^{(k-k')}_{1/N}(\tau, j - k) \d\tau \d\bar \tau
     \end{multline*}
     and similarly
     \begin{multline*}
      \left( {}^\vee CG^{[k, k', j]}_B \circ\iota^*\right)(\varpi_{\cF, 2})\wedge\alpha_\infty
      = \\
      \frac{(-1)^{k + k' + 1} N^{k+k'-2j}}{2} \binom{k'}{j} k!(2\pi i)^{k'+2}(\tau - \bar\tau)^{k'}
      \iota^*(\cF^{\ah, 2})(\tau) E^{(k'-k)}_{1/N}(\tau, j - k') \d\tau \d\bar \tau.
     \end{multline*}
    \end{proposition}

    \begin{proof}
     See the proof of \cite[Proposition 6.2.8]{kingsloefflerzerbes15a}.
    \end{proof}

    \begin{theorem}
     \label{thm:deligne-regulator-formulae}
     If $\varpi_{\cF}$ is the differential $\varpi_{\cF, 1} + (-1)^{j + 1} \varpi_{\cF, 2}$, then we have
     \begin{multline*}
      \left\langle \mathrm{AJ}_{\cH, \cF}\left(\AF^{[k, k', j]}_{\cH, \fN}\right),
      [\varpi_{\cF}] \right\rangle_{Y_1(\fN)} \\
      = (-1)^{k'-j}(2\pi i)^{k+k'-2j}D^{\frac{j+1}{2}}
      \frac{k!k'!}{(k-j)!(k'-j)!}\left. \frac{\d}{\d s}L^{\imp}_{\Asai}(\cF,s)\right|_{s = 1 + j}.
     \end{multline*}
    \end{theorem}

    \begin{proof}
     By Proposition \ref{prop:deligne-regulator-formula}, we have
     \begin{multline*}
      \frac{1}{2\pi i}\int_{Y_{1, \QQ}(N)(\CC)}\varpi_{\cF, 1}\wedge \alpha_\infty =
      (2\pi i)^{-1}\frac{(-1)^{j} N^{k+k'-2j}}{2} \binom{k}{j} k'!(2\pi i)^{k+2} \\
      \times \int (\tau - \bar\tau)^k \iota^*(\cF^{\ah, 1})(\tau)
      E^{(k-k')}_{1/N}(\tau, j - k) \d\tau \d\bar \tau.
     \end{multline*}
     By Remark \ref{rmk:impderiv}, the right-hand side is equal to
     \begin{multline*}
      (2\pi i)^{-1}\frac{(-1)^{j+1} N^{k+k'-2j}}{2} \binom{k}{j} k'!(2\pi i)^{k+2}(2i)^{k+1}\\
      \times   \frac{(-1)^{k'-j}D^{\frac{j+1}{2}}\Gamma(j+1)}{ (k'-j)!N^{k +k'-2j}\, 2^{k-k' + 2j+2}\, (-i)^{k-k'} \pi^{2j+1 - k' }}\  \frac{d}{ds} L^{\imp}_{\Asai}(\cF, s)|_{s=1+j}\\
      = \frac{(-1)^{k'-j}k!k'!(2\pi i)^{k+k'-2j}D^{\frac{j+1}{2}}}{2(k-j)!(k'-j)!}\times \frac{d}{ds} L^{\imp}_{\Asai}(\cF, s)|_{s=1+j}.
     \end{multline*}
     Since $\alpha_\infty$ is real-valued, and $\varpi_{\cF, 2}$ is the complex conjugate of $\varpi_{\cF, 1}$, we must have
     \[ \frac{1}{2\pi i}\int_{Y_{1, \QQ}(N)(\CC)}(\varpi_{\cF, 1} + (-1)^{j + 1} \varpi_{\cF, 2}) \wedge \alpha_\infty = 2 \times \frac{1}{2\pi i}\int_{Y_{1, \QQ}(N)(\CC)}\varpi_{\cF, 1}\wedge \alpha_\infty, \]
     so we obtain the stated formula.
    \end{proof}

    \begin{corollary}
     If the Asai $L$-value $L^{\imp, \prime}_{\Asai}(\cF, 1 + j)$ is non-zero, then the projection of $\AF^{[k, k', j]}_{\mot, \fN}$ to the $\cF$-isotypical component of $H^3_{\mot}(Y_1(\fN), \sH_L^{[\mu]}(2))$ is non-trivial.
    \end{corollary}

    \begin{proof}
     Clear, since the Hodge Asai--Flach class is defined as the image of the motivic Asai--Flach class under the Hodge realisation map, and we have just shown that this Hodge Asai--Flach class is non-trivial.
    \end{proof}

   \subsection{Injectivity of regulators}

    The following conjecture is due to Bloch and Kato, and (independently) Jannsen:

    \begin{conjecture}[{\cite[Conjecture 5.3(i)]{blochkato90}}]
     \label{conj:beilinson}
     Let $X$ be a smooth proper $\QQ$-variety. Then, for any prime $p$ and integers $m, n$ with $m \ne 2n$, the \'etale realisation map gives an isomorphism
     \[
      H^{m}_{\mot}(X, \QQ(n)) \otimes \Qp \to H^1_{\mathrm{g}}\left(\QQ, H^{m-1}_{\et}(X_{\QQbar}, \Qp(n))\right),
     \]
     where $H^1_{\mathrm{g}}(\QQ, -)$ is a subspace of $H^1(\QQ,-)$ (defined by local conditions as in Definition 5.1 of \emph{op.cit.}).
    \end{conjecture}

    Our elements are defined using non-proper varieties, but for certain weights we can lift them to proper ones using work of Wildeshaus, as follows. Suppose $(k, k', j)$ are integers such that the following conditions hold:
    \begin{itemize}
     \item $k, k' \ge 1$,
     \item $k = k' \bmod 2$,
     \item $0 \le j \le \min(k, k')$,
     \item either $k \ne k'$ or $j > 0$.
    \end{itemize}

    Let $\cF$ be an eigenform of level $\fN$ and weight $(k+2, k'+2, t, t')$, for some appropriate $t, t'$, where $\fN$ is such that $U_1^*(\fN)$ is sufficiently small. For simplicity, we suppose that $\cF$ is new of level $\fN$, and that the narrow class number of $F$ is 1. Let $L$ be the coefficient field of $\cF$.

    Choose a smooth compactification $\widetilde{\cA}^r$ of $\cA^r$, where $r = k + k'$. (These exist, and can be constructed using the theory of toroidal compactifications of mixed Shimura varieties; cf.~\cite[\S 4]{wildeshaus12}.)

    \begin{proposition}
     \label{prop:nonvanish}
     With the above notations, suppose that Conjecture \ref{conj:beilinson} holds with $X = \widetilde{\cA}^r$, $m = 3+r$, and $n = 2 + r - j$, for some prime $p$; and suppose also
     \[ \left. \tfrac{\d}{\d s} L^{\imp}_{\Asai}(\cF, s) \right|_{s = 1 + j} \ne 0.\]
     Then the \'etale Asai--Flach class is non-zero in $H^1(\QQ, M_{L_v}(\cF)^*(-j))$, for each finite place $v \mid p$ of $L$.
    \end{proposition}

    \begin{proof}
     The inclusion $\cA^r \into \widetilde{\cA}^r$ induces a pullback map
     \[ H^{3 +r}_{\mot}(\widetilde{\cA}^{r}, L(2+r-j)) \to H^{3 +r}_{\mot}(\cA^{r}, L(2+r-j)). \]
     As in Remark \ref{rmk:powerofA}, the group $H^{3}_{\mot}(Y_1^*(\fN), \TSym^{[k, k']} \sH_L(\cA)(2))$ can be regarded as a direct summand of $H^{3 +r}_{\mot}(\cA^{r}, L(2+r-j))$. It is shown in \cite[Corollaries 3.13 and 3.14]{wildeshaus12} that under the above conditions on $(k, k', j)$, this direct summand lifts canonically to a direct summand of $H^{3 +r}_{\mot}(\widetilde{\cA}^{r}, L(2+r-j))$, for any choice of the smooth compactification $\widetilde{\cA}^{r}$. Since this lifting arises from a direct sum decomposition of motives, it is compatible under the \'etale regulator with an analogous lifting in \'etale cohomology. In particular, Wildeshaus' results, together with the case of Conjecture \ref{conj:beilinson} that we have assumed, imply the injectivity of the map
     \begin{multline*}
      H^3_{\mot}(Y_1^*(\fN), \TSym^{[k, k']} \sH_L(\cA)(2-j)) \otimes_L L_v \\ \rTo H^1\left(\QQ, H^2_{\et}(Y_1^*(\fN), \TSym^{[k, k']} \sH_{L_v}(\cA)(2-j))\right).
     \end{multline*}

     Since $\cF$ is new and the narrow class number is 1, we may find a Hecke correspondence $T_\cF$ acting on $Y_1^*(\fN)$ which acts as the identity on the Hecke eigenspace corresponding to $\cF$, and as 0 on all other Hecke eigenspaces at level $\fN$. We consider the motivic cohomology class $T_\cF \cdot \AF^{[k, k', j]}_{\mot, \fN}$. By the computations of the previous section, if the $L$-value does not vanish, we have $T_\cF \cdot \AF^{[k, k', j]}_{\mot, \fN} \ne 0$; hence $T_\cF \cdot \AF^{[k, k', j]}_{\et, \fN}$ is also non-zero under our present assumptions. However, this class projects to 0 in all Hecke eigenspaces other than the $\cF$-eigenspace, so if it is non-zero, it must map to a non-zero element in $H^1(\QQ, M_{L_v}(\cF)^*(-j))$.
    \end{proof}

    This proposition, combined with Proposition \ref{prop:Lnonvanish} which gives a sufficient condition for the non-vanishing of $\tfrac{\d}{\d s} L^{\imp}_{\Asai}(\cF, s)$, proves Theorem \ref{lthm:nonvanish} of the introduction.

    \begin{remark}
     It seems reasonable to expect that the Asai--Flach elements still lift naturally to a compactification, even for the small weights not covered by Wildeshaus' results; compare \cite[\S 8-9]{brunaultchida16} in the Beilinson--Flach case.

     In the base case $k=k'=j=0$, the Asai--Flach elements lie in the group $H^3_{\mot}(Y_1^*(\fN), L(2)) = H^3_{\mot}(Y_1^*(\fN), \ZZ(2)) \otimes_{\ZZ} L$. A theorem of Suslin \cite[\S 4]{suslin87} shows that the \'etale realisation gives an injective map
     \[ H^3_{\mot}\left(Y_1^*(\fN), \ZZ(2)\right) \otimes \ZZ / p^r \ZZ \rInto H^3_{\et}(Y_1^*(\fN), \ZZ/p^r\ZZ(2)), \]
     for any $r \ge 1$. However, since we do not know if $H^3_{\mot}(Y_1^*(\fN), \ZZ(2))$ contains $p$-divisible elements, this is not enough to conclude that the \'etale Asai--Flach elements are non-zero.
    \end{remark}

 \section{Asai--Iwasawa classes}

  \subsection{Integral coefficient sheaves}

   As noted in Remark \ref{rmk:denoms} above, the theory of relative motives only works well if we take the coefficient ring to be a $\QQ$-algebra; but the theory of \'etale sheaves has no such restriction.

   \begin{definition}\label{def:HRA}
    Let $p$ be an odd prime, $S$ a regular scheme on which $p$ is invertible, and $\cA / S$ an abelian variety; and let $L$ be a finite extension of $\Qp$ with ring of integers $R$. We define a lisse \'etale sheaf of $R$-modules on $S$ by
    \[
     \sH_{R}(\cA) \coloneqq R \otimes_{\Zp} (R^1(\pi_\cA)_* \Zp)^\vee,
    \]
    where $\pi_{\cA}: \cA \to S$ is the structure map.
   \end{definition}

   If we are given an action of $\cO_F$ on $\cA$ by endomorphisms -- as in the case of the abelian varieties $A(U^*)$ over $Y^*(U^*)$ -- then the sheaf $\sH_{R}(\cA)$ is in fact a sheaf of $R \otimes_{\ZZ} \cO_F$-modules, with the $\cO_F$-module structure given by pushforward via the endomorphism action.

   We now suppose that $p$ is unramified in $F$, and that $F$ embeds into the coefficient field $L$. Then $R \otimes_{\ZZ} \cO_F \cong R \oplus R$ via the embeddings $\theta_1, \theta_2$, and we obtain a direct sum decomposition
   \[ \sH_R(\cA) = \sH_R(\cA)^{(1)} \oplus \sH_R(\cA)^{(2)} \]
   where (as before) $\sH_R(\cA)^{(i)}$ denotes the subspace where pushforward by $[x]$, for $x \in \cO_F$, acts as multiplication by $\theta_i(x)$.

   \begin{definition}
    We define
    \[ \TSym^{[k, k']} \sH_R(\cA) \coloneqq \TSym^k \sH_R(\cA)^{(1)} \otimes_R \TSym^{k'} \sH_R(\cA)^{(2)}.\]
   \end{definition}

   Note that after inverting $p$ this becomes isomorphic to the \'etale realisation of the relative Chow motive $\TSym^{[k, k']} \sH_L(\cA)$ defined above.

  \subsection{Lambda-adic sheaves}

   We shall now define sheaves of Iwasawa modules, and maps between them, which are ``$\Lambda$-adic interpolations'' of the \'etale realisations of the relative Chow motives defined in the previous section. This construction is the analogue in the Asai setting of the constructions of \S 5.1 of \cite{kingsloefflerzerbes15b}.

   \begin{definition}
    \label{def:lambdasheaf}
    For $\cA / S$ as in Definition \ref{def:HRA}, and $t: S \to \cA$ a section, we define
    \[ \Lambda_{R}(\cA\langle t \rangle) = \varprojlim_r t^* [p^r]_* (\ZZ / p^r \ZZ) \]
    as an inverse system of lisse \'etale sheaves on $S$, where $[p^r]: \cA \to \cA$ is the $p^r$-multiplication map. If $t = 0$ we write simply $\Lambda_R(\cA)$.
   \end{definition}

   As in \cite[\S 2.3]{kings13}, the sheaf $\Lambda_R(\cA)$ may be interpreted as the sheaf of Iwasawa algebras (with $R$ coefficients) associated to the sheaf of abelian groups $\sH_{\Zp}(\cA)$. It has the following universal property: any map of sheaves of profinite sets $\sH_{\Zp}(\cA) \to \cF$, where $\cF$ is a sheaf of $R$-modules, extends uniquely to a morphism of sheaves of $R$-modules $\Lambda_R(\cA) \to \cF$.

   In particular, if $\cA = A(U^*)$ for some $U^* \subseteq G^*(\hat\ZZ)$, or if $\cE = A_\QQ(U_\QQ)$ for $U_\QQ \subset \GL_2(\hat\ZZ)$ (with $U^*$, resp. $U_\QQ$, being sufficiently small), then we have \emph{moment maps} (\cite[\S\S 2.5.2, 2.5.2]{kings13}, see also \cite[Proposition 4.4.1]{kingsloefflerzerbes15b})
   \begin{gather*}
    \mom^k:  \Lambda_R(\cE) \to \TSym^{k} \sH_R(\cE),\\
    \mom^{[k, k']}: \Lambda_R(\cA) \to \TSym^{[k, k']} \sH_R(\cA),
   \end{gather*}
   for any $k \ge 0$, resp. any $k, k' \ge 0$.

   If $U_\QQ = U^* \cap \GL_2(\AA_f)$, then there is a natural map of sheaves of sets $\sH_{\Zp}(\cE) \to \iota^*\left( \sH_{\Zp}(\cA)\right)$ on $Y_\QQ(U_\QQ)$, so we obtain a map $\Lambda_R(\cE) \to \iota^* \Lambda_R(\cA)$. It is easy to see from the definitions that the maps just defined fit into the following commutative diagram:
   \[
   \begin{diagram}
   \Lambda_R(\cE) & \rTo^{\mom^{k + k'}} & \TSym^{k + k'} \sH_R(\cE) & \rInto & \TSym^{k} \sH_R(\cE) \otimes_{R} \TSym^{k'} \sH_R(\cE)\\
   \dTo & && & \dTo \\
   \iota^* \Lambda_R(\cA) &&\rTo^{\mom^{[k, k']}}&&  \iota^* \TSym^{[k, k']} \sH_{R}(\cA).
   \end{diagram}
   \]

  \subsection{Cyclotomic twists}

   We now extend the above construction to include a Tate twist. For $j \in \ZZ_{\ge 0}$, we define
   \[
    \Lambda_R(\cA)^{[j, j]} \coloneqq \Lambda_R(\cA) \otimes_{R} \TSym_R^{[j, j]} \sH_R(\cA).
   \]
   For integers $k, k' \ge j$ there is a map
   \[ \mom^{[k, k']}: \Lambda_R(\cA)^{[j, j]} \to \TSym^{[k, k']} \sH_{R}(\cA)\]
   defined as the composition
   \begin{multline*}
    \Lambda_R(\cA)^{[j, j]} \rTo^{\mom^{[k-j, k'-j]} \otimes 1} \TSym^{[k-j,k'-j]}_R \sH_R(\cA) \otimes_R  \TSym^{[j,j]}_R \sH_R(\cA) \\ \rTo \TSym^{[k,k']}_R \sH_R(\cA)
   \end{multline*}
   where the second map is given by the product in the symmetric tensor algebra. This is analogous to the definition of the moment map $\mom^k$ in \cite[\S 5.1]{kingsloefflerzerbes15b}.

   \begin{proposition}[{cf.~\cite[Proposition 5.1.2]{kingsloefflerzerbes15b}}]
    \label{prop:moment-compat}
    Let $j \ge 0$. There is a morphism of sheaves on $Y_\QQ(U_\QQ)$
    \[ CG^{[j]}: \Lambda_R(\cE) \to \iota^*\left(\Lambda_R(\cA)^{[j, j]}(-j)\right) \]
    such that for all integers $k, k', j$ with $0 \le j \le \min(k, k')$ we have a commutative diagram
    \[
     \begin{diagram}
      \Lambda_R\left(\cE\right) & \rTo^{\mom^{k + k' - 2j}} & \TSym^{k + k' - 2j}\left(\sH_{\cE}\right)\\
     \dTo^{CG^{[j]}} & & \dTo_{CG^{[k, k', j]}} \\
     \iota^* \Lambda_R(\cA)^{[j, j]}(-j)
      & \rTo^{\mom^{[k, k']}} & \iota^* \TSym^{[k, k']}\sH_{R}(\cA)(-j).
     \end{diagram}
    \]
   \end{proposition}

   Note that $CG^{[j]}$ satisfies the commutation relation
   \begin{equation}
    \label{eq:CGrelation}
    R'(n) \circ CG^{[j]} = n^{2j} CG^{[j]} \circ R'(n)
   \end{equation}
   for $n \in \Zp$, where $R'(n)$ is the operator corresponding to pushforward via the $n$-multiplication map on $\cA$ and $\cE$.


  \subsection{Asai--Iwasawa classes for Hilbert modular surfaces}

   We now define Asai--Iwasawa classes, which are cohomology classes taking values in the $\Lambda$-adic coefficient sheaves of Definition \ref{def:lambdasheaf}. Their function is to $p$-adically interpolate the Asai--Flach classes of the previous section as the parameters $k, k'$ vary.

   Recall that for integers $N \ge 4$, and $c>1$ with $(c, 6 p N) = 1$, we have the \emph{Eisenstein--Iwasawa class} \cite[\S 4]{kingsloefflerzerbes15b},
   \[\cEI_{N}\in H^1_{\et}\Big(Y_{1, \QQ}(N), \Lambda_{\Zp}(\cE\langle t_N \rangle)(1)\Big), \]
   where $t_N$ is the canonical order $N$ section. By applying the $N$-multiplication map $[N]: \Lambda_{\Zp}(\cE\langle t_N \rangle) \to \Lambda_{\Zp}(\cE)$, and base-extending to $R$, we may regard this class as having coefficients $\Lambda_R(\cE)(1)$, for any ring $R$ as above.

   \begin{remark}
    As with the motivic Eisenstein class considered above, the Eisen\-stein--Iwasawa class depends on a choice of $b \in \ZZ / N \ZZ - \{0\}$, and the class is denoted by $\cEI_{b,N}$ in \emph{op.cit.} to emphasise this dependence, but we fix $b = 1$ here and drop it from the notation.
   \end{remark}

   \begin{definition}
    Let $\fN \triangleleft \cO_F$ be such that $U_1^*(\fN)$ is sufficiently small, and let $N = \fN \cap \ZZ$ as usual. For integers $j \ge 0$ and $c > 1$ with $(c, 6pN) = 1$, we define the $j$-th \emph{Asai--Iwasawa class} by
    \[ \cAI_\fN^{[j]} \coloneqq \left(\iota_* \circ CG^{[j]}\right)\left( \cEI_N \right) \in H^3_{\et}\Big(Y^*_1(\fN), \Lambda_R(\cA)(2-j)\Big). \]
   \end{definition}

   Now let $M \ge 1$ be an integer. Via pullback along the natural map $Y^*(M, \fN) \to Y_1^*(\fN)$, we may regard $\cAI_\fN^{[j]}$ as a class in $H^3_{\et}\Big(Y^*(M, \fN), \Lambda_R(\cA)(2-j)\Big)$. If $M \mid \fN$, then the variety $Y^*(M, \fN)$ has an action of the operators $u_a = \stbt 1 a 0 1$ for $a \in \cO_F / M\cO_F$.

   \begin{definition}
    Let $M \ge 1$ and let $\fN \triangleleft \cO_F$ be divisible by $M$. For suitable $c$ as before, and $a \in \cO_F / M \cO_F$, we set
    \[ \cAI_{M, \fN, a}^{[j]} \coloneqq (u_a)_* \left(\cAI_{\fN}^{[j]}\right). \]
    (Thus $\cAI_\fN^{[j]} = \cAI_{1, \fN, 0}^{[j]}$.)
   \end{definition}

   Since the operator $u_a$ for $a \in \ZZ$ commutes with its namesake on $Y_\QQ(M, N)$, and the latter operator stabilises $\cEI_{N}$, we conclude that the class $\cAI_{M,\fN, a}^{[j]}$ actually only depends on the image of $a$ in the quotient $\frac{\cO_F}{M \cO_F + \ZZ}$.

   Finally, we make the following definition:

   \begin{definition}
    Let $M \ge 1$, and let $\fN \triangleleft \cO_F$ be such that $U_1^*(\fN)$ is sufficiently small (but we do \textbf{not} assume now that $M \mid \fN$). We define the \emph{$\Lambda$-adic Asai--Flach class}
    \[ \cAF_{M, \fN, a}^{[j]} \coloneqq (s_M)_* \left( \cAI_{M, M\fN, a}^{[j]}\right)
    \in H^3_{\et}\Big(Y^*_1(\fN) \times \mu_M^\circ, \Lambda_R(\cA)(2-j)\Big), \]
    where $s_M: Y^*(M, M\fN) \to Y_1^*(\fN) \times \mu_M^\circ$ is the ``twisted'' degeneracy map introduced in \S \ref{sect:cycloclasses} above.
   \end{definition}

   From Proposition \ref{prop:moment-compat} and the basic interpolating property of the Eisenstein class $\cEI_N$ \cite[Theorem 4.4.3]{kingsloefflerzerbes15b}, we have the following interpolation formula:

   \begin{theorem}[Interpolation in $k, k'$]
    For any integers $0 \le j \le \min(k, k')$, we have
    \[ \mom^{[k, k']}\left( \cAI_\fN^{[j]}\right) = \left(c^2 - c^{2j-k-k'}\langle c \rangle \right) \AF^{[k, k', j]}_{\et, \fN} \]
    and
    \[ \mom^{[k, k']}\left(\cAF_{M, \fN, a}^{[j]}\right) = \left(c^2 - c^{2j-k-k'}\langle c \rangle \sigma_c^{2}\right) \AF^{[k, k', j]}_{\et, M, \fN, a}, \]
where $\sigma_c$ is the Frobenius as defined in \S\ref{sect:frob}.
   \end{theorem}

   (Note that $2j - k - k' \le 0$, so the factor in brackets is always non-zero.) Thus, these classes interpolate the \'etale images of the motivic Eisenstein classes, for varying $k$ and $k'$ but a fixed $j$. We shall see in due course that these classes can also be interpolated $p$-adically as $j$ varies, but this will need some further preparation and we delay it until \S \ref{sect:cyclotwist} below.

   \begin{remark}
    \label{rmk:refinedelts}
    A slight refinement of the construction is also possible. The abelian variety $\cA / Y_1^*(\fN)$ has a canonical $\cO_F$-linear map $\frac{\fN^{-1}}{\cO_F} \into \cA[\fN]$ (the universal level $U_1^*(\fN)$-structure). In particular, if $n \ge 1$ is an integer dividing $\fN$, then the image of $1/n \pmod{\cO_F}$ defines a canonical section $t_n$ of $\cA$; and we may lift $\cAI^{[j]}_\fN$, resp.~$\cAF^{[j]}_{M, \fN, a}$, to classes with coefficients in $\Lambda_R(\cA\langle t_n \rangle)$, whose images under $[n]_*$ are the classes defined above. This refinement should be useful in studying variation in Hida families. We shall not pursue this here, however, in order to avoid adding yet more subscripts to our notation.
   \end{remark}

 \section{Norm relations}
  \label{sect:normrels}

  In this section, we prove some norm-compatibility relations for the $\Lambda$-adic Asai--Iwasawa and Asai--Flach classes defined in the previous section. We shall state these in pairs, consisting of a norm relation for the classes  $\cAI_{M, \fN, a}^{[j]}$ and another for the classes $\cAF_{M, \fN, a}^{[j]}$. In each case, it is the version for the $\cAI$ which we shall actually prove, but the version for the $\cAF$ which will be useful in applications; the only function of the ``$\cAI$ versions'' in our theory is as a stepping stone towards the ``$\cAF$ versions''. (This is exactly parallel to the roles played by the classes ${}_c \mathcal{RI}$ and ${}_c \mathcal{BF}$ in \cite{kingsloefflerzerbes15b}.)

  \subsection{Statement of the theorems}

   \subsubsection*{The first norm relation: changing \texorpdfstring{$\fN$}{N}} Our first two theorems deal with changing the level $\fN$. Compare \cite[Theorems 3.1.1 and 3.1.2]{leiloefflerzerbes14a}; \cite[Theorem 5.3.1]{kingsloefflerzerbes15a}.

   \refstepcounter{theorem}
   \begin{subtheorem}[Level-compatibility for $\cAI$]
    \label{thm:firstnormA}
    Let $M \ge 1$, $\fN$ an ideal divisible by $M$, $\fl$ a prime ideal of $\cO_F$, and $\ell \ge 1$ the rational prime lying below $\fl$. Then the image of $\cAI_{M, \fl\fN, a}^{[j]}$ under pushforward along the natural projection $\pr_{1, \fl}: Y^*(M, \fl\fN) \rightarrow Y^*(M, \fN)$ is given by
    \[
     \begin{cases}
      \cAI_{M,\fN,a}^{[j]}&\text{if $\ell \mid \Nm_{F/\QQ}(\fN)$,}\\
      (1-\ell^{-2j} \langle\ell^{-1}\rangle R'(\ell) \sigma_\ell^{-2})\cAI_{M, \fN, a}^{[j]}
      &\text{otherwise.}
    \end{cases}\]
   \end{subtheorem}

   \begin{subtheorem}[Level-compatibility for $\cAF$]
    \label{thm:firstnormB}
    Let $M \ge 1$, $\fN$ an ideal of $\cO_F$, $\fl$ a prime ideal of $\cO_F$, and $\ell \ge 1$ the rational prime lying below $\fl$. Then the image of $\cAF_{M, \fl\fN, a}^{[j]}$ under pushforward along the natural projection $\pr_{1, \fl}: Y_1^*(\fl\fN) \times \mu_M^\circ \rightarrow Y_1^*(\fN) \times \mu_M^\circ$ is given by
    \[
     \begin{cases}
      \cAF_{M, \fN, a}^{[j]}&\text{if $\ell \mid M \cdot \Nm_{F/\QQ}(\fN)$},\\
      (1-\ell^{-2j} \langle\ell^{-1}\rangle R'(\ell) \sigma_\ell^{-2})\cAF_{M, \fN, a}^{[j]}&\text{otherwise.}
     \end{cases}
    \]
   \end{subtheorem}

   We shall prove Theorem \ref{thm:firstnormA} in the next section. To deduce Theorem \ref{thm:firstnormB}, we simply apply Theorem \ref{thm:firstnormA} with $(M, \fN)$ replaced by $(M, M\fN)$, and note that the map $(s_M)_*$ commutes with the actions of the operators $R'(\ell)$, $\langle \ell \rangle$, and $\sigma_\ell$.

   \subsubsection*{The second norm relation: changing M (wild case)} The next pair of theorems deal with the significantly deeper question of changing $M$ (and thus the cyclotomic field over which the Asai--Flach elements are defined).

   For $a \in \ZZ_{\ge 1}$, let $\hat\pr_{2, a}$ denote the degeneracy map $Y^*(a M, \fN) \to Y^*(M, \fN)$ given by the matrix $\stbt{a^{-1}}{}{}1$.

   \refstepcounter{theorem}
   \begin{subtheorem}[Cyclotomic compatiblity for $\cAI$]
    \label{thm:UlrelationA}
    Let $M \ge 1$, let $\ell$ be prime, and let $\fN$ be an ideal of $\cO_F$ divisible by $\ell M$. Let $a \in \cO_F / (\ell M \cO_F + \ZZ)$, and suppose that $a$ is a unit at $\ell$ (i.e. the image of $a$ generates $\cO_F / (\ell \cO_F + \ZZ)$). Then
    \[
     \left(\hat\pr_{2, \ell}\right)_*\left(\cAI^{[j]}_{\ell M, \fN, a}\right)=
     \begin{cases}
      U'(\ell) \cdot\left(\cAI^{[j]}_{M, \fN, a}\right)&\text{if $\ell \mid M$,}\\
      \left(U'(\ell) - \ell^j\sigma_\ell\right)\cdot\left(\cAI^{[j]}_{M, \fN, a}\right)
      &\text{if $\ell \nmid M$.}
     \end{cases}
    \]
   \end{subtheorem}

   The corresponding statement for $\cAF$ is considerably simpler. The natural map $\mu_{\ell M}^\circ \to \mu_M^\circ$ corresponds to the inclusion $\QQ(\mu_M) \subset \QQ(\mu_{\ell M})$, so pushforward along this is simply the Galois norm map. Then we have the following relation:

   \begin{subtheorem}[Cyclotomic compatiblity for $\cAF$]
    \label{thm:UlrelationB}
    Let $M \ge 1$, let $\fN$ be an ideal of $\cO_F$, and let $\ell$ be a rational prime such that $\ell \mid \fN$. Suppose $a$ is a generator of $\cO_F / (\ell \cO_F + \ZZ)$. Then we have
    \[
     \operatorname{norm}_{\QQ(\mu_m)}^{\QQ(\mu_{\ell m})} \left(\cAF^{[j]}_{\ell M, \fN, a}\right) =
     \begin{cases}
      U'(\ell) \cdot\left(\cAF^{[j]}_{M, \fN, a}\right)&\text{if $\ell \mid M$,}\\
      \left(U'(\ell) - \ell^j\sigma_\ell\right)\cdot\left(\cAF^{[j]}_{M, \fN, a}\right)
      &\text{if $\ell \nmid M$.}
     \end{cases}
    \]
   \end{subtheorem}

   Just as before, Theorem \ref{thm:UlrelationB} follows readily from Theorem \ref{thm:UlrelationA}, but with the important caveat that we need to assume that $\ell \mid \fN$ in order for the Hecke operator $U'(\ell)$ to commute with $(s_M)_*$.

   \subsubsection*{The second norm relation: changing M (tame case)}

    We now come to the most intricate, and most important, of our norm-compatibility relations, where we introduce a new prime to $M$ which does \emph{not} divide $N$.

    \refstepcounter{theorem}
    \begin{subtheorem}
     \label{thm:TlrelationA}
     Let $M \ge 1$, $\fN \triangleleft \cO_F$ an ideal divisible by $M$, $\ell$ a prime which does not divide $\Nm_{F/\QQ}(\fN)$, and $a \in \cO_F / (\ell M \cO_F + \ZZ)$ which is a unit at $\ell$. Suppose also that one of the following holds:
     \begin{enumerate}[(i)]
      \item $\ell$ is inert in $F$;
      \item $\ell$ is split in $F$ and the primes $\fl, \bar\fl$ above $\ell$ are narrowly principal.
     \end{enumerate}
     Then pushforward via the composition $\pr_{1, \ell} \mathop{\circ} \hat\pr_{2, \ell}: Y^*(\ell M, \ell \fN) \to Y^*(M, \fN)$ maps $\cAI^{[j]}_{\ell M, \ell \fN, a}$ to the class
     \[  \ell^j \sigma_\ell\Bigg[ (\ell - 1)\left(1-\ell^{-2j} \langle\ell^{-1}\rangle R'(\ell) \sigma_\ell^{-2}\right) - \ell P_\ell'(\ell^{-1-j} \sigma_\ell^{-1}) \Bigg] \cdot \cAI^{[j]}_{M, \fN, a}, \]
     where $P_\ell'(X)$ is the operator-valued Asai Euler factor of Definition \ref{def:eulerfactor}.
    \end{subtheorem}

    We shall, in fact, only use case (i) of this theorem in the present paper (since primes inert in $F$ will suffice for our Euler system arguments). Hence, we shall not give full details of the proof of case (ii), although we include the statement (and a brief sketch of the proof) for the sake of completeness.

    The ``$\cAF$ version'' of this is the following, which is the fundamental Euler system norm relation for our $\Lambda$-adic classes:

    \begin{subtheorem}
     \label{thm:TlrelationB}
     Let $M \ge 1$, $\fN \triangleleft \cO_F$ an ideal, $\ell$ a prime which does not divide $M \cdot \Nm_{F/\QQ}(\fN)$, and $a \in \cO_F / (\ell M \cO_F + \ZZ)$ which is a unit at $\ell$. Suppose also that one of the following holds:
     \begin{enumerate}[(i)]
      \item $\ell$ is inert in $F$;
      \item $\ell$ is split in $F$ and the primes $\fl, \bar\fl$ above $\ell$ are narrowly principal.
     \end{enumerate}
     Then the Galois norm map $\operatorname{norm}_{\QQ(\mu_m)}^{\QQ(\mu_{\ell m})}$ maps $\cAF^{[j]}_{\ell M, \fN, a}$ to the class
     \[ \ell^j \sigma_\ell\Bigg[ (\ell - 1)\left(1-\ell^{-2j} \langle\ell^{-1}\rangle R'(\ell) \sigma_\ell^{-2}\right) - \ell P_\ell'(\ell^{-1-j} \sigma_\ell^{-1}) \Bigg] \cdot \cAF^{[j]}_{M, \fN, a}, \]
     where $P_\ell'(X)$ is the operator-valued Asai Euler factor of Definition \ref{def:eulerfactor}.
    \end{subtheorem}

    It is Theorem \ref{thm:TlrelationB} which will furnish us with Kolyvagin deriviative classes in order to bound Selmer groups. As will be clear by this stage, Theorem \ref{thm:TlrelationB} follows immediately from Theorem \ref{thm:TlrelationA}.

   \begin{remark}
    All of the above norm-compatibility relations also hold (with exactly the same proofs) for the refined elements mentioned in Remark \ref{rmk:refinedelts} above, as long as we restrict to values of $\fN$ divisible by the auxilliary integer $n$.
   \end{remark}

  \subsection{Proof of the first norm relation}

   \begin{proof}[Proof of Theorem \ref{thm:firstnormA}]
    Let $N$ be the positive integer generating $\fN \cap \ZZ$, so that $\ell \mid N$ if and only if $\ell \mid \Nm_{F/\QQ}(\fN)$. Recall from \cite[Theorem~4.3.2]{kingsloefflerzerbes15b} that if $\pr_{1, \ell}:Y_{\QQ}(M,\ell N)\rightarrow Y_{\QQ}(M,N)$ is the natural projection coming from the inclusion of congruence subgroups, then
    \[
     \pr_{1, \ell*} \left(\cEI_{\ell N}\right)=
     \begin{cases}
      \cEI_{N} &\text{if }\ell \mid N\\
      \left(1 - R'(\ell) \stbt {\ell^{-1}} 0 0 {\ell^{-1}}^*\right) \cEI_{N}&\text{otherwise.}
     \end{cases}
    \]
    Here $\stbt {\ell^{-1}} 0 0 {\ell^{-1}}$ is considered as an element of the upper-triangular Borel subgroup of $\GL_2(\hat \ZZ)$, which normalises $U_{\QQ}(M, N)$ and thus acts on $Y_\QQ(M, N)$.

    Now let $N'$ be the positive integer generating $\fl\fN \cap \ZZ$; we must have either $N' = \ell N$ or $N' = N$, and the latter case can only occur if $\ell \mid N$. We fix a lifting of $a$ to an element of $\cO_F / M\cO_F$, and consider the commutative diagram
    \[
     \begin{diagram}[small]
      Y_{\QQ}(M, N') &\rInto^{u_a \circ \iota} &  Y^*(M,\fl\fN)\\
      \dTo<{\pr_{1, ?}} & & \dTo>{\pr_{1, \fl}}\\
      Y_{\QQ}(M ,N) &\rInto^{u_a \circ \iota}& Y^*(M, \fN).
     \end{diagram}
    \]
    (The left vertical arrow is either $\pr_{1, \ell}$ or the identity, depending whether $N' = N\ell$ or $N' = N$.)

    By definition, we have $\cAI_{M, \fN, a} = (u_{a*} \circ \iota_* \circ CG^{[j]}) \left( \cEI_N\right)$, and similarly $\cAI_{M, \fl \fN, a} =  (u_{a*} \circ \iota_* \circ CG^{[j]})\left(\cEI_{N'}\right)$.

    Because the diagram is commutative, and the Clebsch--Gordan map $CG^{[j]}$ is compatible with the $\pr_1$ maps (since these act as the identity on each fibre of the abelian varieties), we have
    \begin{align*}
     \pr_{1, \fl*} \left(\cAI_{M, \fN, a}^{[j]}\right) &= \left(\pr_{1, \fl*} \mathop\circ u_{a*} \circ \iota_* \circ CG^{[j]}\right) \cEI_{1, N'}\\
     &= \left(u_{a*} \circ \iota_* \circ CG^{[j]} \circ \pr_{1,?*}\right) \cEI_{1, N'}.
    \end{align*}
    If $\ell \mid N$, then $N$ and $N'$ have the same prime factors, so this is simply
    \[ \left(u_{a*} \circ \iota_* \circ CG^{[j]}\right) \cEI_{1, N} = \cAI^{[j]}_{M, \fN, a}.\]
    In the case $\ell \nmid N$, we have $N' = \ell N$. The action of the centre of $\GL_2(\ZZ / N\ZZ)$ commutes with $\iota_*$, $u_{a*}$ and $CG^{[j]}$, and the same is true of $R'(\ell)$ up to a factor of $\ell^{-2j}$ arising from the Clebsch--Gordan map (cf.~equation \ref{eq:CGrelation} above). Since we have
    \[ \stbt {\ell^{-1}} 0 0 {\ell^{-1}}^* = \langle \ell^{-1} \rangle \sigma_\ell^{-2}\]
    as automorphisms of $Y^*(M, \fN)$, this gives the result.
   \end{proof}

  \subsection{Proof of the second norm relation}

   We shall now prove Theorem \ref{thm:UlrelationA}, following closely the arguments of \cite[\S 5.4]{kingsloefflerzerbes15b}.

   We fix $M$, $\fN$, $\ell$ and $a$ as in the statement of the theorem, and we fix a lifting of $a$ to an element of $\cO_F / \ell M \cO_F$. As usual, we let $N$ be such that $\fN \cap \ZZ = N\ZZ$. We write
   \[ \iota_{M, \fN, a}: Y_{\QQ}(M, N) \rInto Y^*(M, \fN) \]
   for the composition $u_a \circ \iota$, and similarly for $\iota_{\ell M, \fN, a}$. Furthermore, we define
   \[ \iota_{M(\ell), \fN, a}:Y_{\QQ}(\ell M, N) \rTo Y^*(M(\ell), \fN)\]
   to be the composite of $\iota_{\ell M,\fN,a}$ with the natural projection $Y^*(\ell M,\fN)\rightarrow Y^*(M(\ell),\fN)$.

   \begin{lemma}[{cf.~\cite[Lemma 5.4.1]{kingsloefflerzerbes15b}}]
    \label{lem:cartesian}
    The map $\iota_{M(\ell),\fN,a}$ is a closed embedding. If $\ell \mid M$, then the diagram
    \[
     \begin{diagram}[small]
      Y_{\QQ}(\ell M, N) & \rInto^{\iota_{M(\ell),\fN,a}} & Y^*(M(\ell),\fN)\\
      \dTo<{\hat\pr_{1, \ell}} & & \dTo>{\hat\pr_{1, (\ell)}}\\
      Y_{\QQ}(M,N) &\rInto^{\iota_{M,\fN,a}} & Y^*(M,\fN)
     \end{diagram}
    \]
    is Cartesian, where the vertical maps are the natural degeneracy maps.
   \end{lemma}

   \begin{proof}
    The image of $\iota_{M(\ell), \fN, a}$ is the modular curve of level
    \[ \GL_2(\AA_f) \cap u_a^{-1} U^*(M(\ell), \fN) u_a. \]
    An easy computation shows that this intersection is precisely those $\stbt r s t u \in \GL_2(\hat\ZZ)$ such that
    \[
     \tbt{r + at}{s + a(u-r)-a^2t}{t}{-at+u} = 1 \bmod \tbt{M}{\ell M}{\fN}{\fN}
    \]
    and since we are assuming $\ell M \mid \fN$, we conclude that $t = 0, u = 1 \bmod N$, $r = 1 \bmod M$, and $s + a(r-1) = 0 \bmod \ell M$, so that $a(r-1) = 0$ in $\cO_F/ (\ell M + \ZZ)$.

    Since we assumed that $a$ generates $\cO_F / (\ell M + \ZZ)$, it follows that $r-1$ is divisible by $\ell M$, and hence also that $s$ is divisible by $\ell M$. So the intersection is equal to $U_{\QQ}(\ell M, N)$, and $\iota_{M(\ell), \fN, a}$ is an isomorphism onto its image, as required.

    To obtain the Cartesian property in the case $\ell \mid M$ we simply compare degrees: both horizontal maps are injective, and both vertical maps are finite \'etale of degree $\ell^2$, so we are done.
   \end{proof}

   We now consider the omitted case $\ell \nmid M$. Let $\tilde{a}$ be the unique element of $\cO_F / \ell M$ such that $\tilde a = 0 \bmod \ell$ and $\tilde a = a \bmod M$. Then the following is easily verified:

   \begin{lemma}
    \label{lem:cartesian2}
    Let $\gamma:Y_\QQ(M(\ell),N)\rightarrow Y^*(M(\ell),\fN)$ be the diagonal map $\iota$ composed with the action of $\tbt 1 {\tilde{a}} 0 1$. Then, the following diagram is Cartesian:
    \[
     \begin{diagram}[small]
      Y_{\QQ}(\ell M,N) \sqcup Y_{\QQ}(M(\ell),N) & \rTo^{\left(\iota_{M(\ell),\fN,a},\gamma\right)} & Y^*(M(\ell),\fN)\\
      \dTo<{\left(\hat\pr_{1, \ell}, \hat\pr_{1, (\ell)}\right)} & & \dTo>{\hat\pr_{1, (\ell)}} \\
      Y_{\QQ}(M,N) & \rTo^{\iota_{M,\fN,a}} & Y^*(M,\fN).
     \end{diagram}
    \]\qed
   \end{lemma}

   With these ingredients in place, the proof of the Theorem \ref{thm:UlrelationA} proceeds exactly as in \cite[Theorem 5.4.4]{kingsloefflerzerbes15b}:

   \begin{proof}[Proof of Theorem \ref{thm:UlrelationA}]
    We first consider the case $\ell \mid M$. In this case, the cartesian diagram of Lemma \ref{lem:cartesian} shows that the pushforward of $\cAI^{[j]}_{\ell M, N, a}$ along the degeneracy map $Y^*(M \ell, \fN) \to Y^*(M(\ell), \fN)$ is equal to the pullback of $\cAI^{[j]}_{M, N, a}$ along the natural degeneracy map $Y^*(M(\ell), \fN) \to Y^*(M, \fN)$.

    Hence the image of $\cAI^{[j]}_{\ell M, N, a}$ under $(\hat\pr_{2,\ell})_*$ is equal to the image of $\cAI^{[j]}_{M, N, a}$ under the composition $(\hat\pr_{2, (\ell)})_* \circ (\hat\pr_{1, (\ell)})^*$, where
    \[ \hat\pr_{1, (\ell)}, \hat\pr_{2, (\ell)} : Y^*(M(\ell), \fN) \to Y^*(M, \fN)\]
    are respectively the natural degeneracy map and the ``twisted'' degeneracy map induced by $\stbt {1}{}{} \ell$. This composition $(\hat\pr_{2, (\ell)})_* \circ (\hat\pr_{1, (\ell)})^*$ is exactly the definition of the Hecke operator $U'(\ell)$.

    In the case $\ell \nmid M$, the same argument using Lemma \ref{lem:cartesian2} shows that
    \[ U'(\ell) \cdot \cAI^{[j]}_{M, N, a} = \hat\pr_{2, \ell*} \left(\cAI^{[j]}_{\ell M, N, a}\right) + (\hat\pr_{2, (\ell)*} \circ \gamma_* \circ CG^{[j]})(\cEI_N). \]
    One checks that there is a commutative diagram
    \[
     \begin{diagram}[small]
      Y_{\QQ}(M(\ell), N) & \rTo^{\gamma} & Y^*(M(\ell), \fN) \\
      \dTo<{\hat\pr_{2, (\ell)}} & & \dTo>{\hat\pr_{2, (\ell)}}\\
      Y_\QQ(M, N) & \rTo^{\iota_{M, N, \ell^{-1} a}} & Y^*(M, \fN).
     \end{diagram}
    \]
    and by \cite[Theorem 4.3.3]{kingsloefflerzerbes15a}, the left-hand $\hat\pr_{2, (\ell)}$ sends $\cEI_N$ to itself, but induces a factor of $\ell^j$ in the Clebsch--Gordan map (since it acts as an isogeny of degree $\ell$ on the elliptic curve $\cE$). This gives
    \begin{align*}
     (\hat\pr_{2, (\ell)*} \circ \gamma_* \circ CG^{[j]})(\cEI_N)
     &= (\iota_{M, N, \ell^{-1}a{}*} \circ \hat\pr_{2, (\ell)*} \circ CG^{[j]})(\cEI_N)\\
     &= \ell^j (\iota_{M, N, \ell^{-1}a{}*} \circ CG^{[j]} \circ \hat\pr_{2, (\ell)*})(\cEI_N)\\
     &= \ell^j (\iota_{M, N, \ell^{-1}a{}*} \circ CG^{[j]})(\cEI_N)\\
     &= \ell^j \cAI^{[j]}_{M, N, \ell^{-1}a} = \ell^j \sigma_\ell \cdot \cAI^{[j]}_{M, N, \ell^{-1}a}.\tag*{\qedhere}
    \end{align*}
   \end{proof}

  \subsection{Proof of Theorem \ref{thm:TlrelationA} (inert primes)}

   Rather than attack Theorem \ref{thm:TlrelationA} head-on, we shall attempt to deduce it from other simpler norm relations, using compatibilities in the Hecke algebra (the strategy introduced in \cite[Appendix]{leiloefflerzerbes14b}). We first introduce some notation.

   Let $M$ be an integer, and $\fN$ an ideal of $\cO_F$ divisible by $M$, as usual. Let $\fa$ be a prime ideal of $\cO_F$ with a totally positive generator $\alpha$. Then, as well as the obvious degeneracy map
   \[ \pr_{1, \alpha}: Y^*(M, \fa \fN) \to Y^*(M, N) \]
   whose effect on Asai--Iwasawa elements was studied in Theorem \ref{thm:firstnormA}, there is a second degeneracy map
   \[ \pr_{2, \alpha}: Y^*(M, \fa \fN) \to Y^*(M, N)\]
   given by $\tau \mapsto \alpha \tau$ on $\cH_F$.

   (The former map was denoted previously by $\pr_{1, \fa}$, since it is independent of $\alpha$ and makes sense whether or not $\fa$ is narrowly principal; but when a generator $\alpha$ exists, we use the alternative notation $\pr_{1,\alpha}$ for this map, for harmony with $\pr_{2, \alpha}$.)

   For $a \in \ZZ_{\ge 1}$ we also have maps $Y_\QQ(M, aN) \to Y_\QQ(M, N)$ defined similarly, which we denote by the same symbols $\pr_{1, a}$, $\pr_{2, a}$.

   \begin{proposition}
    Let $M \mid N$ be integers, and let $\ell$ be a prime. Then the Eisenstein--Iwasawa classes on $Y_\QQ(M, \ell N)$ and $Y_\QQ(M, N)$, considered with coefficients in $\Lambda_R(\cE)(1)$, satisfy the relation
    \[ (\pr_{2, \ell})_*\left(\cEI_{\ell N}\right) =
     \begin{cases}
      \ell R'(\ell) \cdot \cEI_{N} & \text{if $\ell \mid N$},\\
      \ell R'(\ell) \left(1 - R'(\ell) \stbt{\ell^{-1}}{}{}{\ell^{-1}}^*\right) \cdot \cEI_{N} & \text{if $\ell \nmid N$.}
     \end{cases}
    \]
   \end{proposition}

   \begin{proof}
    This is Corollary 4.3.6 of \cite{kingsloefflerzerbes15b}. (Note, however, that in \emph{op.cit.} the class $\cEI_N$ is considered to have coefficients in $\Lambda_{\Zp}(\cE\langle t_N \rangle)(1)$, so we must apply the map $[N]_*$ to obtain classes in $\Lambda_{\Zp}(\cE)(1)$; since we are comparing classes with two different values of $N$ this introduces a factor of $R'(\ell)$ which is not present in \emph{op.cit.}.)
   \end{proof}

   \begin{corollary}
    \label{cor:pr2}
    Let $M \ge 1$, $\fN \triangleleft \cO_F$ divisible by $M$, and $\ell$ prime. Then
    \[ (\pr_{2, \ell})_*\left( \cAI^{[j]}_{M, \ell \fN, a} \right) =
     \begin{cases}
      \ell^{1 - j}R'(\ell) \cdot \cAI^{[j]}_{M, \fN, \ell a} & \text{if $\ell \mid \fN$},\\
      \ell^{1 - j}R'(\ell) \sigma_\ell^{-1}\left(1-\ell^{-2j} \langle\ell^{-1}\rangle R'(\ell) \sigma_\ell^{-2}\right) \cdot \cAI^{[j]}_{M, \fN, a} & \text{if $\ell \nmid \fN$}.
     \end{cases}
    \]
   \end{corollary}

   \begin{proof}
    This follows from the previous proposition and commutativity of pushforward maps around the diagram
    \[
     \begin{diagram}
      Y_\QQ(M, \ell N) & \rTo^{u_a \circ \iota} & Y^*(M, \ell \fN)\\
      \dTo^{\pr_{2, \ell}} & & \dTo_{\pr_{2, \ell}}\\
      Y_\QQ(M, N) &\rTo^{u_{\ell a} \circ \iota} & Y^*(M, \fN).
     \end{diagram}
    \]
    (The $\ell^j$ factors appear because of the failure of the map $CG^{[j]}$ to commute with pushforward via isogenies, exactly as in the $\GL_2 \times \GL_2$ situation; cf~\cite[proof of Theorem 5.4.1]{kingsloefflerzerbes15b})
   \end{proof}

   \begin{proof}[Proof of Theorem \ref{thm:TlrelationA} for $\ell$ inert in $F$]
    We shall now prove Theorem \ref{thm:TlrelationA} in the inert case. We are interested in the image of $\cAI^{[j]}_{\ell M, \ell \fN, a}$ under the map
    \[ Y(\ell M, \ell \fN) \rTo^{\hat\pr_{2, \ell}} Y(M, \ell \fN) \rTo^{\pr_{1, \ell}} Y(M, \fN). \]
    By Theorem \ref{thm:UlrelationA} (applied with $(\ell, M, \fN)$ replaced by $(\ell, M, \ell \fN)$), we know that
    \[ \hat\pr_{2, \ell*}\left( \cAI^{[j]}_{\ell M, \ell \fN, a}\right) = (U'(\ell) - \ell^j \sigma_\ell) \cdot \cAI^{[j]}_{M, \ell \fN, a}.
    \]
    A double coset computation (using the fact that $\ell\cO_F$ is a prime ideal) shows that
    \begin{equation}
     \label{eq:pridentity}
     \pr_{1, \ell*} \mathop\circ U'(\ell) = T'(\ell) \circ \pr_{1, \ell*} - \langle \ell^{-1} \rangle \circ \pr_{2, \ell*}.
    \end{equation}
    Hence we have
    \begin{align*}
     (\pr_{1, \ell} \circ \hat\pr_{2, \ell})_* \left(\cAI^{[j]}_{\ell M, \ell \fN, a}\right)
     &= \left[\pr_{1, \ell*} \circ (U'(\ell) - \ell^j \sigma_\ell) \right] \cAI^{[j]}_{M, \ell \fN, a} \\
     &= \left[ \left( T'(\ell) - \ell^j \sigma_\ell\right) \pr_{1, \ell*} - \langle \ell^{-1} \rangle \pr_{2,\ell*}\right] \cAI^{[j]}_{M, \ell \fN, a}.
    \end{align*}
    Substituting the formulae for $\pr_{1, \ell*}\left(\cAI^{[j]}_{M, \ell \fN, a}\right)$ from Theorem \ref{thm:firstnormA}, and for $\pr_{2,\ell*}\left(\cAI^{[j]}_{M, \ell \fN, a}\right)$ from Corollary \ref{cor:pr2}, and rearranging, we obtain the theorem.
   \end{proof}

  \subsection{The case of split primes (sketch)}

   For completeness, we sketch the proof of case (ii) of Theorem \ref{thm:TlrelationA}, in which $\ell$ is split in $F$ and the primes $\fl, \bar\fl$ above $\ell$ are narrowly principal. Thus, there is a totally positive element $\lambda$ such that $\fl = (\lambda)$, $\bar\fl = (\bar\lambda)$, and $\lambda \bar\lambda = \ell$. We fix, for the duration of this section, a choice of such a $\lambda$.

   \begin{theorem}
    \label{thm:thirdnorm}
    For any $a \in \cO_F / (M\cO_F + \ZZ)$, the following relation holds, modulo $p$-torsion if $\ell = p$:
    \begin{multline*}
     \pr_{2, \lambda*}\left( \cAI^{[j]}_{M, \fl\fN, a} \right) =\\
     \begin{cases}
      \ell^{-j}R'(\lambda) \cdot U'(\bar\lambda) \cdot
      \cAI^{[j]}_{M, \fN, \ell a} & \text{if $\bar\fl \mid \fN$,}\\
      \sigma_\ell^{-1}\ell^{-j}R'(\lambda) \left( T'(\bar\lambda) -
      \sigma_\ell^{-1} \cdot \ell^{-j}R'(\bar\lambda) \cdot \langle \bar\lambda^{-1} \rangle \cdot U'(\lambda) \right)
      \cAI^{[j]}_{M, \fN, a} & \text{if $\bar\fl \nmid \fN$ but $\fl \mid \fN$,}\\
      \sigma_\ell^{-1}\ell^{-j}R'(\lambda) \left( T'(\bar\lambda)
      - \sigma_\ell^{-1} \cdot \ell^{-j}R'(\bar\lambda) \cdot \langle \bar\lambda^{-1} \rangle \cdot T'(\lambda) \right) \cdot \cAI^{[j]}_{M, \fN, a}& \text{if $\fl, \bar \fl \nmid \fN$.}
     \end{cases}
    \end{multline*}

   \end{theorem}

   \begin{proof}
    This is virtually identical to the proof of Theorem 5.5.1 of \cite{kingsloefflerzerbes15b} (which is the ``degenerate case $F = \QQ \oplus \QQ$'').
   \end{proof}

   \begin{remark}
    Note that $\lambda$ is only well-defined up to multiplication by $\cO_F^{\times +}$. However, the validity of the theorem is independent of the choice of $\lambda$, since replacing $\lambda$ with $\alpha \lambda$ for $\alpha \in \cO_F^{\times +}$ has the effect of acting on both sides by the operator $\stbt 10 0 \alpha$.
   \end{remark}

   \begin{proof}[Proof of Theorem \ref{thm:TlrelationA} for $\ell$ split in $F$]
    As in the inert case, we need to compute
    \[ \left[\pr_{1, \ell*} \circ \left( U'(\ell) - \ell^j \sigma_\ell\right)\right] \cAI^{[j]}_{\fM, \ell \fN, a}.\]
    We factor $\pr_{1, \ell}$ as the composite $\pr_{1, \bar\lambda} \circ \pr_{1, \lambda}$, and similarly $U'(\ell) = U'(\lambda) U'(\bar\lambda)$. Using the analogues of \eqref{eq:pridentity} with $\lambda$ and $\bar\lambda$ in place of $\ell$, we obtain
    \begin{multline*} \pr_{1, \ell*} \mathop\circ \left( U'(\ell) - \ell^j \sigma_\ell\right) = (T'(\ell) - \ell^j \sigma_\ell) \pr_{1, \ell*}\\ - \langle \lambda^{-1} \rangle T'(\bar\lambda) \pr_{2, \lambda*} \pr_{1, \bar\lambda*} - \langle \bar\lambda^{-1} \rangle T'(\lambda) \pr_{1, \lambda*} \pr_{2, \bar\lambda*} + \langle \ell^{-1} \rangle \pr_{2, \ell*}.
    \end{multline*}
    The effect of each of these four degeneracy maps $Y^*(M, \ell \fN) \to Y^*(M, \fN)$ on the Asai--Iwasawa element has been calculated above: $\pr_{1, \ell}$ by Theorem \ref{thm:firstnormA}, $\pr_{2, \ell}$ by Corollary \ref{cor:pr2}, and the cross terms  $\pr_{2, \lambda} \circ \pr_{1, \bar\lambda}$ and $ \pr_{1, \lambda} \circ \pr_{2, \bar\lambda}$ by Theorem \ref{thm:thirdnorm}. Combining all of these ingredients and rearranging gives the theorem.
   \end{proof}

   \begin{remark}
    Since the results of this paper were initially announced, a strengthened form of Theorem \ref{thm:TlrelationA} has been proved by Giada Grossi (in preparation); this shows that the assumption that the primes above $\ell$ are narrowly principal when $\ell$ is split is not needed, and the assertion in fact holds for any prime $\ell \nmid \Nm_{F/\QQ}(\fN)$ unramified in $F$.
   \end{remark}


  \section{Cyclotomic twists}
   \label{sect:cyclotwist}


   Our next goal is to show that the Asai--Flach elements can be interpolated as the twist parameter $j$ varies. Recall that we defined
   \[ \cAF^{[j]}_{m, \fN, a} \coloneqq (s_m)_* \cAI^{[j]}_{m, m\fN, a} \]
   where the map
   \[
    s_m: Y^*(m, m\fN)\rightarrow Y_1^*(\fN)\times \mu_{m}^\circ,
   \]
   is given by the action of $\stbt {m^{-1}} 0 0 1$.


 \subsection{Compatibility with cyclotomic twists}

   We now set $M = p^r$. It is clear that $s_{p^r}$ induces a map on the torsion sheaves $\sH_{\cA,r} = \sH_R(\cA) / p^r \sH_R(\cA)$.

   \begin{notation}
    For $\star\in\{\cE,\cA\}$, write from now on $\sH_{\star,r}^k = \TSym^k \sH_{\star,r}$ and $\Lambda_{\star,r}=\Lambda_r(\sH_{\star,r})$. Write $x_\star$ and $y_\star$ for the order $p^r$ sections of $\sH_{\star,r}$ over $Y_{\QQ}(p^r, p^r N)$ (if $\star=\cE$), resp. over $Y^*(p^r,p^r\fN)$ (if $\star=\cA$).
   \end{notation}

   \begin{remark}
    The sections $\kappa^*(x_{\cA})$ and $\kappa^*(y_{\cA})$ agree with the images of $x_{\cE}$ and $y_{\cE}$ under the map
    \[ H^0(Y_{\QQ}(p^r,p^rN),\sH_{\cE,r})\rTo H^0(Y_{\QQ}(p^r,p^r N),\kappa^*\sH_{\cA,r})\]
    induced by the injection $\sH_{\cE,r} \into \kappa^*\sH_{\cA,r}$.
   \end{remark}

   \begin{remark}
    On $H^0(Y^*(p^r,p^r\fN),\TSym^{[k,k']}\sH_{\cA,r})$, the map $(u_a)_* = (u_{-a})^*$ sends $x_{\cA}^{[i]}y_{\cA}^{[k-i]}\otimes x_{\cA}^{[l]}y_{\cA}^{[k'-l]}$ to $x_{\cA}^{[i]}y_{\cA}^{[k-i]}\otimes (x_{\cA}-ay_{\cA})^{[l]}y_{\cA}^{[k'-l]}$.
   \end{remark}

   \begin{theorem}\label{cyctwistdiag}
    The following diagram commutes:
    \begin{diagram}
     H^1_{\et}\big(Y(p^r, p^r N), \Lambda_{\cE,r}^{[0,0]}(1)\big)
     &  \rTo^{\cup y_{\cE}^{\otimes 2j}} & H^1_{\et}\big(Y(p^r,p^rN),\Lambda_{\cE,r}^{[j, j]}(1)\big)\\
     \dTo^{CG^{[j]}}  & & \dTo^{(a-\overline{a})^j j!}\\
     H^1_{\et}\big(Y(p^r,p^rN),\iota^*\Lambda_{\cA,r}^{[j,j]}(1 - j)\big) & & H^1_{\et}\big(Y(p^r,p^rN),\iota^*\Lambda_{\cA,r}^{[j, j]}(1)\big)\\
     \dTo^{\iota_*} & & \dTo^{\iota_*} \\
      H^3_{\et}\big(Y^*(p^r,p^r\fN),\Lambda_{\cA,r}^{[j,j]}(2-j)\big) & & H^3_{\et}\big(Y^*(p^r,p^r\fN),\Lambda_{\cA,r}^{[j, j]}(2)\big)\\
     \dTo^{(u_a)_*} & & \dTo^{(u_a)_*} \\
      H^3_{\et}\big(Y^*(p^r,p^r\fN),\Lambda_{\cA,r}^{[j, j]}(2-j)\big) & & H^3_{\et}\big(Y^*(p^r,p^r\fN),\Lambda_{\cA,r}^{[j, j]}(2)\big)\\
      \dTo^{(s_{p^r})_*} & & \dTo^{(s_{p^r})_*} \\
      H^3_{\et}\big(Y_1(\fN)\times\mu_{p^r}^\circ,\Lambda_{\cA,r}^{[j, j]}(2-j)\big) & \rTo^{\cup (\zeta_{p^r})^{\otimes j}}&H^3_{\et}\big(Y_1(\fN)\times\mu_{p^r}^\circ,\Lambda_{\cA,r}^{[j, j]}(2)\big).
    \end{diagram}
   \end{theorem}

   \begin{proof}
    The proof is basically identical to the proof of \cite[Theorem 6.2.4]{kingsloefflerzerbes15b}.
   \end{proof}

   \begin{corollary}\label{cor:imageundermom}
    We have
    \[ (a-\overline{a})^jj!(1\otimes \mom^{[j,j]})\left(\cAF^{[0]}_{p^r,\fN,a,r}\right)=\cAF^{[j]}_{p^r,\fN,a,r}\otimes \zeta_{p^r}^{\otimes j}.\hfill\qed\]
   \end{corollary}


 \subsection{Cyclotomic twists of Asai--Flach elements}

  Note that for any integers $0 \le j \le \min(k, k')$ we have maps
  \begin{align*}
   \id\otimes \mom^{[j,j]}_r:  \Lambda_r\otimes \Lambda_r & \rTo \Lambda_r^{[j,j]},\\
   \mom^{[k-j,k'-j]}_r\otimes\id: \Lambda_r^{[j,j]} & \rTo \left(\TSym^{[k-j,k'-j]}\sH_{\cA,r}\right) \otimes \left(\TSym^{[j,j]} \sH_{\cA,r}\right).
  \end{align*}
  We write $\mom_r^{[k-j,k'-j]} \cdot \id$ for the composition of $\mom^{[k-j,k'-j]}_r\otimes\id$ with the symmetrised tensor product map
  \[ \left(\TSym^{[k-j,k'-j]}\sH_{\cA,r}\right) \otimes \left(\TSym^{[j,j]} \sH_{\cA,r}\right)\rTo \TSym^{[k,k']}\sH_{\cA,r} .\]
  Denote by $u$ the natural morphism of sheaves $\Lambda_{\cA,r}\rightarrow \Lambda_{\cA,r}\otimes\Lambda_{\cA,r}$.

  \begin{lemma}
   \label{lem:momidentity}
    For all $0 \le j\leq k$, we have the following identity of moment maps:
   \[ (\mom^{[k-j,k'-j]}_r\cdot \id)\circ (1\otimes \mom^{[j,j]}_r)\circ u=\tbinom{k}{j}\tbinom{k'}{j}\mom^{[k,k']}_r.\]
  \end{lemma}
  \begin{proof}
   See \cite[Lemma 6.3.1]{kingsloefflerzerbes15b}.
  \end{proof}

  \begin{definition}
   \label{def:eord}
   Let $e_{\ord}' \coloneqq \lim_{n\rightarrow\infty} U'(p)^{n!}$ be the ordinary idempotent attached to $U'(p)$; and let $\Lambda_{\Gamma}(-\mathbf{j})$ be the Iwasawa algebra of $\Gal(\QQ(\mu_{p^\infty}) / \QQ)$, with $\Gal(\QQbar / \QQ)$ acting by the inverse of the canonical character $\mathbf{j}$, as in \cite[Notation 6.3.3]{kingsloefflerzerbes15b}.
  \end{definition}

  \begin{theorem}
   \label{thm:iwasawaelt}
   For any prime $p\geq 3$, $\fN$ an ideal with $p \mid \fN$, $m \ge 1$ an integer coprime to $p$, and $c > 1$ coprime to $6pm \Nm_{F/\QQ}(\fN)$, there is a class
   \[
    \cAF_{m,\fN,a} \in
    H^3_{\et}\Big(Y_1^*(\fN)\times\mu_{m}^\circ,
    \Lambda_R(\sH_{\cA}) \otimes \Lambda_{\Gamma}(2-\mathbf{j})\Big)
   \]
   such that
   \begin{align*}
    (\mom^{[k,k']}\otimes& \mom^{j}_{\Gamma})\left(\cAF_{m,\fN,a}\right)\\
     = &\left(1-p^j(U_p')^{-1}\sigma_p\right)\left(c^2-c^{-(k+k'-2j)}\langle c\rangle \sigma_c^2\right)\frac{e_{\ord}'\left(\cAF^{[j]}_{m,\fN,a}\right)}{(a-\overline{a})^jj!\tbinom{k}{j}\tbinom{k'}{j}}.
   \end{align*}

  \end{theorem}
  \begin{proof}
   Analogous to the proof of \cite[Theorem 6.3.3]{kingsloefflerzerbes15b}.
  \end{proof}

 \section{Iwasawa theory}

  \subsection{Iwasawa cohomology classes for ordinary eigenforms}
   \label{sect:iwa1}

   Let $\cF$ be an eigenform, with eigenvalues in some coefficient field $L$ and weight $(k + 2, k' + 2, t, t')$, with $k, k' \ge 0$. Let $p$ be a prime dividing the level $\fN$, and unramified in $F / \QQ$.

   \begin{definition}
    \label{def:ordinary}
    We say that $\cF$ is \emph{ordinary at $p$} (with respect to some choice of prime $v \mid p$ of $L$) if its eigenvalue $\alpha_p(\cF)$ for the Hecke operator $\mathcal{U}_0(p) \coloneqq p^{-(t+t')}\mathcal{U}(p)$ is a unit at $v$.
   \end{definition}

   The normalisation factor $p^{-(t + t')}$ corresponds to the difference between the Hecke operators $\mathcal{U}(p)$ on $G$ and $U(p)$ on $G^*$, cf.~Proposition \ref{prop:jpullback}.

   \begin{theorem}
    If $\cF$ is an ordinary eigenform, then for any $m \ge 1$ coprime to $p$, and any $c > 1$ coprime to $6pm\Nm_{F/\QQ}(\fN)$, there exists a class
    \[ \cAF^{\cF}_{M, a} \in H^1\left(\ZZ[\mu_m, \tfrac{1}{mp\Nm_{F/\QQ}(\fN)}], M_{L_v}(\cF)^* \otimes \Lambda_{\Gamma}(-\mathbf{j})\right)\]
    such that, for every $0 \le j \le \min(k, k')$ and $r \ge 0$, the image of $\cAF^{\cF}_{m, a}$ in $H^1(\ZZ[\mu_{mp^r}, \tfrac{1}{mp\Nm_{F/\QQ}(\fN)}], M_{L_v}(\cF)^*(-j))$ is given by
    \[
     \frac{(c^2 - c^{2j-k-k'}\varepsilon_\cF(c)\sigma_c^{2})}{(a - \bar{a})^j j! \binom{k}{j} \binom{k'}{j}} \cdot
     \begin{cases}
      \alpha_p(\cF)^{-r} \AF^{[\cF, j]}_{\et, mp^r, a} & \text{if $r \ge 1$,}\\
      \left(1 - \tfrac{p^j \sigma_p}{\alpha_p(\cF)} \right) \AF^{[\cF, j]}_{\et, m, a} &
      \text{if $r = 0$.}
     \end{cases}
    \]
   \end{theorem}

   \begin{proof}
    Since $\cF$ is ordinary, the projection map $\pr_{\cF}$ factors through the ordinary projector $e'_{\ord}$ of Definition \ref{def:eord}. We can therefore apply Theorem \ref{thm:iwasawaelt}, which shows that the images of the $\Lambda$-adic Asai--Flach classes for different $j$ under the ordinary projector are interpolated by an Iwasawa cohomology class. We define $\cAF^{\cF}_{m, a}$ to be the image of this class under $\pr_{\cF}$; the defining property of the class in Theorem \ref{thm:iwasawaelt} gives the stated interpolation formula.
   \end{proof}

   This is the first part of Theorem \ref{lthm:iwasawa} of the introduction.

   \begin{remark}
    Exactly as in the Rankin--Selberg case, if the Dirichlet character obtained by restricting $\varepsilon_\cF$ to $\hat\ZZ^\times$ does not have conductor dividing $mp^\infty$, then we may multiply $\cAF^{\cF}_{m, a}$ by a suitable element of $\Lambda_{\Gamma} \otimes L_v[\Gal(\QQ(\mu_m) / \QQ)]$ to dispense with the $c$ factors. Cf.~\cite[\S 6.8.1]{leiloefflerzerbes14a}.
   \end{remark}

  \subsection{Local properties at \texorpdfstring{$p$}{p}}

   We now turn to the second part of Theorem \ref{lthm:iwasawa}, which is a description of the localisation of $\cAF^{\cF}_{m, a}$ at $p$. We first need to establish some local properties of the Galois representation $M_{L_v}(\cF)^*$ itself. We shall deduce these from a well-known result of Wiles regarding the local properties of the \emph{standard} Galois representation $\rstd$ (which we abbreviate simply as $\rho$).

   \begin{theorem}[{\cite[Theorem 2]{wiles88}}]
    \label{thm:wiles}
    Suppose $\cF$ is ordinary at $p$. For any prime $\fp \mid p$ of $F$, the restriction of $\rho$ to the decomposition group $D_{\fp}$ at $\fp$ is reducible, with a one-dimensional crystalline subrepresentation $\rho^+_{\fp}$ such that the linearised Frobenius $\Phi = \varphi^{[F_\fp : \Qp]}$ acts on $\Dcris(\rho^+_\fp)$ as multiplication by the $\mathcal{U}(\fp)$-eigenvalue $\lambda_\fp(\cF)$ of $\cF$.
   \end{theorem}

   \begin{corollary}
    \label{cor:asaifiltn}
    If $\cF$ is ordinary at $p$, then the restriction of $M_{L_v}(\cF)^*$ to the decomposition group at $p$ has a 3-step filtration
    \[ M_{L_v}(\cF)^* = \Fil^0 \supset \Fil^1 \supset \Fil^2 \supset \Fil^3 = 0\]
    in which the graded pieces have dimensions 1, 2 and 1 respectively; and the quotient $\Gr^0 M_{L_v}(\cF)^* = \Fil^0 / \Fil^1$ is unramified,
    with arithmetic Frobenius acting via $\alpha_p(\cF)$.
   \end{corollary}

   \begin{proof}
    Since we have $M_{L_v}(\cF) \cong \tInd_F^\QQ (\rho)(t_1 + t_2)$, the assertion concerning $M_{L_v}(\cF)^*$ is equivalent to the assertion that $\tInd_F^\QQ(\rho)$ has a filtration with graded pieces of dimension $1, 2, 1$, and the subspace $\Fil^2 \tInd_F^\QQ(\rho)$ is crystalline, with crystalline Frobenius acting as $\lambda_p(\cF) = p^{t + t'} \alpha_p(\cF)$.

    We first consider the case when $p$ is split in $F$. Our coefficient field $L$ is (by definition) a subfield of $\CC$, containing the images of the two embeddings $\theta_1, \theta_2: F \into \RR$. Hence $\theta_1^{-1}(v)$ is a place of $F$ above $p$ whose decomposition group is identified with $D_p$; we denote this place by $\fp$, and its Galois conjugate by $\fq$, so that $D_{\fq} = \sigma D_{\fp} \sigma^{-1}$ is a decomposition group at $\fq$. As a representation of $\Gal(\overline{F} / F)$, we have $\tInd_F^\QQ (\rho) \cong \rho \otimes \rho^\sigma$, where $\sigma$ denotes some choice of lift to $\Gal(\QQbar / \QQ)$ of the nontrivial element of $\Gal(F / \QQ)$. By Theorem \ref{thm:wiles}, the two terms in the tensor product have 1-dimensional $D_p$-stable subspaces $\rho_{\fp}^+$ and $\rho_{\fq}^+$, which are crystalline with $\varphi$-eigenvalues $\lambda_{\fp}$ and $\lambda_{\fq}$ respectively. Hence the tensor product $\rho_{\fp}^+ \otimes \rho_{\fq}^+$ is a one-dimensional subspace of $\tInd_F^\QQ (\rho)$ which is $D_p$-stable and crystalline, with Frobenius acting as $\lambda_\fp(\cF)\lambda_{\fq}(\cF) = \lambda_p(\cF)$. This gives the 1-dimensional filtration step; and, similarly, the sum $\rho_{\fp} \otimes \rho_{\fq}^+ + \rho_{\fp}^+ \otimes \rho_{\fq}$ is a 3-dimensional $D_p$-stable subspace.

    The case of $p$ inert is more elaborate. In this case, if $\fp = p\cO_F$ is the unique prime above $p$, the decomposition group $D_{\fp}$ is an index 2 subgroup of $D_p$, and we have $\tInd_F^\QQ (\rho) |_{D_p} = \tInd_{D_\fp}^{D_p}(\rho_\fp)$ as representations of $D_p$, where $\rho_{\fp} = \rho |_{D_\fp}$. Since tensor induction is a functor (although not an additive one), one obtains morphisms of $D_p$-representations
    \begin{align*}
     \tInd_{D_\fp}^{D_p}(\rho^+_{\fp}) &\into \tInd_{D_\fp}^{D_p}(\rho_{\fp})
     \quad\text{and}\\
     \tInd_{D_\fp}^{D_p}(\rho) &\twoheadrightarrow \tInd_{D_\fp}^{D_p}(\rho_{\fp} / \rho_{\fp}^+)
    \end{align*}
    whose composition is zero. These give the required $D_p$-stable filtration. Moreover, from the explicit construction of tensor induction in \S\ref{sect:galrep}, one checks that the eigenvalue of $\varphi$ on $\Dcris\left( \tInd(\rho^+_{\fp})\right)$ coincides with that of $\Phi = \varphi^2$ on $\Dcris(\rho^+_{\fp})$, which is $\lambda_p(\cF)$.
   \end{proof}

   With this in hand, we can complete the proof of Theorem \ref{lthm:iwasawa}:

   \begin{corollary}
    The image of $\cAF^{\cF}_{m, a}$ in $H^1\left(\Qp, \Gr^0 M_{L_v}(\cF)^* \otimes \Lambda_\Gamma(-\mathbf j)\right)$ is zero.
   \end{corollary}

   \begin{proof}
    Since $\cAF^{[\cF, j]}_{mp^r, a}$ is the image of a motivic cohomology class for $0 \le j \le \min(k, k')$ and $r \ge 0$, it must lie in the Bloch--Kato $H^1_\mathrm{g}$ subspace, by a theorem of Nekov\'a\v{r} and Nizio\l{} \cite[Theorem B]{nekovarniziol}. However, an Iwasawa cohomology class for an unramified Galois representation which is in $H^1_\mathrm{g}$ at every finite level must be zero, by \cite[Lemma 8.1.5]{kingsloefflerzerbes15b}.
   \end{proof}

  \subsection{The motivic \texorpdfstring{$p$-adic $L$-function}{p-adic L-function}}
   \label{sect:motivicL}

   We now assume $p$ is split in $F$, and we let $\fp$ and $\fq$ be the primes of $F$ above $p$, with $v$ lying above $\fp$, as in the proof of \ref{cor:asaifiltn}. We suppose that $\cF$ is ordinary at $p$, and we let $\alpha_\fp$ and $\alpha_{\fq}$ be the eigenvalues of $\cF$ for the operators $p^{-t} \mathcal{U}(\fp)$ and $p^{-t'} \mathcal{U}(\fq)$; these are in $\cO_{L, v}^\times$, and $\alpha_p = \alpha_{\fp} \alpha_\fq$.

   For convenience we shall also assume that $p \mathop{\|} \fN$, and that $\cF$ is a $p$-stabilisation of an eigenform of level $\fN / p$. In particular, the conductor of the character $\varepsilon$ of $\cF$ is coprime to $p$, so $\varepsilon(\fp)$ and $\varepsilon(\fq)$ are defined. We set $\beta_\fp = p^{k + 1} \varepsilon(\fp) / \alpha_\fp$, $\beta_{\fq} = p^{k' + 1} \varepsilon(\fq) / \alpha_{\fq}$; then the eigenvalues of Frobenius on $\Dcris(M_{L_v}(\cF))$ are
   \[ \big\{ \alpha_{\fp} \alpha_{\fq}, \beta_\fp \alpha_\fq, \alpha_{\fp} \beta_{\fq}, \beta_\fp \beta_\fq\big\}. \]

   We shall impose the following hypothesis:
   \begin{itemize}
    \item (NEZ, for ``no exceptional zero''): None of these four quantities are powers of $p$;  equivalently, the local Euler factor $P_p(\cF, X)$ does not vanish at $p^{-j}$ for any $j \in \ZZ$.
   \end{itemize}

   \begin{remark}
    All four quantities are $p$-Weil numbers of weight $k + k' + 2$, and their $p$-adic valuations are $\{0, k + 1, k' + 1, k + k' + 2\}$, so hypothesis (NEZ) is automatic if $k \ne k'$.
   \end{remark}

   \begin{lemma}
    \label{lem:1dimquot}
    There exists a 1-dimensional quotient $\Gr^1 M_{L_v}(\cF)^* \twoheadrightarrow M_{\fp}$ which is crystalline of Hodge--Tate weight $k' + 1$, and such that Frobenius acts on $\Dcris(M_{\fp})$ by $(\alpha_\fp \beta_\fq)^{-1}$.
   \end{lemma}

   \begin{proof}
    It follows easily from the proof of Corollary \ref{cor:asaifiltn} that in the split case $\Gr^1 M_{L_v}(\cF)^*$ is isomorphic to the direct sum of two one-dimensional crystalline representations, with crystalline Frobenius eigenvalues $(\alpha_\fp \beta_\fq)^{-1}$ and $(\alpha_\fq \beta_\fp)^{-1}$.
   \end{proof}

   \begin{remark}
    Note that $M_\fp$ is uniquely determined if and only if $\alpha_\fp \beta_\fq \ne \beta_\fp \alpha_\fq$. In the exceptional case $\alpha_\fp \beta_\fq = \beta_\fp \alpha_\fq$ (which can only occur if $k = k'$), the graded piece $\Gr^1 M_{L_v}(\cF)^*$ is isomorphic to the direct sum of two copies of the same representation, and we simply choose an arbitrary 1-dimensional quotient. (This case \emph{always} occurs if $\cF$ is a twist of a base-change from $\GL_2 / \QQ$.)
   \end{remark}

   \begin{definition}
    We write
    \[ \mathcal{L}^{\mathrm{PR}}: H^1(\Qp, M_{\fp} \otimes \Lambda_\Gamma(-\mathbf{j})) \to \Dcris(M_\fp) \otimes_{\Zp} \Lambda_{\Gamma} \]
    for the Perrin-Riou big logarithm map (c.f. \cite[Definition~3.4]{leiloefflerzerbes11}).
   \end{definition}

   Because of hypothesis (NEZ), this map is an isomorphism of $L_v \otimes_{\Zp} \Lambda_\Gamma$-modules (cf.~Theorem 8.2.3 and Remark 8.2.4 of \cite{kingsloefflerzerbes15b}). It is characterised by the following interpolation property: for any character of $\Gamma$ of the form $j + \eta$, with $j \in \ZZ$ and $\eta$ a finite-order character of conductor $p^r$, then (after extending $L$ if necessary, so that $\eta$ takes values in $L_v^\times$) we have a commutative diagram
   \begin{diagram}[small]
    H^1(\Qp, M_{\fp} \otimes \Lambda_\Gamma(-\mathbf{j}))
    & \rTo^{\mathcal L^{\mathrm{PR}}} & \Dcris(M_\fp) \otimes_{\Zp} \Lambda_{\Gamma}\\
    \dTo & & \dTo\\
    H^1(\Qp, M_\fp(-j-\eta)) & \rTo & \Dcris(M_\fp)
   \end{diagram}
   in which the vertical arrows are given by specialisation at $\mathbf{j} = j + \eta$, and the bottom horizontal arrow is given by
   \[
    \left.\begin{cases}
     \left( 1 - \frac{p^j}{\alpha_\fp \beta_{\fq}}\right) \left(1 - \frac{\alpha_\fp \beta_{\fq}}{p^{1 + j}} \right)^{-1} & \text{if $r = 0$} \\
     \left( \frac{p^{1 + j}}{\alpha_\fp \beta_{\fq}}\right)^r G(\eta^{-1})^{-1} & \text{if $r \ge 1$}
    \end{cases}\right\} \cdot
    \begin{cases}
     \tfrac{(-1)^{k'-j}}{(k'-j)!} \log & \text{if $j \le k'$,}\\
     (j-k'-1)! \exp^* & \text{if $j > k'$}
    \end{cases}
   \]
   Here $G(\eta^{-1}) = \sum_{a \in (\ZZ / p^r \ZZ)^\times} \eta(a)^{-1} \zeta_{p^r}^a$ is the Gauss sum, and $\log$ and $\exp^*$ are the Bloch--Kato logarithm and dual-exponential maps for the de Rham representation $M_\fp(-j-\eta)$. Cf.~\cite[Theorem 8.2.8]{kingsloefflerzerbes15b}.

   Attached to the eigenform $\cF$, we have the Asai--Flach class $\cAF^{\cF}_{1, a}$. The localisation of this class at $p$ maps to zero in $\Gr^0 M_{L_v}(\cF)^*$, as we have seen; so we may consider it as a class in the Iwasawa cohomology of $\Gr^1 M_{L_v}(\cF)^*$, and project it to the quotient $M_\fp$.

   \begin{definition}
   For any integer $c > 1$ coprime to $6p\Nm_{F/\QQ}(\fN)$, we define
    \[
     {}_c L_{\fp, \Asai}^{\imp}(\cF) = (\mathcal{L}^{\mathrm{PR}} \circ \pr_{M_\fp} \circ \loc_p)\left( \cAF^{\cF}_{1, a}\right) \in \Lambda_\Gamma \otimes \Dcris(M_\fp),
    \]
    and
    \[
     L_{\fp, \Asai}^{\imp}(\cF) = \left(c^2 - c^{2\mathbf{j}-k-k'}\varepsilon_\cF(c)\right)^{-1} {}_c L_{\fp, \Asai}^{\imp}(\cF) \in \operatorname{Frac} \Lambda_{\Gamma} \otimes \Dcris(M_\fp)
    \]
    (which is independent of $c$). Finally, we set
    \[
     L_{\fp, \Asai}(\cF) = \left(\prod_{\ell \mid \Nm_{F/\QQ}(\fN)} C_\ell(\ell^{-1-\mathbf{j}})^{-1}\right) L_{\fp, \Asai}^{\imp}(\cF) \in \operatorname{Frac} \Lambda_{\Gamma} \otimes \Dcris(M_\fp),
    \]
    where $C_\ell \in L[X]$ are the polynomials from Definition \ref{def:Limp}.
   \end{definition}

   \begin{remark} \
    \begin{enumerate}
     \item Note that $L_{\fp, \Asai}^{\imp}(\cF)$ can be viewed as a $p$-adic meromorphic function on the weight space $\mathcal{W} = \Spec \Lambda_\Gamma$. Since ${}_c L_{\fp, \Asai}^{\imp}(\cF)$ is analytic, the only possible poles of $L_{\fp, \Asai}^{\imp}(\cF)$ are at zeroes of the factor $\left(c^2 - c^{2\mathbf{j}-k-k'}\varepsilon_\cF(c)\right)$. In particular, the function $L_{\fp, \Asai}^{\imp}(\cF)$ is analytic everywhere if $\varepsilon_{\cF} |_{\hat\ZZ^\times}$ is non-trivial; and if $\varepsilon_{\cF} |_{\hat\ZZ^\times}$ is trivial then it has at most two poles, one at $\mathbf{j} = \tfrac{k + k'}{2} + 1$ and the other at  $\mathbf{j} = \tfrac{k + k'}{2} + 1 + \eta$ where $\eta$ is the nontrivial quadratic character of $\Gamma$.

     \item The definition of these $L$-functions still makes sense if (NEZ) is not satisfied; in this case $\mathcal{L}^{\mathrm{PR}}$ takes values in $\Dcris(M_{\fp}) \otimes I^{-1}$, where $I$ is a certain ideal in $\Lambda_{\Gamma}$.

     \item Thus there are three possible sources of poles for the primitive $L$-function $L_{\fp, \Asai}$  in general: those arising from the cancellation of the $c$ factor, those arising from zeroes of the polynomials $C_\ell$, and those arising from singularities of the Perrin--Riou map when (NEZ) does not hold. We expect, nonetheless, that if $\cF$ is non-CM and not a twist of a base-change from $\GL_2 / \QQ$, then $L_{\fp, \Asai}(\cF)$ should be analytic everywhere.
    \end{enumerate}
   \end{remark}

   We formulate the following conjecture relating the $p$-adic and complex $L$-funct\-ions:

   \begin{conjecture}
    Suppose $k > k'$, and let $j$ be an integer with $k' < j \le k$. Then $L_{\fp, \Asai}(\cF)$ and $L_{\fp, \Asai}^{\imp}(\cF)$ are analytic at $\mathbf{j}=j$, and we have
    \begin{align*}
     L_{\fp, \Asai}(\cF)(j) = 0 &\Longleftrightarrow L_{\Asai}(\cF, 1+j) = 0,\\
     L_{\fp, \Asai}^{\imp}(\cF)(j) = 0 &\Longleftrightarrow L_{\Asai}^{\imp}(\cF, 1+j) = 0.
    \end{align*}
   \end{conjecture}

   \begin{remark}
    As we have emphasised in the introduction, we \textbf{cannot} prove this conjecture, so we cannot rule out the possibility that $L_{\fp, \Asai}^{\imp}(\cF)$ is identically zero.
   \end{remark}

  \subsection{Big image results}
   \label{sect:bigimage}
   Let $\cF$ be any Hilbert modular eigenform for $F$, of level $U_1(\fN)$ for some $\fN$, and weight $(k+2, k'+2, t, t')$ with $k, k' \ge 0$. (We do \emph{not} assume in this section that $\cF$ be ordinary, that $p \mid \fN$, or that $p$ be split in $F$.)

   \begin{definition}
    We say that $\cF$ satisfies condition (BI) (for ``big image'') at $v$ if the following two statements hold for some (or, equivalently, any) $\Gal(\QQbar / \QQ)$-stable $\cO_{L, v}$-lattice $T$ in $\rho_{\cF, v}^{\Asai}$:
    \begin{enumerate}[(i)]
     \item $T \otimes k_v$ is an irreducible $k_v[\Gal(\QQbar / \QQ(\mu_{p^\infty}))]$-module, where $k_v$ is the residue field of $\cO_{L, v}$.
     \item There exists $\tau \in \Gal(\QQbar / \QQ(\mu_{p^\infty}))$, lifting the non-trivial element $\sigma \in \Gal(F / \QQ)$, such that $T / (\tau - 1)T$ is free of rank 1 over $\cO_{L, v}$.
    \end{enumerate}
   \end{definition}

   This is a slight strengthening of $\operatorname{Hyp}(K_\infty, T)$ of \cite{rubin00}, with the field $K_\infty$ in \emph{op.cit.} taken to be $\QQ(\mu_{p^\infty})$. (Our condition on $\tau$ is slightly more restrictive, since we also require $\tau$ to act nontrivially on $F$.)

   In the remainder of this section, we shall give some criteria which imply that condition (BI) is satisfied for a plentiful supply of primes $v$. As the isomorphism class of $\rho_{\cF, v}^{\Asai}$ depends only on the newform associated to $\cF$, we may assume without loss of generality that $\cF$ is itself a newform. We impose the following hypotheses on $\cF$:
   \begin{enumerate}
    \item $\cF$ is not of CM type;
    \item $\cF$ is not a twist of a base-change from $\GL_2 / \QQ$.
   \end{enumerate}

   \begin{theorem}[Lapid--Rogawski]
    Let $\sigma$ be the nontrivial element of $\Gal(F / \QQ)$, and let $\cF^\sigma$ be the internal conjugate of $\cF$ (the unique newform whose $\cT(\fn)$-eigenvalue is $\lambda(\fn^\sigma)$ for all $\fn$). Then there is no Hecke character $\kappa$ such that $\cF^\sigma = \cF \otimes \kappa$.
   \end{theorem}

   \begin{proof}
    This is a special case of the main theorem of \cite{lapidrogawski98}.
   \end{proof}

   We have defined above Galois representations $\rho_{\cF, v}^{\Asai}$ and $\rstd$, for every prime $v$ of $L$, which are unique up to conjugation in $\GL_2(L_v)$. After conjugating appropriately, we can and do assume that the images of these representations lie in $\GL_2(\cO_{L, v})$.

   \begin{proposition}\label{prop:irreducible}
    The representation $\rho_{\cF, v}^{\Asai}$ is absolutely irreducible, and remains so as a representation of $G_{F^{\ab}}$, for all primes $v$ of $L$. For all but finitely many $v$ this remains true after reduction modulo $v$.
   \end{proposition}

   \begin{proof}
    The characteristic 0 statement, for all $v$, follows from \cite[Remark 5.21]{nekovar-semisimplicity}; so let us prove the statement regarding reduction modulo $v$ for almost all $v$.

    We first consider the case where $\cF$ is not only not twist-equivalent to $\cF^\sigma$, but is not twist-equivalent to any Galois conjugate of $\cF^\sigma$. Then we may apply \cite[Theorem 3.4.1]{loeffler17} to $\cF$ and $\cF^\sigma$. The theorem is stated in \emph{op.cit.} for elliptic modular forms, but it applies also to Hilbert modular forms (as noted in Remark 2.3.2 of \emph{op.cit.}). This shows that there is a subfield $K$ of $L$ such that for all but finitely many $v$, the image of $G_{F^{\ab}}$ under $\rstd \times \rho_{\cF^\sigma, v}^{\mathrm{std}}$ is conjugate to $\SL_2(\cO_{K, u}) \times \SL_2(\cO_{K, u})$, where $u$ is the prime of $K$ below $v$. Hence the tensor product of these two representations is irreducible mod $v$ as a representation of $G_{F^{\ab}}$, and this coincides with the restriction of $\rho_{\cF, v}^{\Asai}$.

    We now consider the case where $\cF$ is Galois-conjugate to a twist of $\cF^\sigma$, but not equal to a twist of $\cF^\sigma$. In this case, the same argument shows that for almost all $v$, either the image of $G_{F^\ab}$ under $\rstd \times \rho_{\cF^\sigma, v}^{\mathrm{std}}$ is conjugate to $\SL_2(\cO_{K, u}) \times \SL_2(\cO_{K, u})$, or $[K_u : \Qp] > 1$ and the image of $G_{F^\ab}$ is conjugate to the image of $\SL_2(\cO_{K, u})$ under a map of the form $(\id, \alpha)$ for some $\alpha \in \Gal(K_u / \Qp)$. Finally, $\alpha$ cannot be the identity, since otherwise $\cF$ would be twist-equivalent to $\cF^\sigma$. If $v$ is not one of the finitely many primes ramifying in $L/\QQ$, it follows that $\alpha$ acts nontrivially on the residue field $k_u$ of $K_u$. Since the tensor product of the standard representation of $\SL_2(k_u)$ and its conjugate by $\alpha$ is irreducible (a simple case of the classification of irreducible representations of $\SL_2$ of a finite field in defining characteristic \cite[\S 30]{brauernesbitt41}), we are done.
   \end{proof}

   \begin{theorem}\label{thm:existtau}
    Suppose there is at least one ramified prime of $F$ which does not divide the level of $\cF$.

    If $\cF$ is not Galois-conjugate to any twist of $\cF^\sigma$, then Condition (BI) is satisfied at all but finitely many primes $v$ of $L$. If $\cF$ is Galois-conjugate to a twist of $\cF^\sigma$, then Condition (BI) is satisfied at all but finitely many degree 1 primes $v$ of $L$.
   \end{theorem}

   The proof of Theorem \ref{thm:existtau} will take several steps. We assume without loss of generality that $L$ is the smallest extension of $\QQ$ containing the Hecke eigenvalues of $\cF$.

   \begin{definition}[{cf.~\cite[\S B.3]{nekovar12}}]
    An \emph{inner twist} of $\cF$ is a pair $(\alpha, \chi)$, where $\alpha$ is an embedding $L \into \QQbar$ and $\chi$ is a finite-order $\QQbar$-valued Hecke character of $F$, such that $\alpha(\cF) = \cF \otimes \chi$.
   \end{definition}

   One knows that if $(\alpha, \chi)$ is an inner twist, then $\alpha(L) = L$ and $\chi$ takes values in $L$; since $\cF$ is non-CM-type, $\chi$ is uniquely determined by $\alpha$, and the $\alpha \in \operatorname{Aut}(L / \QQ)$ which give inner twists are precisely those which are trivial on the subfield $K \subseteq L$ generated by the quotients $\lambda_{\fp}(\cF)^2 / \varepsilon(\fp)$, as $\fp$ ranges over primes of $F$. Moreover, for all inner twists $(\alpha, \chi)$, the character $\chi$ is unramified outside the primes dividing the level of $\cF$.

   Note that $(\alpha, \chi) \to (\alpha, \chi^\sigma)$ gives a bijection between the inner twists of $\cF$ and those of $\cF^\sigma$.

   \begin{lemma}
    Suppose there is at least one ramified prime of $F$ which does not divide the level of $\cF$.

    Then, for all but finitely many primes $p$, there exists $\tau \in G_{\QQ} - G_F$ such that $\tau$ acts trivially on $\QQ(\mu_{p^\infty})$, and for any inner twist $(\alpha, \chi)$ of $\cF$ or $\cF^\sigma$, we have $\chi(\tau^2) = 1$.
   \end{lemma}

   \begin{proof}
    By assumption, there is some prime $\ell \mid D$ which is coprime to the level of $\cF$, and hence coprime to the conductors of all of the Dirichlet characters $\chi |_{\hat\ZZ^\times}$ where $(\alpha, \chi)$ varies over the inner twists of $\cF$ or $\cF^\sigma$. Therefore, we may find primes $q$ which are quadratic non-residues modulo $D$ and such that $\chi(q) = 1$ for all such $q$.

    Let $F'$ be the finite abelian extension of $F$ cut out by all of the characters $\chi$ and $\chi^\sigma$. Then $F' / \QQ$ is Galois, and if $\tau_0$ is the conjugacy class of any $q$ as above, we have $\tau_0^2 = 1$ in $\Gal(F' / F)$.

    If $p$ is not one of the finitely many primes ramifying in $F' / \QQ$, then $F'$ is linearly disjoint from $\QQ(\mu_{p^\infty})$ over $\QQ$ (since one field is unramified at $p$ and the other totally ramified). So we may find $\tau \in G_{\QQ}$ which acts trivially on the cyclotomic field and as $\tau_0$ on $F'$, and this $\tau$ satisfies the conditions.
   \end{proof}

   \begin{corollary}
    In the setting of the previous lemma, if $\cF$ is not Galois-conjugate to a twist of $\cF^\sigma$, then for all but finitely many primes $v$ of $L$ we have
    \[\SL_2(\cO_{K, u}) \subseteq \{ \rho(\tau^2): \tau \in G_{\QQ(\mu_{p^\infty})}, \tau \notin G_F\}, \]
    where $\rho = \rstd$ and $u$ is the prime of $K$ below $v$. If $\cF$ is Galois-conjugate to some twist of $\cF^\sigma$, then this holds for all but finitely many $v$ of degree 1.
   \end{corollary}

   \begin{proof}
    Let $\tau$ be any element as in the previous lemma. Then $\rho(\tau^2)$ lies in $\SL_2(\cO_{K, w})$, since $\tau^2$ is in the kernel of all the inner twists of $\cF$ and of the cyclotomic character.

    However, if $\tau$ satisfies the conclusions of the lemma, so does $\gamma \tau$ for any $\gamma \in G_{F^{\ab}}$; and replacing $\tau$ by $\gamma \tau$ replaces $\rho(\tau^2)$ by $\rho(\gamma) \rho(\tau^2) \rho(\tau^{-1} \gamma \tau)$.

    If $\cF$ is not Galois-conjugate to any twist of $\cF^\sigma$, then (as we have seen in the proof of Proposition \ref{prop:irreducible}) as $\gamma$ varies over $G_{F^{\ab}}$, the pair $(\rho(\gamma), \rho(\tau^{-1} \gamma \tau))$ hits every element of $\SL_2(\cO_{K, u}) \times \SL_2(\cO_{K, u})$, so in particular $\rho( (\gamma \tau)^2 )$ can take every value in $\SL_2(\cO_{K, u})$. The same holds if $\cF$ is Galois-conjugate to some twist of $\cF^\sigma$, as long as the automorphism of $K$ mapping $\cF$ to a twist of $\cF^\sigma$ is not contained in the decomposition group of $u$; in particular this holds if $K_u = \Qp$ as claimed.
   \end{proof}

   \begin{proposition}
    If $p \ne 2$, and $\tau \in G_\QQ$ is such that $\tau \notin G_F$ and $\rstd(\tau^2)$ is conjugate in $\GL_2(\cO_{L, v})$ to $\stbt 1 1 0 1$, then the quotient $\cO_{L, v}^{\oplus 4} / (\rho_{\cF, v}^{\Asai}(\tau) - 1)$ is free of rank 1 over $\cO_{L, v}$.
   \end{proposition}

   \begin{proof}
    If we fix a basis $(v_1, v_2)$ of the underlying space of $\rstd$ in which $\tau^2$ acts as $\stbt 1 1 0 1$, then $(v_1 \otimes v_1, v_2 \otimes v_1, v_1 \otimes v_2, v_2 \otimes v_2)$ is a basis of $\rho_{\cF, v}^{\Asai}$ and the matrix of $\tau$ in this basis is $\left(\begin{smallmatrix} 1 & 0 & 1 & 0 \\ 0 & 0 & 1 & 0 \\ 0 & 1 & 0 & 1 \\ 0 & 0 & 0 & 1\end{smallmatrix}\right)$. The Jordan normal form of this matrix is $\left(\begin{smallmatrix} 1 & 1 & 0 & 0 \\ 0 & 1 & 1 & 0 \\ 0 & 0 & 1 & 0 \\ 0 & 0 & 0 & -1\end{smallmatrix}\right)$, and one can check that the similarity transformation relating these matrices lies in $\GL_4(\Zp)$ for any $p \ne 2$. So the space of coinvariants of $\rho_{\cF, v}^{\Asai}(\tau)$ is free of rank 1 as required.
   \end{proof}

   This completes the proof of Theorem \ref{thm:existtau}.

  \subsection{Bounding Selmer groups}
   \label{sect:boundSel}
   We shall now give the proof of Theorem \ref{lthm:boundSel} of the introduction. For the convenience of the reader, we shall recall the list of hypotheses we are imposing.

   \begin{itemize}
    \item $\cF$ is an eigenform of level $\fN$, with coefficients in a number field $L \supset F$ and weight $(k+2, k'+2, t, t')$, where $k, k' \ge 0$.
    \item $p$ is a rational prime, with $p = \fp \fq$ split in $F$ and $p \mathop{\|} \fN$.
    \item $v$ is a prime of $L$ above $\fp$.
    \item $\cF$ is ordinary at $p$ (with respect to $v$), and is the $p$-stabilisation of an eigenform of level $\fN/p$.
    \item The hypotheses (NEZ) of \S\ref{sect:motivicL} and (BI) of \S\ref{sect:bigimage} hold.
    \item $p \ge k + k' + 3$.
   \end{itemize}

   We also fix a choice of 1-dimensional subquotient $M_{\fp}$ of $M_{L_v}(\cF)^*$ as in Lemma \ref{lem:1dimquot}, and a basis $\Omega_{\fp}$ of the 1-dimensional $L_v$-vector space $\Dcris(M_\fp)$. Finally, we choose an integer $c > 1$ coprime to $6p\Nm_{F/\QQ}(\fN)$.

   Let $R = \cO_{L, v}$. We let $M_R(\cF)^*$ be the $R$-submodule of $M_{L_v}(\cF)^*$ generated by the image of $H^2\left(Y_1^*(\fN)_{\overline\QQ}, \TSym^{[k, k']} \sH_R(\cA)(2)\right)$; this is non-zero (by comparison with de Rham cohomology) and stable under $\Gal(\QQbar / \QQ)$, and hence must be a lattice of full rank, since we have shown that $M_{L_v}(\cF)^*$ is irreducible.

   \begin{lemma}
    For every finite extension $K / \QQ$, every finite set of primes $S$ containing all primes dividing $p \Nm_{F/\QQ}(\fN)$, and every $j \in \ZZ$, the projection map
    \[ \pr_{\cF}: H^3\left( Y_1^*(\fN)_{\cO_{K, S}}, \TSym^{[k, k']} \sH_R(\cA)(2-j)\right) \to H^1\left(\cO_{K, S}, M_{L_v}(\cF)^*(-j)\right)\]
    factors through $H^1\left(\cO_{K, S}, M_R(\cF)^*(-j)\right)$.
   \end{lemma}

   \begin{proof}
    Let $\fm$ denote the maximal ideal of the Hecke algebra of level $U_1(\fN)$ (with $R$ coefficients) corresponding to $\cF$. Our conditions on $v$ imply that $\fm$ satisfies the condition $(\mathbf{LI}_{\Ind \overline{\rho}})$ of \cite{dimitrov05} (this is where $p \ge k + k' + 3$ is used); so by Theorem 0.3(ii) of \emph{op.cit.}, the localisation of the cohomology of $Y_1(\fN)$ at $\fm$ vanishes outside the middle degree. Thus the projection map $\pr_{\cF}$ is defined over $R$.
   \end{proof}

   We now define an appropriate Selmer group. We define $A = M_{R}(\cF)(1) \otimes \Qp/\Zp$; and we let $\Fil^\fp A$ be the submodule of $A$ (of corank 2) dual to the kernel of $\Fil^1 M_R(\cF)^* \to M_{\fp}$.

   \begin{definition}
    We set
    \[ \operatorname{Sel}^{(\fp)}(\QQ(\mu_{p^\infty}), A) =
    \left\{
     \begin{array}{c}
      x \in H^1(\QQ(\mu_{p^\infty}), \cF): \loc_\ell(x) = 0 \text{ for $\ell \ne p$}, \\
     \loc_p(x) \in \operatorname{image} H^1(\Qp(\mu_{p^\infty}), \Fil^\fp A)
     \end{array}
    \right\}
    \]
    and
    \[ X^{(\fp)}(\QQ(\mu_{p^\infty}), \cF) = \operatorname{Sel}^{(\fp)}(\QQ(\mu_{p^\infty}), \cF)^\vee \]
    (where $\vee$ denotes Pontryagin dual).
   \end{definition}

   \begin{theorem}[Theorem \ref{lthm:boundSel}]
    There exists an integer $n$ such that
    \[ \operatorname{char}_{\Lambda_\Gamma} X^{(\fp)}(\QQ(\mu_{p^\infty}), A) \mid  \frac{p^n{}_c L_{\fp, \Asai}^{\imp}(\cF)}{\Omega_\fp}. \]
   \end{theorem}

   \begin{proof}
    This follows by exactly the same Euler system argument as in \cite[Theorem 11.6.4]{kingsloefflerzerbes15b}. (Note that the Euler system norm relations are only used for primes $\ell$ whose Frobenii act on $M_{L_v}(\cF)^*$ as a conjugate of $\tau$; all such primes $\ell$ are necessarily inert in $F$, because $\tau$ maps to $\sigma$ in $\Gal(F / \QQ)$. Hence the fact that we have not established the norm relations for all primes split in $F$ does not cause any trouble here.)
   \end{proof}

\renewcommand{\MR}[1]{MR \href{http://www.ams.org/mathscinet-getitem?mr=#1}{#1}.}
\providecommand{\bysame}{\leavevmode\hbox to3em{\hrulefill}\thinspace}
\providecommand{\MR}{\relax\ifhmode\unskip\space\fi MR }
\providecommand{\MRhref}[2]{%
  \href{http://www.ams.org/mathscinet-getitem?mr=#1}{#2}
}
\providecommand{\href}[2]{#2}


\begin{thebibliography}{MVW06}

\bibitem[Anc15]{ancona15}
 G.~Ancona,
 \href{http://dx.doi.org/10.1007/s00229-014-0708-4}{\emph{D{\'e}composition de
   motifs ab{\'e}liens}}, Manuscripta Math. \textbf{146} (2015), no.~3--4,
 307--328. \MR{3312448}

 \bibitem[Asa77]{asai77}
 T.~Asai, \href{http://dx.doi.org/10.1007/BF01391220}{\emph{On certain
   {D}irichlet series associated with {H}ilbert modular forms and {R}ankin's
   method}}, Math. Ann. \textbf{226} (1977), no.~1, 81--94. \MR{0429751}

\bibitem[Be{\u\i}84]{beilinson84}
 A.~Be{\u\i}linson,
 \href{http://dx.doi.org/10.1007/BF02105861}{\emph{Higher regulators and
   values of {$L$}-functions}}, Current problems in mathematics, {V}ol. 24,
 Itogi Nauki i Tekhniki, Akad. Nauk SSSR Vsesoyuz. Inst. Nauchn. i Tekhn.
 Inform., Moscow, 1984, pp.~181--238. \MR{760999}

 \bibitem[BDR15]{BDR15b}
 M.~Bertolini, H.~Darmon, and V.~Rotger,
 \href{http://dx.doi.org/10.1090/S1056-3911-2015-00675-0}{\emph{{B}eilinson--{F}lach
   elements and {E}uler systems {II}: the {B}irch and {S}winnerton-{D}yer
   conjecture for {H}asse--{W}eil--{A}rtin {$L$}-functions}}, J. Algebraic Geom.
 \textbf{24} (2015), no.~3, 569--604. \MR{3344765}

 \bibitem[BK90]{blochkato90}
 S.~Bloch and K.~Kato,
 \href{http://dx.doi.org/10.1007/978-0-8176-4574-8}{\emph{{$L$}-functions and
   {T}amagawa numbers of motives}}, The {G}rothendieck {F}estschrift, {V}ol.\
 {I} (P.~Cartier et~al., eds.), Progr. Math., vol.~86, Birkh{\"a}user,
 Boston, MA, 1990, pp.~333--400. \MR{1086888}

 \bibitem[BN41]{brauernesbitt41}
 R.~Brauer and C.~Nesbitt, \href{http://dx.doi.org/10.2307/1968918}{\emph{On the
   modular characters of groups}}, Ann. of Math. (2) \textbf{42} (1941),
 556--590. \MR{0004042}

 \bibitem[BC16]{brunaultchida16}
 F.~Brunault and M.~Chida,
 \href{http://dx.doi.org/10.1007/s40316-015-0043-5}{\emph{Regulators for
   {R}ankin--{S}elberg products of modular forms}}, Ann. Math. Qu\'ebec
 \textbf{40} (2016), no.~2, 221--249. \MR{3529182}

 \bibitem[BL84]{brylinskilabesse84}
 J.-L. Brylinski and J.-P. Labesse,
 \href{http://www.numdam.org/item?id=ASENS_1984_4_17_3_361_0}{\emph{Cohomologie
   d'intersection et fonctions {$L$} de certaines vari\'et\'es de {S}himura}},
 Ann. Sci. \'Ecole Norm. Sup. (4) \textbf{17} (1984), no.~3, 361--412.
 \MR{777375}

 \bibitem[DR14]{darmonrotger12}
 H.~Darmon and V.~Rotger,
 \href{http://dx.doi.org/10.24033/asens.2227}{\emph{Diagonal cycles and
   {E}uler systems {I}: a {$p$}-adic {G}ross--{Z}agier formula}}, Ann. Sci.
 {\'E}cole Norm. Sup. (4) \textbf{47} (2014), no.~4, 779--832. \MR{3250064}

 \bibitem[DLP16]{DLP}
 L.~Demb{\'e}l{\'e}, D.~Loeffler, and A.~Pacetti,
 \href{http://arxiv.org/abs/1612.06625}{\emph{Non-paritious {H}ilbert modular
   forms}}, preprint, 2016, \path{arXiv:1612.06625}.

 \bibitem[DM91]{deningermurre91}
 C.~Deninger and J.~Murre,
 \href{http://dx.doi.org/10.1515/crll.1991.422.201}{\emph{Motivic
   decomposition of abelian schemes and the {F}ourier transform}}, J. reine
 angew. Math. \textbf{422} (1991), 201--219. \MR{1133323}

 \bibitem[Dim05]{dimitrov05}
 M.~Dimitrov,
 \href{http://dx.doi.org/10.1016/j.ansens.2005.03.005}{\emph{Galois
   representations modulo {$p$} and cohomology of {H}ilbert modular varieties}},
 Ann. Sci. \'Ecole Norm. Sup. (4) \textbf{38} (2005), no.~4, 505--551.
 \MR{2172950}

 \bibitem[Dim09]{dimitrov09}
 \bysame, \href{http://dx.doi.org/10.1112/S0010437X09004205}{\emph{On {I}hara's
   lemma for {H}ilbert modular varieties}}, Compos. Math. \textbf{145} (2009),
 no.~5, 1114--1146. \MR{2551991}

 \bibitem[Gro66]{grothendieck66}
 A.~Grothendieck, \href{http://dx.doi.org/10.1007/bf02684807}{\emph{On
   the de {R}ham cohomology of algebraic varieties}}, Pub. Math. IH{\'E}S
 \textbf{29} (1966), 95--103. \MR{0199194}

 \bibitem[HW98]{HW98}
 A.~Huber and J.~Wildeshaus,
 \href{https://www.math.uni-bielefeld.de/documenta/vol-03/03.html}{\emph{Classical
   motivic polylogarithm according to {B}eilinson and {D}eligne}}, Doc. Math.
 \textbf{3} (1998), 27--133. \MR{1643974}

 \bibitem[Im94]{im94}
 J.~Im, \emph{Twisted tensor {$L$}-functions attached to {H}ilbert modular
  forms}, Elliptic curves and related topics, CRM Proc. Lecture Notes, vol.~4,
 Amer. Math. Soc., Providence, RI, 1994, pp.~111--119. \MR{1260958}

 \bibitem[Kat04]{kato04}
 K.~Kato,
 \href{http://smf4.emath.fr/en/Publications/Asterisque/2004/295/html/smf_ast_295_117-290.html}{\emph{{$P$}-adic
   {H}odge theory and values of zeta functions of modular forms}},
 Ast{\'e}risque \textbf{295} (2004), 117--290, Cohomologies $p$-adiques et
 applications arithm{\'e}tiques. III. \MR{2104361}

 \bibitem[Kin98]{kings98}
 G.~Kings,
 \href{http://dx.doi.org/10.1215/S0012-7094-98-09202-X}{\emph{Higher
   regulators, {H}ilbert modular surfaces, and special values of
   {$L$}-functions}}, Duke Math. J. \textbf{92} (1998), no.~1, 61--127.
 \MR{1609325}

 \bibitem[Kin15]{kings13}
 \bysame, \href{http://dx.doi.org/10.1017/CBO9781316163757.013}{\emph{Eisenstein
   classes, elliptic {S}oul{\'e} elements and the {$\ell$}-adic elliptic
   polylogarithm}}, The {B}loch--{K}ato conjecture for the {R}iemann zeta
 function (J.~Coates et al., eds.), London Math. Soc. Lecture Note Ser., vol. 418, Cambridge
 Univ. Press, 2015. \MR{3497682}

 \bibitem[KLZ15]{kingsloefflerzerbes15a}
 G.~Kings, D.~Loeffler, and S.L.~Zerbes,
 \href{http://arxiv.org/abs/1501.03289}{\emph{Rankin--{E}isenstein classes for
   modular forms}}, to appear in Amer. J. Math., 2015, \path{arXiv:1501.03289}.

 \bibitem[KLZ17]{kingsloefflerzerbes15b}
 \bysame,
 \href{http://dx.doi.org/10.4310/CJM.2017.v5.n1.a1}{\emph{Rankin--{E}isenstein
   classes and explicit reciprocity laws}}, Cambridge J. Math. \textbf{5}
 (2017), no.~1, 1--122. \MR{3637653}

 \bibitem[LR98]{lapidrogawski98}
 E.~Lapid and J.~Rogawski,
 \href{http://dx.doi.org/10.1515/form.10.2.175}{\emph{On twists of cuspidal
   representations of {${\rm GL}(2)$}}}, Forum Math. \textbf{10} (1998), no.~2,
 175--197. \MR{1611951}

 \bibitem[LLZ11]{leiloefflerzerbes11}
 A.~Lei, D.~Loeffler, and S.L.~Zerbes,
 \href{http://dx.doi.org/10.2140/ant.2011.5.1095}{\emph{Coleman maps and the
   $p$-adic regulator}}, Algebra \& Number Theory \textbf{5} (2011), no.~8,
 1095--1131. \MR{2948474}

 \bibitem[LLZ14]{leiloefflerzerbes14a}
 \bysame, \href{http://dx.doi.org/10.4007/annals.2014.180.2.6}{\emph{Euler
   systems for {R}ankin--{S}elberg convolutions of modular forms}}, Ann. of
 Math. (2) \textbf{180} (2014), no.~2, 653--771. \MR{3224721}

 \bibitem[LLZ15]{leiloefflerzerbes14b}
 \bysame, \href{http://dx.doi.org/10.1112/S0010437X14008033}{\emph{Euler systems
   for modular forms over imaginary quadratic fields}}, Compos. Math.
 \textbf{151} (2015), no.~9, 1585--1625. \MR{3406438}

 \bibitem[Liu16]{liu16}
 Y.~Liu,
 \href{http://dx.doi.org/10.1007/s00222-016-0645-9}{\emph{Hirzebruch--{Z}agier
   cycles and twisted triple product {S}elmer groups}}, Invent. Math.
 \textbf{205} (2016), no.~3, 693--780. \MR{3539925}

 \bibitem[Loe17]{loeffler17}
 D.~Loeffler, \href{http://dx.doi.org/10.1017/S0017089516000367}{\emph{Images
   of adelic {G}alois representations for modular forms}}, Glasgow Math. J.
 \textbf{59} (2017), no.~1, 11--25.

 \bibitem[LSZ]{loefflerskinnerzerbes16}
 D.~Loeffler, C.~Skinner, and S.L.~Zerbes,
 \href{http://arxiv.org/abs/1608.06112}{\emph{Syntomic regulators of
   {A}sai--{F}lach classes}}, preprint, \path{arXiv:1608.06112}.

 \bibitem[MVW06]{mazzavoevodskyweibel06}
 C.~Mazza, V.~Voevodsky, and C.~Weibel, \emph{Lecture notes on
  motivic cohomology}, Clay Mathematics Monographs, vol.~2, Amer. Math. Soc.,
 Providence, RI, 2006. \MR{2242284}

 \bibitem[Mil90]{milne-mixed}
 J.S.~Milne, \emph{Canonical models of (mixed) {S}himura varieties and
  automorphic vector bundles}, Automorphic forms, {S}himura varieties, and
 {$L$}-functions, {V}ol.\ {I} ({A}nn {A}rbor, {MI}, 1988), Perspect. Math.,
 vol.~10, Academic Press, Boston, MA, 1990, pp.~283--414. \MR{1044823}

 \bibitem[Nek12]{nekovar12}
 J.~Nekov{\'a}{\v{r}},
 \href{http://dx.doi.org/10.4153/CJM-2011-077-6}{\emph{Level raising and
   anticyclotomic {S}elmer groups for {H}ilbert modular forms of weight two}},
 Canad. J. Math. \textbf{64} (2012), no.~3, 588--668. \MR{2962318}

 \bibitem[Nek16]{nekovar-semisimplicity}
 \bysame,
 \href{http://webusers.imj-prg.fr/~jan.nekovar/pu/semi.pdf}{\emph{{E}ichler--{S}himura relations and semi-simplicity of {\'e}tale cohomology of quaternionic
   {S}himura varieties}}, to appear in Ann.~Sci.~E.N.S., 2016.

 \bibitem[NN16]{nekovarniziol}
 J.~Nekov{\'a}{\v{r}} and W.~Nizio{\l},
 \href{http://dx.doi.org/10.2140/ant.2016.10.1695}{\emph{Syntomic cohomology
   and {$p$}-adic regulators for varieties over {$p$}-adic fields}}, Algebra \&
 Number Theory \textbf{10} (2016), no.~8, 1695--1790. \MR{3556797}

 \bibitem[Rub00]{rubin00}
 K.~Rubin, \emph{Euler systems}, Annals of Mathematics Studies, vol. 147,
 Princeton Univ. Press, 2000. \MR{1749177}

 \bibitem[Sai86]{saito86}
 M.~Saito, \href{http://dx.doi.org/10.3792/pjaa.62.360}{\emph{Mixed
   {H}odge modules}}, Proc. Japan Acad. Ser. A Math. Sci. \textbf{62} (1986),
 no.~9, 360--363. \MR{888148}

 \bibitem[Sai89]{saito89}
 \bysame, \emph{Introduction to mixed {H}odge modules}, Ast{\'e}risque
 \textbf{179-180} (1989), 10, 145--162, Actes du Colloque de Th\'eorie de
 Hodge (Luminy, 1987). \MR{1042805}

 \bibitem[Shi78]{shimura78}
 G.~Shimura, \href{http://dx.doi.org/10.1215/S0012-7094-78-04529-5}{\emph{The
   special values of the zeta functions associated with {H}ilbert modular
   forms}}, Duke Math. J. \textbf{45} (1978), no.~3, 637--679. \MR{507462}

 \bibitem[Sus87]{suslin87}
 A.~Suslin, \href{http://dx.doi.org/10.1007/BF00533985}{\emph{Torsion in
   {$K_2$} of fields}}, {$K$}-Theory \textbf{1} (1987), no.~1, 5--29.
 \MR{899915}

 \bibitem[TX16]{tianxiao16}
 Y.~Tian and L.~Xiao,
 \href{http://smf4.emath.fr/en/Publications/Asterisque/2016/382/html/smf_ast_382_73-162.php}{\emph{{$p$}-adic
   cohomology and classicality of overconvergent {H}ilbert modular forms}},
 Ast{\'e}risque \textbf{382} (2016), 73--162. \MR{3581176}

 \bibitem[Wil12]{wildeshaus12}
 J.~Wildeshaus, \href{http://dx.doi.org/10.1093/imrn/rnr109}{\emph{On the
   interior motive of certain {S}himura varieties: the case of
   {H}ilbert--{B}lumenthal varieties}}, Int. Math. Res. Notices \textbf{2012}
 (2012), no.~10, 2321--2355. \MR{2923168}

 \bibitem[Wil88]{wiles88}
 A.~Wiles, \href{http://dx.doi.org/10.1007/BF01394275}{\emph{On ordinary
   {$\lambda$}-adic representations associated to modular forms}}, Invent. Math.
 \textbf{94} (1988), no.~3, 529--573. \MR{969243}

\end{thebibliography}
\end{document}